\RequirePackage{fix-cm}
\documentclass[smallextended]{svjour3}       
\smartqed  
\usepackage{graphicx}
\usepackage[utf8]{inputenc}
\usepackage[english]{babel}
\usepackage{bbm}

\usepackage[nottoc]{tocbibind}
\usepackage{mathrsfs,amsfonts,amssymb,amsmath}

\usepackage{amsthm}
\usepackage{enumerate}
\usepackage{graphicx,cite}
\usepackage{romannum}
\usepackage{hyperref}
\usepackage{authblk}
\usepackage{graphicx}
\usepackage{tikz}
\usetikzlibrary{decorations.pathreplacing, arrows.meta}

\allowdisplaybreaks

\hypersetup{colorlinks=true, linkcolor=blue, citecolor=red}

\DeclareMathOperator{\diam}{diam}

\DeclareMathOperator{\Leb}{Leb}

\DeclareMathOperator{\interior}{int}

\numberwithin{equation}{section}

\newtheorem{assumption}{Assumption}

\usepackage[square,numbers]{natbib}
\begin{document}
\pagenumbering{arabic}
\setlength{\belowdisplayskip}{0pt}

\title{Statistical properties for irregular observables in slowly mixing hyperbolic systems}


\author{Leonid A. Bunimovich \and  Yaofeng Su   
        
}


\institute{Leonid A. Bunimovich \at
              School of Mathematics, Georgia Institute of Technology,  Atlanta, USA\\
             \email{leonid.bunimovich@math.gatech.edu}
             \and
             Yaofeng Su  \at
             Southern University of Science and Technology, Shenzhen, China\\ 
              \email{suyaofeng@sustech.edu.cn}            
}

\date{Received: date / Accepted: date}

\maketitle

\begin{abstract} 
We prove various statistical properties (i.e., decay of correlations, central limit theorems with convergence rates, maximal large deviations, and almost sure invariance principles) for non-smooth observables in the form of indicator functions in polynomially mixing hyperbolic dynamical systems. Such results were not known (or even expected to hold) for slowly mixing hyperbolic systems, although they are important for applications to physics and other sciences. 

\end{abstract}

\tableofcontents

\section{Introduction}\ \par

Statistical properties are the major characteristics of  dynamical systems with chaotic behavior. The corresponding mathematical results in this area always assumed that functions (observables) on the phase space have some regularity. For instance, traditionally it was assumed that the observables satisfy a (local) H\"older condition. 
However, an analysis of a fundamental process of transport in the phase space of a system in question requires considering the evolution of characteristic (indicator) functions of measurable subsets of the phase space. Likewise, the experiments (real as well as numerical ones) deal with characteristic functions of such subsets. Indeed, in numerical simulations (computer experiments) orbits under investigation start within some subset (sometimes an element of some partition) of the phase space. In real experiments initial conditions are just slightly changed (fluctuate) because a system is desired to be initially maintained in some state (which of course fluctuates from one experiment to another because of interactions with an ``environment", that experimentalists try to minimize). Various such experiments are discussed e.g.  in the books \cite{ott, Strogatz}.

The present paper deals with  statistical properties of hyperbolic dynamical systems,  which form an important class among the so-called dynamical systems with chaotic behavior, and some of such systems appear as classical models in mechanics, statistical mechanics, quantum mechanics, optics as well as in various other applications in science and engineering \cite{laser, Prl, Prl1}. In most of these applications hyperbolic dynamical systems are slowly mixing. So, correlations there decay slowly, i.e. power-like. For this class of systems, \textbf{we derive new methods and obtain a series of new results on statistical properties (including decay of correlations, central limit theorems with convergence rates, maximal large deviations and almost sure invariance principles) when the observable belongs to a class of irregular indicator functions}.  Particularly, these results are applicable to  hyperbolic billiard systems with polynomial (as well as with exponential) mixing rates. Previously, \begin{enumerate}
    \item Such statistical properties for fast/slowly mixing hyperbolic systems were derived only for regular observables, and the methods used there essentially rely on the smoothness of observables  (see, e.g., \cite{MN1, MN2,melbourne09,alexey1,alexey,nicolldp,Y,denkerbook,KOREPANOVcoboundary,Chazottes2015, REY-BELLET_YOUNG_2008,penerate});
    \item Such statistical properties for irregular observables of bounded variation were obtained for exponential mixing expanding systems (see e.g., \cite{LIVERANI_2013,saussolmap,baladibook} and the references therein). However, such systems form  a rather  tiny class among all dynamical systems of physical interest.
\end{enumerate} 

It was pointed out that some limit theorems with irregular observables (e.g., the Birkhoff ergodic theorem) hold for some hyperbolic systems, see e.g. \cite{Sinai1995}, but their convergence rate can be arbitrarily slow \cite{ergodicrate} depending on how 
``bad" the regularity of observables is. In contrast, in the present  paper we prove that convergence rates of limit theorems considered here have uniform upper bounds for our class of irregular observables.

The organization of the paper is as follows. In section \ref{cmz} we define the main object of our study, i.e., hyperbolic systems with Chernov-Markarian-Zhang (CMZ) structures. The sections \ref{decorrelation}-\ref{nonsmoothasip} contain the main results for CMZ structures, including decay of correlations, maximal large deviations, central limit theorems with convergence rates, and almost sure invariance principles. Applications of the main theorems to hyperbolic billiards are considered in section \ref{app}.
\\\\ \text{Notation and conventions}\begin{enumerate}
 \item $C_z$ denotes a constant depending on $z$. 
    \item The notation $``a_n \precsim_z b_n"$  ($``a_n=O_{z}(b_n)"$) means that there is a constant $C_z \ge 1$ such that (s.t.) $ |a_n| \le C_z  |b_n|$ for all $n \ge 1$, whereas the notation $``a_n \precsim b_n"$ (or $``a_n=O(b_n)"$) means that there is a constant $C \ge 1$ such that $ |a_n| \le C  |b_n|$ for all $n \ge 1$. Next, $``a_n \approx_z b_n"$ and $a_n=C_z^{\pm 1}b_n$ mean that there is a constant $C_z \ge 1$ such that  $ C_z^{-1}  |b_n| \le |a_n | \le C_z |b_n|$ for all $n \ge 1$. Further, the notation $``a_n \approx b_n"$ means that there is a constant $C \ge 1$ such that $ C^{-1}  |b_n| \le |a_n | \le C |b_n|$ for all $n \ge 1$. Finally, $``a_n =o(b_n)"$ means that  $\lim_{n \to \infty}|a_n/b_n|=0$. 
    \item By $\mathbbm{1}_A$ we denote the characteristic function (indicator) of a measurable set $A$.
    \item $\mathbb{N}_0=\{0,1,2,3,\cdots\}$, $\mathbb{N}=\{1,2,3,\cdots\}$.
    \item For any $n$-dimensional (sub)manifold $M$, $\Leb_M$ denotes the $n$-dimensional  Lebesgue measure, $\mu_M$ denotes a measure defined on $M$.
\end{enumerate}

\section{Hyperbolic systems with Chernov-Markarian-Zhang (CMZ) structures}\label{cmz}

In this section we define an invertible, polynomially mixing hyperbolic dynamics $f$ on a bounded Riemannian manifold $(\mathcal{M},\Leb_{\mathcal{M}}, d)$ with a \textbf{Chernov-Markarian-Zhang (CMZ) structure}, where $d$ is a Riemannian distance and $\Leb_{\mathcal{M}}$ is the Lebesgue measure of $\mathcal{M}$. The systems with such structures are a class of hyperbolic billiards, see \cite{bbb, hongkun,markarian, billiardwithvariousrates,Y,poisson,pene}.

 Suppose that there is a fixed closed subset $X\subseteq \mathcal{M}$ with $\Leb_{\mathcal{M}}(X)>0$, called an induced subset. The first return time to $X$ is $R: X\to \mathbb{N}$. We assume that $X$ can be partitioned into countably many connected domains $X_i$, i.e.,
\begin{equation}\label{partition}
    X=\bigcup_{i\ge 1} X_i \mod{0},
\end{equation} so that $R$ is constant on each $X_i$, $\gcd \{R\}=1$ and  \[\Leb_{\mathcal{M}} (\partial X_i)=0,\quad \interior{X_i} \bigcap \interior{X_j}=\emptyset \text{ for } i \neq j.\]

 Denote by $\mathbb{S}\subseteq X$ the singularity set for $f^R$, which has zero Lebesgue measure, $\bigcup_{i\ge 1}\partial X_i\subseteq \mathbb{S}$ and $\mathbb{S}^c \subseteq X$ consists of countably many open connected components. Unstable manifolds $\gamma^u$ (resp. stable manifolds $\gamma^s$) in $X$ are the connected components of $(\bigcup_{i \ge 0}(f^R)^{i}\mathbb{S})^c$ (resp. $(\bigcup_{i \ge 0}(f^R)^{-i}\mathbb{S})^c$). From the construction, $\{f^i\gamma^u, 0\le i < R|_{\gamma^u}\}$ $(\text{resp. } \{f^i\gamma^s, 0\le i < R|_{\gamma^s}\})$
 are unstable manifolds (resp. stable manifolds) in $\mathcal{M}$.  
 A closed and connected set of the unstable (resp. stable) manifold in $\mathcal{M}$ will be called an unstable (resp. stable) disk. For convenience, we  denote each unstable (resp. stable) manifold/disk in $\mathcal{M}$ by $\gamma^u$ (resp. $\gamma^s$) throughout. The induced metrics on $\gamma^u, \gamma^s$ are $d_{\gamma^u}, d_{\gamma^s}$ respectively.

 Suppose that there is a family of disjoint stable disks $\Gamma^s$ in $X$ and a family of disjoint unstable disks $\Gamma^u$ in $X$ such that each $\gamma^u\in \Gamma^u$ intersects each $\gamma^s \in \Gamma^s$ at only one point. Also assume that there is a hyperbolic product structure $\Lambda:=\left(\bigcup_{\gamma^s \in \Gamma^s} \gamma^s\right) \bigcap \left(\bigcup_{\gamma^u\in \Gamma^u} \gamma^u\right) \subsetneq X$, a partition of $\Lambda:=\bigsqcup_{i\ge 0} \Lambda_i$ where $\Lambda_i=\left(\bigcup_{\text{some } \gamma^s\in \Gamma^s}\gamma^s\right)\bigcap \left(\bigcup_{\gamma^u\in \Gamma^u}\gamma^u\right) $, and a return (to $\Lambda$) time function $R_p: \Lambda \to \mathbb{N}$ where $R_p|_{\Lambda_i}$ is constant, such that $f^{R_p}\Lambda_i=\left(\bigcup_{\gamma^s \in \Gamma^s} \gamma^s\right) \bigcap \left(\bigcup_{\text{some }\gamma^u\in \Gamma^u} \gamma^u\right)$.

 Define a hyperbolic Young tower $\Delta$ and a map $F : \Delta \to \Delta$ by the following relations
\begin{align*}
    &\Delta:=\{(x,l) \in \Lambda \times \mathbb{N}_0: 0 \le l < R_p(x) \},  \\
    &F(x,l):=\begin{cases}
 (x,l+1),      &l < R_p(x)-1\\
\left(f^{R_p}(x),0\right),  & l=R_p(x)-1\\
\end{cases}.
\end{align*}

An equivalence relation $\sim$ on $\Delta$ is defined by
\[(x,m) \sim (y,n) \text{ if and only if } m=n, x,y \in \gamma^s \text{ for some } \gamma^s \in \Gamma^s.\]

Now one can define a quotient expanding Young tower  $\widetilde{\Delta}:=\Delta/\sim$,  quotient hyperbolic product structures $\widetilde{\Lambda}:=\Lambda/\sim, \widetilde{\Lambda}_i:=\Lambda_i/\sim$, quotient maps $\widetilde{F} : \widetilde{\Delta} \to \widetilde{\Delta}$, $\widetilde{F^{R_p}}: \widetilde{\Lambda} \to \widetilde{\Lambda}$, and canonical projections $\widetilde{\pi}_{\Delta}:\Delta \to \widetilde{\Delta}$ and $\widetilde{\pi}_{\Lambda}:\Lambda \to \widetilde{\Lambda}$ such that $\widetilde{F} \circ \widetilde{\pi}_{\Delta} =\widetilde{\pi}_{\Delta}\circ F$, $\widetilde{F^{R_p}}\circ \widetilde{\pi}_{\Lambda}=\widetilde{\pi}_{\Lambda}\circ F^{R_p}$. We introduce a family of partitions $(\mathcal{Q}_k)_{k \ge 0}$ of $\Delta$ as 
\[\mathcal{Q}_0:=\{ \Delta\}, \quad \mathcal{Q}_1:=\{\Lambda_l\times \{k\} : l \ge 0, k < R_p|_{\Lambda_l}\}, \quad \mathcal{Q}_k:= \bigvee_{0 \le i \le k-1} F^{-i} \mathcal{Q}_1,\]
another family of partitions $(\widetilde{\mathcal{Q}}_k)_{k \ge 0}$ of $\widetilde{\Delta}$ as 
\[\widetilde{\mathcal{Q}}_0:=\{\widetilde{\Delta}\}, \quad \widetilde{\mathcal{Q}}_1:= \{\widetilde{\Lambda}_l\times \{k\} : l \ge 0, k < R_p|_{\Lambda_l}\}, \quad \widetilde{\mathcal{Q}}_k:= \bigvee_{0 \le i \le k-1} (\widetilde{F})^{-i} \widetilde{\mathcal{Q}}_1.\]

For any \(x, y \in \widetilde{\Delta}\), define a separation time by
\[
s(x,y):= \min\bigl\{ n \ge 0 \;:\; \widetilde{F}^n(x) \text{ and } \widetilde{F}^n(y) \text{ lie in different elements of } \widetilde{\mathcal{Q}}_1 \bigr\}.
\]

The dynamical systems $(\mathcal{M},f)$, $(\Delta,F)$, $ (\widetilde{\Delta}, \widetilde{F})$, $(\widetilde{\Lambda},\widetilde{F^{R_p}})$ have the following properties: there are constants $\beta \in (0,1), C\ge 1, \alpha>0$ such that  \begin{enumerate}
    \item For any $ \gamma^s$ (resp. $ \gamma^u$) $\subseteq \mathcal{M}$, $x,y \in \gamma^s$ (resp. $ \gamma^u$),  $n\ge 1$, 
    \begin{align}\label{0}
        d\big(f^n(x), f^n(y)\big) \le Cn^{-\alpha}\text{ (resp. } d\big(f^{-n}(x), f^{-n}(y)\big) \le Cn^{-\alpha}).
    \end{align}

\item $(\mathcal{M}, f)$ can be lifted to a hyperbolic Young tower, namely, there is a semi-conjugacy $\pi:\Delta \to \mathcal{M}$ defined as 
$\pi(x,l):=f^l(x)$ satisfying $\pi \circ F=f\circ \pi$. 
\item Markovian: $\widetilde{F^{R_p}}\widetilde{\Lambda}_i=\widetilde{\Lambda}$ for any $\widetilde{\Lambda}_i$.

\item Bounded distortion: for any $z_1,z_2 $ in the same element of $\widetilde{\mathcal{Q}}_1$, $
\big|\log \frac{|\det D\widetilde{F}(z_1)|}{|\det D\widetilde{F}(z_2)|}\big|
\le C \, \beta^{\, s(\widetilde{F}(z_1),\, \widetilde{F}(z_2))}$, where $|\det D\widetilde{F}|$ is the Radon-Nikodym derivative with respect to the Lebesgue measure on $\gamma^u$.

\item It follows from \cite{Y1} that there exist probability measures $\mu_{\widetilde{\Lambda}}$, $\mu_{\Lambda}$ on $\widetilde{\Lambda}, \Lambda$ respectively such that $\big(F^{R_p}\big)_{*}\mu_{\Lambda}=\mu_{\Lambda}$, $\big(\widetilde{F^{R_p}}\big)_{*}\mu_{\widetilde{\Lambda}}=\mu_{\widetilde{\Lambda}}$, $\left(\widetilde{\pi}_{\Lambda}\right)_{*}\mu_{\Lambda}=\mu_{\widetilde{\Lambda}}$. If $\int R_p d\mu_{\Lambda}< \infty$, then there exist probability measures $\mu_{\widetilde{\Delta}}$, $\mu_{\Delta}$ on $\widetilde{\Delta}$, $\Delta$ respectively such that $F_{*}\mu_{\Delta}=\mu_{\Delta}$, $\widetilde{F}_{*}\mu_{\widetilde{\Delta}}=\mu_{\widetilde{\Delta}}$, $ \left(\widetilde{\pi}_{\Delta}\right)_{*}\mu_{\Delta}=\mu_{\widetilde{\Delta}}$. An SRB probability measure $\mu_{\mathcal{M}}$ on $\mathcal{M}$ is obtained by $\pi_* \mu_{\Delta}$. An SRB probability measure on $X$ is $\mu_{X}:=\mu_{\mathcal{M}}(X)^{-1}\mu_{\mathcal{M}}|_X$.
\end{enumerate}

\begin{remark}
     For many purposes the $\gcd\{R_p\}=1$  is irrelevant provided that the dynamics $f: (\mathcal{M}, \mu_\mathcal{M}) \to (\mathcal{M}, \mu_\mathcal{M}) $ is mixing (see Remark 2.2 in \cite{melboune}). Indeed all dynamical systems, which we consider in applications (Section \ref{app}), do have a countable Markov partition. And any hyperbolic ergodic dynamical systems with singularities (e.g.  dispersing billiards) in section \ref{app} only have countably infinite Markov partition (see \cite{bunimarkovsinaibilliard}). Also an ergodic completely hyperbolic (all Lyapunov exponents do not vanish) dynamical system is mixing. Therefore Young towers are mixing. So, to simplify the argument of our proof, we only assume aperiodicity (i.e., $\gcd\{R_p\}=1$) throughout.
\end{remark}

\begin{remark}\label{importantremark}
According to Proposition 3.2 of \cite{melboune}, $\int_X R^{\xi} d\mu_{\mathcal{M}}< \infty$ for some $\xi>1$ implies $\int R_p^{\xi-\epsilon} d\mu_{\Lambda}<  \infty$ for any  $\epsilon \in (0,\xi-1)$.  
\end{remark} 

We include an important result from \cite{Y1}:

\begin{lemma}[see \cite{Y1}]\label{youngmixing}
  If $\int R_p^{\xi} d\mu_{\Lambda}< \infty$ for some $\xi>1$, then there is $C>0$ such that, for any $A\in \sigma(\mathcal{Q}_{m}), \xi' \in  \sigma(\bigcup_{i\ge 0} \mathcal{Q}_{i})$ and any $p>m$, \[\big| \int \mathbbm{1}_{A} \mathbbm{1}_{\xi'}\circ F^p d\mu_{\Delta}-\int \mathbbm{1}_{A}d\mu_{\Delta} \int  \mathbbm{1}_{\xi'}d\mu_{\Delta}\big|
    \le C (p-m)^{1-\xi}\mu_{\Delta}(A),\]
    \[\big| \int \mathbbm{1}_{\widetilde{A}} \mathbbm{1}_{\widetilde{\xi'}}\circ \widetilde{F}^p d\mu_{\widetilde{\Delta}}-\int \mathbbm{1}_{\widetilde{A}}d\mu_{\widetilde{\Delta}} \int  \mathbbm{1}_{\widetilde{\xi'}}d\mu_{\widetilde{\Delta}}\big|
    \le C (p-m)^{1-\xi}\mu_{\Delta}(A),\] and for any $p>0$, 
    \[\big| \int \mathbbm{1}_{A} \mathbbm{1}_{\xi'}\circ F^p d\mu_{\Delta}-\int \mathbbm{1}_{A}d\mu_{\Delta} \int  \mathbbm{1}_{\xi'}d\mu_{\Delta}\big|
    \le C m^{\xi-1}p^{1-\xi}\mu_{\Delta}(A),\]
    \[\big| \int \mathbbm{1}_{\widetilde{A}} \mathbbm{1}_{\widetilde{\xi'}}\circ \widetilde{F}^p d\mu_{\widetilde{\Delta}}-\int \mathbbm{1}_{\widetilde{A}}d\mu_{\widetilde{\Delta}} \int  \mathbbm{1}_{\widetilde{\xi'}}d\mu_{\widetilde{\Delta}}\big|
    \le C m^{\xi-1}p^{1-\xi}\mu_{\Delta}(A),\]
\end{lemma}where $\widetilde{A}:=\widetilde{\pi}_{\Delta}A$, $\widetilde{\xi}':=\widetilde{\pi}_{\Delta}\xi'$. \\

\textbf{Observables:} Fix a collection  $\mathcal{B}$ of (dis)connected domains $B\subseteq \mathcal{M}$, which consist of finitely many  components having finitely many piecewise smooth boundaries with bounded principal curvatures, meeting along finitely many edges or corners and having finite boundary measure $\Leb_{\partial B}(\partial B)<\infty$. The observables of  hyperbolic systems, which we consider throughout the paper, are  $\mathbbm{1}_{B}$ where $B \in \mathcal{B}$. A simple example of such $B \subseteq \mathbb{R}^2$ is a finite union of polygonal regions. 

In what follows we present our main results for slowly mixing hyperbolic systems with CMZ structures and observables \(1_B\) with \(B\in\mathcal B\).

\section{Decay of correlations}\label{decorrelation}

The decay of correlation (also referred to as the mixing rate) is a fundamental statistical property of chaotic dynamical systems. It quantifies the speed at which such a system loses memory of its initial conditions (see \cite{baladibook} for a comprehensive overview) and plays a crucial role in establishing various limit theorems. When the observables are regular, the decay of correlation for hyperbolic dynamical systems $(\mathcal{M},f, \mu_{\mathcal{M}})$ with CMZ structures  was proved in \cite{Y,hongkun,markarian,billiardwithvariousrates}. Now we prove such results for our class of irregular observables.

\begin{theorem}\label{mixingrate}
Suppose that a hyperbolic dynamical system $(\mathcal{M},f, \mu_{\mathcal{M}})$ with a CMZ structure as described in section \ref{cmz} satisfies additional assumptions $\frac{d\mu_{\mathcal{M}}}{d\Leb_{\mathcal{M}}}\in L^{\infty}$ and $\int R_p^{\xi} d\mu_{\Lambda} < \infty$ for some $\xi>1$. Then for any $B,B'\in \mathcal{B}$, there is a constant $p_{B,B'}>1$ such that for any $p \ge p_{B,B'}$, 
\begin{align*}
    \Big|\int&\mathbbm{1}_{B} \mathbbm{1}_{B'}\circ f^p d\mu_{\mathcal{M}}-\int\mathbbm{1}_{B}d\mu_{\mathcal{M}} \int \mathbbm{1}_{B'}d\mu_{\mathcal{M}}\Big|\\
    &\precsim  p^{1-\xi}\Leb_{\mathcal{M}}(B)+p^{-\alpha}\Leb_{\partial B}(\partial B)+ p^{-\alpha}\Leb_{\partial B'}(\partial B')
\end{align*} 
where the constant in $\precsim$ does not depend on $p$, $\Leb_{\partial B}(\partial B), \Leb_{\mathcal{M}}(B)$, $\Leb_{\mathcal{M}}(B')$ or $\Leb_{\partial B'}(\partial B')$. 

Alternatively, if the polynomial decay estimates in (\ref{0}) and Lemma \ref{youngmixing} become exponential decay estimates $e^{-\epsilon n}$ and $e^{-\epsilon_1 (p-m)}$, respectively, for some $\epsilon, \epsilon_1>0$, then
\begin{align*}
    \Big|\int&\mathbbm{1}_{B} \mathbbm{1}_{B'}\circ f^p d\mu_{\mathcal{M}}-\int\mathbbm{1}_{B}d\mu_{\mathcal{M}} \int \mathbbm{1}_{B'}d\mu_{\mathcal{M}}\Big|\\
    &\precsim  e^{-\epsilon_1 p/2}\Leb_{\mathcal{M}}(B)+e^{-\epsilon p/4}\Leb_{\partial B}(\partial B)+ e^{-\epsilon p/4}\Leb_{\partial B'}(\partial B')
\end{align*}where the constant in $\precsim$ does not depend on $p$, $\Leb_{\partial B}(\partial B), \Leb_{\mathcal{M}}(B)$, $\Leb_{\mathcal{M}}(B')$ or $\Leb_{\partial B'}(\partial B')$. 
\end{theorem}
\begin{proof}
 Similar to the approach developed in \cite{poisson, bbb, mldp, pene}, we will make use of the partitions of $\Delta$. By $\mu_{\mathcal{M}}=\pi_*\mu_{\Delta}$ and the invariance of $F$ (i.e., $F_{*}\mu_{\Delta}=\mu_{\Delta}$) we have
\begin{align*}
\int \mathbbm{1}_{B} \mathbbm{1}_{B'} \circ f^p  d\mu_{\mathcal{M}} =\int \mathbbm{1}_{B} \circ \pi \circ F^{m} \mathbbm{1}_{B'} \circ \pi \circ  F^{m}\circ F^p d\mu_{\Delta}. 
\end{align*}

Let 
 $A_1:=F^{-m} \pi^{-1} B$, $A_0:=\bigsqcup_{Q \in \mathcal{Q}_{2m}:Q \bigcap A_1\neq \emptyset}Q$ and $A_2:=\bigsqcup_{Q \in \mathcal{Q}_{2m}:Q \bigcap (A_0\setminus A_1)\neq \emptyset}Q$, $A'_1:=F^{-m} \pi^{-1} B'$, $A_0':=\bigsqcup_{Q \in \mathcal{Q}_{2m}:Q \bigcap A_1'\neq \emptyset}Q$ and $A_2':=\bigsqcup_{Q \in \mathcal{Q}_{2m}:Q \bigcap (A_0'\setminus A_1')\neq \emptyset}Q$. Then $A_1 \bigcup A_2=A_0$, $A_1' \bigcup A_2'=A_0'$. The sets $A_0$, $A_2$, $A_0'$, $A_2'$ are all $\sigma(\bigcup_{k\ge 0}\mathcal{Q}_k)$-measurable. Therefore we have
 \begin{align*}
     \int \mathbbm{1}_{B} \mathbbm{1}_{B'} \circ f^p  d\mu_{\mathcal{M}} =\int \mathbbm{1}_{A_1}  \mathbbm{1}_{A_1'}\circ F^p d\mu_{\Delta}=\int \mathbbm{1}_{A_0}  \mathbbm{1}_{A_0'} \circ F^p d\mu_{\Delta}+O(\mu_{\Delta}(A_2))+O(\mu_{\Delta}(A_2')),
 \end{align*} 
 \[\int\mathbbm{1}_{B}d\mu_{\mathcal{M}} \int \mathbbm{1}_{B'}d\mu_{\mathcal{M}}=\int \mathbbm{1}_{A_0} d\mu_{\Delta} \int \mathbbm{1}_{A_0'}  d\mu_{\Delta}+O(\mu_{\Delta}(A_2))+O(\mu_{\Delta}(A_2')).\] 
 
 Hence, by Lemma \ref{youngmixing}, when $p> 2m$, we have
\begin{align*}
   &\Big|\int\mathbbm{1}_{B} \mathbbm{1}_{B'}\circ f^p d\mu_{\mathcal{M}}-\int\mathbbm{1}_{B}d\mu_{\mathcal{M}} \int \mathbbm{1}_{B'}d\mu_{\mathcal{M}}\Big|\\
    &\le \left|\int \mathbbm{1}_{A_0}  \mathbbm{1}_{A_0'} \circ F^p d\mu_{\Delta}-\mu_{\Delta}(A_0) \mu_{\Delta}(A_0')\right|+O(\mu_{\Delta}(A_2))+O(\mu_{\Delta}(A_2'))\\
   &\precsim (p-2m)^{1-\xi}\mu_{\Delta}(A_0)+O(\mu_{\Delta}(A_2))+O(\mu_{\Delta}(A_2')).
\end{align*}

Before proceeding with the estimate, we need the following estimates for $B,B'$, which can also be found in \cite{pene,poisson}. 

\begin{lemma}\label{boundarymeasure}
    Let $\mathcal{N}_r(A)$ be an $r$-neighborhood of a set $A$. Then for any sufficiently large $m$, $\mu_{\Delta}(A_2)\le \mu_{\mathcal{M}}\left(\mathcal{N}_{Cm^{-\alpha}}(\partial B)\right)\precsim m^{-\alpha} \Leb_{\partial B}(\partial B), \mu_{\Delta}(A_0)\le \mu_{\mathcal{M}}\left(\mathcal{N}_{Cm^{-\alpha}}(B)\right)$, $\mu_{\Delta}(A_2')\le \mu_{\mathcal{M}}\left(\mathcal{N}_{Cm^{-\alpha}}(\partial B')\right)\precsim m^{-\alpha} \Leb_{\partial B'}(\partial B')$, $\mu_{\Delta}(A_0')\le \mu_{\mathcal{M}}\left(\mathcal{N}_{Cm^{-\alpha}}(B')\right)$ where the constants in ``$\precsim$" depend on the principal curvatures of $\partial B,\partial B'$ and $C$ but do not depend on $m, \Leb_{\partial B}(\partial B), \Leb_{\partial B'}(\partial B')$. 
\end{lemma}
 
\begin{proof}
    Observe that $\mu_{\Delta}(A_2)\le \mu_{\Delta}\left(F^{-m}\pi^{-1}\pi F^m A_2\right)=\mu_{\mathcal{M}}\left(\pi F^m A_2\right)$. Since $A_2:=\bigsqcup_{Q \in \mathcal{Q}_{2m}:Q \bigcap (A_0\setminus A_1)\neq \emptyset}Q$, for each $Q$ contained in $A_2$, there exist $x_1, x_2 \in Q $, such that $\pi(F^mx_1)\in B, \pi(F^mx_2)\notin B$, and (\ref{0}) implies $\diam \pi(F^m Q) \le C m^{-\alpha}$. So $\pi(F^m Q) \subseteq \mathcal{N}_{Cm^{-\alpha}}(\partial B)$. Hence $\pi(F^m A_2) \subseteq \mathcal{N}_{Cm^{-\alpha}}(\partial B)$. The same arguments can be applied  to $A_2', A_0,A_0'$.
\end{proof} 

We continue the estimate by taking $m=\lfloor p/4 \rfloor$ for sufficiently large $p$ and applying Lemma \ref{boundarymeasure}.  
\begin{align*}
    & \precsim p^{1-\xi}[\mu_{\mathcal{M}}(B)+\mu_{\mathcal{M}}(\mathcal{N}_{Cm^{-\alpha}}(\partial B))]+O(\mu_{\mathcal{M}}(\mathcal{N}_{Cm^{-\alpha}}(\partial B)))+O(\mu_{\mathcal{M}}(\mathcal{N}_{Cm^{-\alpha}}(\partial B')))\\
    & \precsim p^{1-\xi}[\Leb_{\mathcal{M}}(B)+p^{-\alpha}\Leb_{\partial B}(\partial B)]+p^{-\alpha}\Leb_{\partial B}(\partial B)+ p^{-\alpha}\Leb_{\partial B'}(\partial B')\\
    &\precsim  p^{1-\xi}\Leb_{\mathcal{M}}(B)+p^{-\alpha}\Leb_{\partial B}(\partial B)+ p^{-\alpha}\Leb_{\partial B'}(\partial B').
    \end{align*} 
    
    Also note that $\big|\int \mathbbm{1}_{B} \mathbbm{1}_{B'}\circ f^p d\mu_{\mathcal{M}}-\int\mathbbm{1}_{B}d\mu_{\mathcal{M}} \int \mathbbm{1}_{B'}d\mu_{\mathcal{M}}\big|=O(\Leb_{\mathcal{M}}(B))$ for any $p>0$. So we conclude the proof of the theorem in the polynomial mixing case for all $p\ge 1$. 
    
    For the exponential mixing case, the argument is the same, but one replaces $m^{-\alpha},p^{-\alpha},p^{1-\xi}, (p-2m)^{1-\xi}$ with $e^{-\epsilon m}, e^{-\epsilon p/4}, e^{-\epsilon_1 p/2}, e^{-\epsilon_1 (p-2m)}$ respectively.
\end{proof}

\begin{remark}
    It is well-known (see e.g. \cite{nonholder}) that even for the uniformly hyperbolic systems with the strongest random (``chaotic") properties, there exist observables for which the decay of correlation is non-exponential. Our Theorem \ref{mixingrate} shows that the correlation for hyperbolic systems still decays exponentially for our class of indicator observables.
\end{remark}

\section{Maximal large deviations}
Large deviations and maximal large deviations have many important applications to various core problems in chaotic dynamical systems (e.g., see \cite{mldp} for more details). When the observables are regular, \cite{melbourne09,nicolldp} proved a polynomial large deviation principle and \cite{mldp} proved a maximal large deviation principle for hyperbolic dynamical systems $(\mathcal{M},f, \mu_{\mathcal{M}})$ with CMZ structures. Now we prove such results for our class of irregular observables. 
\begin{assumption}\label{A} 
There exists a sequence $D_\ell \in \mathcal{B}$ such that $
\lim_{\ell \to \infty}\Leb_{\mathcal{M}}(D_\ell)= 0$,
and, for every fixed $\ell$, there exist constants $c_\ell > 0$ and $n_\ell \ge 1$ such that
\begin{gather*}
\mu_{\mathcal{M}}\left(\bigcap_{j=0}^{n-1} f^{j}D_\ell\right)= \mu_{\mathcal{M}}\left(\bigcap_{j=0}^{n-1} f^{-j}D_\ell\right)
\ge c_\ell n^{1-\xi} \text{ for any } n \ge n_\ell,\\
C \cup D_{\ell}, C \setminus D_{\ell} \in \mathcal{B} \text{ for any }C \in \mathcal{B}.
\end{gather*} 
\end{assumption}

Now we can state the main theorem.
\begin{theorem}[Maximal large deviations]\label{mldp} \par
Suppose that a hyperbolic dynamical system $(\mathcal{M},f, \mu_{\mathcal{M}})$ with a CMZ structure described in section \ref{cmz} satisfies additional assumptions $\frac{d\mu_{\mathcal{M}}}{d\Leb_{\mathcal{M}}}\in L^{\infty}$ and $\mu_{\mathcal{M}}|_X(R>n)\approx n^{-\xi}$ for some $\xi>1$. Then for any  $B \in \mathcal{B}$ with $\Leb_{\mathcal{M}}(B)\in (0, \Leb_{\mathcal{M}}(\mathcal{M}))$, 
\begin{enumerate}
    \item the maximal large deviations for the observable $\mathbbm{1}_{B}$ hold, that is, for any $\epsilon>0$, \[\liminf_{N\to \infty }\frac{
\log \mu_{\mathcal{M}}\Big(\sup_{n\ge N}\Big| \frac{\sum_{0\le i\le n-1}\mathbbm{1}_B\circ f^i}{n}-\mu_{\mathcal{M}}(B)\Big|\ge \epsilon \Big)}{-\log N}\ge \xi-1.\] 

\item Endow $\mathcal{A}:=\{B \in  \mathcal{B}:\Leb_{\mathcal{M}}(B)\in (0, \Leb_{\mathcal{M}}(\mathcal{M}))\}$ with the $L^1$-metric, that is, the $L^1$-distance of two sets $B_1,B_2 \in \mathcal{A}$ is $|\mathbbm{1}_{B_1}-\mathbbm{1}_{B_2}|_1$. If Assumption \ref{A} holds, then there is a dense set of $B\in \mathcal{A}$ such that there is $\epsilon_0>0$, for any $\epsilon \in (0,\epsilon_0)$, 
\begin{gather}\label{1}
    \liminf_{n \to \infty}\frac{\log \mu_{\mathcal{M}}\Big(\Big| \frac{\sum_{0\le i\le n-1}\mathbbm{1}_B\circ f^i}{n}-\mu_{\mathcal{M}}(B)\Big|\ge \epsilon \Big)}{-\log n}\le  \xi-1.
\end{gather}

Therefore, there is a dense set of $B\in \mathcal{A}$ such that there is $\epsilon_0>0$, for any $\epsilon \in (0,\epsilon_0)$, 
\[\liminf_{N \to \infty}\frac{\log \mu_{\mathcal{M}}\Big(\sup_{n\ge N}\Big| \frac{\sum_{0\le i\le n-1}\mathbbm{1}_B\circ f^i}{n}-\mu_{\mathcal{M}}(B)\Big|\ge \epsilon \Big)}{-\log N}=\xi-1.\]
\end{enumerate}
\end{theorem}
\begin{remark}
Unlike the $L^{\infty}$-dense set in \cite{melbourne09, nicolldp}, this $L^1$-dense set is not open.
\end{remark}

\begin{proof}[Proof of Theorem \ref{mldp} (lower bounds)]  The proof will make use of the Young tower. Let $k\ge 1$ (to be determined), lift the deviation to $\Delta$ and use $(F^k)_*\mu_{\Delta}=\mu_{\Delta}$, that is, 
    \begin{align*}
        &\mu_{\mathcal{M}}\Big(\sup_{n\ge N}\Big| \frac{\sum_{0\le i\le n-1}\mathbbm{1}_B\circ f^i}{n}-\mu_{\mathcal{M}}(B)\Big|\ge \epsilon \Big)\\
        &=\mu_{\Delta}\Big(\sup_{n\ge N}\Big| \frac{\sum_{0\le i\le n-1}\mathbbm{1}_{\pi^{-1}B}\circ F^i}{n}-\mu_{\Delta}(\pi^{-1}B)\Big|\ge \epsilon \Big)\\
        &=\mu_{\Delta}\Big(\sup_{n\ge N}\Big| \frac{\sum_{0\le i\le n-1}\mathbbm{1}_{F^{-k}\pi^{-1}B}\circ F^i}{n}-\mu_{\Delta}(F^{-k}\pi^{-1}B)\Big|\ge \epsilon \Big).
    \end{align*}

    Cover $F^{-k}\pi^{-1}B$ with a family of  $\xi_{2k,i} \in \mathcal{Q}_{2k}$. Collect such $\xi_{2k,i}$ which is completely contained in $F^{-k}\pi^{-1}B$, denote it by $\mathring{\xi}_{2k,i}$. Collect the $\xi_{2k,i}$ which is not completely contained in $F^{-k}\pi^{-1}B$, denote it by $\partial \xi_{2k,i}$ and let $A_2:=\bigcup_i \partial \xi_{2k,i}$. Then $\bigcup_i \mathring{\xi}_{2k,i} \subseteq  F^{-k}\pi^{-1}B \subseteq \bigcup_i \xi_{2k,i}$ and $\mu_{\Delta}(A_2)\precsim k^{-\alpha} \Leb_{\partial B}(\partial B)$ according to Lemma \ref{boundarymeasure}. Also note that \[\frac{\sum_{0\le i\le n-1}\mathbbm{1}_{\bigcup_i \mathring{\xi}_{2k,i}}\circ F^i}{n}\le \frac{\sum_{0\le i\le n-1}\mathbbm{1}_{F^{-k}\pi^{-1}B}\circ F^i}{n} \le \frac{\sum_{0\le i\le n-1}\mathbbm{1}_{\bigcup_i \xi_{2k,i}}\circ F^i}{n},\]
    then when $k \gg 1$, $\epsilon -\mu_{\Delta}(A_2)>\epsilon/2$ and \begin{align}
        &\mu_{\Delta}\Big(\sup_{n\ge N}\Big| \frac{\sum_{0\le i\le n-1}\mathbbm{1}_{F^{-k}\pi^{-1}B}\circ F^i}{n}-\mu_{\Delta}(F^{-k}\pi^{-1}B)\Big|\ge \epsilon \Big) \nonumber\\
        &\le \mu_{\Delta}\Big(\sup_{n\ge N}\Big| \frac{\sum_{0\le i\le n-1}\mathbbm{1}_{\bigcup_i \mathring{\xi}_{2k,i}}\circ F^i}{n}-\mu_{\Delta}(F^{-k}\pi^{-1}B)\Big|\ge \epsilon \Big) \nonumber\\
        &\quad +\mu_{\Delta}\Big(\sup_{n\ge N}\Big| \frac{\sum_{0\le i\le n-1}\mathbbm{1}_{\bigcup_i \xi_{2k,i}}\circ F^i}{n}-\mu_{\Delta}(F^{-k}\pi^{-1}B)\Big|\ge \epsilon \Big) \nonumber\\
        &\le \mu_{\Delta}\Big(\sup_{n\ge N}\Big| \frac{\sum_{0\le i\le n-1}\mathbbm{1}_{\bigcup_i \mathring{\xi}_{2k,i}}\circ F^i}{n}-\mu_{\Delta}(\bigcup_i \mathring{\xi}_{2k,i})\Big|\ge \epsilon -\mu_{\Delta}(A_2)\Big) \nonumber\\
        &\quad +\mu_{\Delta}\Big(\sup_{n\ge N}\Big| \frac{\sum_{0\le i\le n-1}\mathbbm{1}_{\bigcup_i \xi_{2k,i}}\circ F^i}{n}-\mu_{\Delta}(\bigcup_i \xi_{2k,i})\Big|\ge \epsilon -\mu_{\Delta}(A_2)\Big)\nonumber\\
        &\le \mu_{\Delta}\Big(\sup_{n\ge N}\Big| \frac{\sum_{0\le i\le n-1}\mathbbm{1}_{\bigcup_i \mathring{\xi}_{2k,i}}\circ F^i}{n}-\mu_{\Delta}(\bigcup_i \mathring{\xi}_{2k,i})\Big|\ge \epsilon/2\Big) \nonumber\\
        &\quad +\mu_{\Delta}\Big(\sup_{n\ge N}\Big| \frac{\sum_{0\le i\le n-1}\mathbbm{1}_{\bigcup_i \xi_{2k,i}}\circ F^i}{n}-\mu_{\Delta}(\bigcup_i \xi_{2k,i})\Big|\ge \epsilon/2\Big) \label{2}.
    \end{align}

    To proceed with the proof, we will verify the conditions of Theorem 1 in \cite{mldp}, that is, if we denote the transfer operator of $\widetilde{F}: \widetilde{\Delta}\to \widetilde{\Delta}$ by $P$ and apply Lemma \ref{youngmixing} and Remark \ref{importantremark} to the observables $\mathbbm{1}_{\bigcup_i \widetilde{\xi_{2k,i}}}-\mu_{\widetilde{\Delta}}(\bigcup_i \widetilde{\xi_{2k,i}})$ and $\mathbbm{1}_{\bigcup_i \widetilde{\mathring{\xi}_{2k,i}}}-\mu_{\widetilde{\Delta}}(\bigcup_i \widetilde{\mathring{\xi}_{2k,i}})$, then for any small $\eta>0$, any $n> 4k$,
    \[\int \Big|P^n\Big(\mathbbm{1}_{\bigcup_i \widetilde{\mathring{\xi}_{2k,i}}}-\mu_{\widetilde{\Delta}}(\bigcup_i \widetilde{\mathring{\xi}_{2k,i}})\Big)\Big|d\mu_{\widetilde{\Delta}}\precsim_{\eta} (n-2k)^{-\xi+1+\eta}\precsim_{\eta,\xi} n^{-\xi+1+\eta},\]
     \[\int \Big|P^n\Big(\mathbbm{1}_{\bigcup_i \widetilde{\xi_{2k,i}}}-\mu_{\widetilde{\Delta}}(\bigcup_i \widetilde{\xi_{2k,i}})\Big)\Big|d\mu_{\widetilde{\Delta}}\precsim_{\eta} (n-2k)^{-\xi+1+\eta}\precsim_{\eta,\xi} n^{-\xi+1+\eta}.\]

     On the other hand, $\int |P^n(\mathbbm{1}_{\bigcup_i \widetilde{\xi_{2k,i}}}-\mu_{\widetilde{\Delta}}(\bigcup_i \widetilde{\xi_{2k,i}}))|d\mu_{\widetilde{\Delta}}=O(1), \int |P^n(\mathbbm{1}_{\bigcup_i \widetilde{\mathring{\xi}_{2k,i}}}-\mu_{\widetilde{\Delta}}(\bigcup_i \widetilde{\mathring{\xi}_{2k,i}}))|d\mu_{\widetilde{\Delta}}=O(1)$ for any $n\le 4k$. Hence, for any $n>0$,
    \[\int \Big|P^n\Big(\mathbbm{1}_{\bigcup_i \widetilde{\mathring{\xi}_{2k,i}}}-\mu_{\widetilde{\Delta}}(\bigcup_i \widetilde{\mathring{\xi}_{2k,i}})\Big)\Big|d\mu_{\widetilde{\Delta}}\precsim_{\eta,\xi} k^{\xi-1-\eta}n^{-\xi+1+\eta},\]
     \[\int \Big|P^n\Big(\mathbbm{1}_{\bigcup_i \widetilde{\xi_{2k,i}}}-\mu_{\widetilde{\Delta}}(\bigcup_i \widetilde{\xi_{2k,i}})\Big)\Big|d\mu_{\widetilde{\Delta}}\precsim_{\eta,\xi} k^{\xi-1-\eta}n^{-\xi+1+\eta}.\]

     So the conditions of Theorem 1 in \cite{mldp} are verified, and we can continue the estimate of (\ref{2}).
     \begin{align*}
         & \le \mu_{\Delta}\Big(\sup_{n\ge N}\Big| \frac{\sum_{0\le i\le n-1}\mathbbm{1}_{\bigcup_i \mathring{\xi}_{2k,i}}\circ F^i}{n}-\mu_{\Delta}(\bigcup_i \mathring{\xi}_{2k,i})\Big|\ge \epsilon/2\Big) \\
         &\quad +\mu_{\Delta}\Big(\sup_{n\ge N}\Big| \frac{\sum_{0\le i\le n-1}\mathbbm{1}_{\bigcup_i \xi_{2k,i}}\circ F^i}{n}-\mu_{\Delta}(\bigcup_i \xi_{2k,i})\Big|\ge \epsilon/2\Big)\\
         & = \mu_{\widetilde{\Delta}}\Big(\sup_{n\ge N}\Big| \frac{\sum_{0\le i\le n-1}\mathbbm{1}_{\bigcup_i \widetilde{\mathring{\xi}_{2k,i}}}\circ \widetilde{F}^i}{n}-\mu_{\widetilde{\Delta}}(\bigcup_i \widetilde{\mathring{\xi}_{2k,i}})\Big|\ge \epsilon/2\Big) \\
         &\quad +\mu_{\widetilde{\Delta}}\Big(\sup_{n\ge N}\Big| \frac{\sum_{0\le i\le n-1}\mathbbm{1}_{\bigcup_i \widetilde{\xi_{2k,i}}}\circ \widetilde{F}^i}{n}-\mu_{\widetilde{\Delta}}(\bigcup_i \widetilde{\xi_{2k,i}})\Big|\ge \epsilon/2\Big)\precsim k^{\xi-1-\eta}N^{-\xi+1+\eta}.
     \end{align*}

     Choose $k=\lfloor N^{\eta'}\rfloor$ with small enough $\eta'>0$, then \begin{align*}
         \mu_{\Delta}\Big(\sup_{n\ge N}\Big| \frac{\sum_{0\le i\le n-1}\mathbbm{1}_{F^{-k}\pi^{-1}B}\circ F^i}{n}-\mu_{\Delta}(F^{-k}\pi^{-1}B)\Big|\ge \epsilon \Big) \precsim N^{(-\xi+1+\eta)(1-\eta')}.
     \end{align*}

     In other words, \begin{align*}
         \liminf_{N\to \infty }\frac{
\log \mu_{\mathcal{M}}\Big(\sup_{n\ge N}\Big| \frac{\sum_{0\le i\le n-1}\mathbbm{1}_B\circ f^i}{n}-\mu_{\mathcal{M}}(B)\Big|\ge \epsilon \Big)}{-\log N}\ge (\xi-1-\eta)(1-\eta')
     \end{align*}holds for any small $\eta,\eta'>0$. Hence we conclude the proof for lower bounds by letting $\eta,\eta' \to 0$.
\end{proof}

\begin{proof}[Proof of Theorem \ref{mldp} (upper bounds)]
If (\ref{1}) fails for $B \in \mathcal{A}$, choose $\ell$ so large that $\mu_{\mathcal{M}}(D_\ell) < 1/4$ and $D_\ell$ is arbitrarily small in $L^1$. Define
\[
C =
\begin{cases}
B \cup D_\ell, & \mu_{\mathcal{M}}(B) \le \frac{1}{2}, \\[2mm]
B \setminus D_\ell, & \mu_{\mathcal{M}}(B) > \frac{1}{2}.
\end{cases}
\]
and $C \in \mathcal{A}$. If $\mu_{\mathcal{M}}(B) \le 1/2$, then on $\bigcap_{j=0}^{n-1} f^{-j}D_\ell$, 
$\frac{1}{n} \sum_{j=0}^{n-1} \mathbbm{1}_C \circ f^j = 1$,
and $
1 - \mu_{\mathcal{M}}(C)
\ge 1 - \mu_{\mathcal{M}}(B) - \mu_{\mathcal{M}}(D_\ell)
> 1/4$. If $\mu_{\mathcal{M}}(B) > 1/2$, then on the same set, $\frac{1}{n} \sum_{j=0}^{n-1} \mathbbm{1}_C \circ f^j = 0$, and $\mu_{\mathcal{M}}(C)
\ge \mu_{\mathcal{M}}(B) - \mu_{\mathcal{M}}(D_\ell)
> 1/4$. Therefore, for every $0 < \epsilon < 1/4$ and every
$n\ge n_\ell$,
\[
\mu_{\mathcal{M}}\left(
\left|
\frac{1}{n}\sum_{j=0}^{n-1} \mathbbm{1}_C \circ f^j - \mu_{\mathcal{M}}(C)
\right| \ge \epsilon
\right)
\ge c_\ell n^{1-\xi}.
\]

Consequently,
\[
\limsup_{n\to\infty}
\frac{
\log \mu_{\mathcal{M}}\left(
\left|
\frac{1}{n}\sum_{j=0}^{n-1} \mathbbm{1}_C \circ f^j - \mu_\mathcal{M}(C)
\right| \ge \epsilon
\right)}
{-\log n}
\le \xi - 1
\]
Moreover, $\|\mathbbm{1}_C - \mathbbm{1}_B\|_{1}
\precsim \Leb_{\mathcal{M}}(D_\ell) \to 0$.
Thus these $C$ form a dense subset of $\mathcal{A}$ and satisfy (\ref{1}).
\end{proof}

\section{Central limit theorems (CLT) with convergence rates}
The CLT is one of the central statistical properties of chaotic dynamical systems. When the observables are regular, \cite{penerate, bunimovich2, Y} proved the CLT (with convergence rates) for Sinai billiards (a special case of hyperbolic systems with CMZ structures). Now we prove such a CLT for our class of irregular observables. 
\begin{theorem}[CLT with convergence rates for indicator observables]\label{clt}\ \par
Suppose that a hyperbolic dynamical system $(\mathcal{M},f, \mu_{\mathcal{M}})$ with a CMZ structure described in section \ref{cmz} satisfies additional assumptions $\alpha>1$, $\frac{d\mu_{\mathcal{M}}}{d\Leb_{\mathcal{M}}}\in L^{\infty}$ and $\int R_p^{\xi} d\mu_{\Lambda}< \infty$ for some $\xi>2$. Given any $B\in \mathcal{B}$  with $\Leb_{\mathcal{M}}(B)\in (0, \Leb_{\mathcal{M}}(\mathcal{M}))$,  \begin{enumerate}
    \item If $\sigma^2:=\mu_{\mathcal{M}}(B)-\mu_{\mathcal{M}}(B)^2+2\sum_{i \ge 1}\int [\mathbbm{1}_B-\mu_{\mathcal{M}}(B)]\circ f^i[\mathbbm{1}_B-\mu_{\mathcal{M}}(B)]d\mu_{\mathcal{M}}=0$, then there is an $L^1$-function $\phi$ such that $\mathbbm{1}_B-\mu_{\mathcal{M}}(B)=\phi \circ f -\phi$.

\item If $\sigma^2\neq 0$, then  $\sigma^2\in (0,\infty)$ and there is a constant $C>0$ such that  for any $t\in \mathbb{R},n>0$\begin{align*}
   &\Big|\int \exp\Big(it\frac{\sum_{0\le i\le n-1}[\mathbbm{1}_B\circ f^i-\mu_{\mathcal{M}}(B)]}{\sqrt{n}}\Big)d\mu_{\mathcal{M}}- e^{-\frac{1}{2}\sigma^2 t^2}\Big|\le C (\vert{}t\vert{}+ t^4)n^{-\frac{\min\{\xi-2, \alpha-1\} \min\{\xi-2, 1\} (\xi-2)^2}{32(\xi-1)^2(3\xi-4)}} ,
   \end{align*}
   \begin{align*}
   &\sup_{x\in \mathbb{R}}\Big|\mu_{\mathcal{M}}\Big(\frac{\sum_{0\le i\le n-1}[\mathbbm{1}_B\circ f^i-\mu_{\mathcal{M}}(B)]}{\sqrt{n}}<x\Big)- \frac{1}{\sqrt{2\pi\sigma^2}}\int_{-\infty}^x e^{-\frac{1}{2\sigma^2 } t^2}dt \Big|\le C n^{-\frac{\min\{\xi-2, \alpha-1\} \min\{\xi-2, 1\} (\xi-2)^2}{192(\xi-1)^2(3\xi-4)}} .
   \end{align*}
   \end{enumerate}
\end{theorem}

\begin{remark}
The assumption $\xi>2$ is optimal, because the billiards of interest (which satisfy $\xi\le 2$) do not satisfy the CLT for some special indicator functions (whose supports are elements of a partition), see \cite{gouezelbalint, cuspsmorethanone, hongkunstable, melbournelevy2}. We will prove different limit theorems for these types of billiards for indicator observables in a different paper. 
\end{remark}
\begin{remark}
     Since $\mathbbm{1}_{B}$ is not a dynamically regular function, the convergence rates for the CLT are not as good as those in the standard CLT context.
\end{remark}

\subsection{Scheme of the proof}
    \begin{enumerate}
        \item \label{excludecobound} Prove $\sigma^2\in [0, \infty)$ and that $\mathbbm{1}_B-\mu_{\mathcal{M}}(B)$ is a coboundary when $\sigma^2=0$. 
        \item \label{approximatedbyregularset} When $\sigma^2\in (0,\infty)$, note that $\mathbbm{1}_B\circ f^i-\mu_{\mathcal{M}}(B)=_d  \mathbbm{1}_{F^{-k}\pi^{-1}B}\circ F^i-\mu_{\Delta}(F^{-k}\pi^{-1}B)$ ($k$ to be determined), approximate $\mathbbm{1}_{F^{-k}\pi^{-1}B}$ by $\mathcal{Q}_{2k}$-measurable function $\mathbbm{1}_{A}$ and estimate
        \begin{align*}
        &\Big|\int \exp\Big(it\frac{\sum_{0\le i\le n-1}[\mathbbm{1}_B\circ f^i-\mu_{\mathcal{M}}(B)]}{\sqrt{n}}\Big)d\mu_{\mathcal{M}}-\int \exp\Big(it\frac{\sum_{0\le i\le n-1}[\mathbbm{1}_{A}\circ F^i-\int \mathbbm{1}_{A} d\mu_{\Delta}]}{\sqrt{n}}\Big)d\mu_{\Delta}\Big|
        \end{align*}
        \item \label{severalvariances} Prove the existence of $\sigma^2_k:=\lim_{n \to \infty }\int \Big(\frac{\sum_{0\le i\le n-1}[\mathbbm{1}_{A}\circ F^i-\int \mathbbm{1}_{A} d\mu_{\Delta}]}{\sqrt{n}}\Big)^2d\mu_{\Delta}$,  estimate $|\sigma^2_k-\sigma^2|$  and $\Big|\int \Big(\frac{\sum_{0\le i\le n-1}[\mathbbm{1}_{A}\circ F^i-\int \mathbbm{1}_{A} d\mu_{\Delta}]}{\sqrt{n}}\Big)^2d\mu_{\Delta}-\sigma_k^2\Big|$.
        \item \label{verifysutams} Verify the condition of Remark 2.8 in \cite{Sutams} and derive the moment inequalities in \cite{Sutams} for $\mathbbm{1}_{A}-\int \mathbbm{1}_{A} d\mu_{\Delta}$.
        \item \label{estimatepreclt} Estimate $\Big|\int \exp\Big(it\frac{\sum_{0\le i\le n-1}[\mathbbm{1}_{A}\circ F^i-\int \mathbbm{1}_{A} d\mu_{\Delta}]}{\sqrt{n}}\Big)d\mu_{\Delta}-\exp\Big(-  \frac{t^2}{2}\int \Big(\frac{\sum_{0\le i\le n-1}[\mathbbm{1}_{A}\circ F^i-\int \mathbbm{1}_{A} d\mu_{\Delta}]}{\sqrt{n}}\Big)^2d\mu_{\Delta}\Big)\Big|$ 
        \item \label{provingclt} Choose $k=n^{\eta}$ for a suitable $\eta>0$, combine all estimates in previous steps and estimate $\Big|\int \exp\Big(it\frac{\sum_{0\le i\le n-1}[\mathbbm{1}_B\circ f^i-\mu_{\mathcal{M}}(B)]}{\sqrt{n}}\Big)d\mu_{\mathcal{M}}-e^{-0.5 \sigma^2 t^2}\Big|$.
         \item \label{provefurtherclt} Estimate $\sup_{x\in \mathbb{R}}|\mu_{\mathcal{M}}(\frac{\sum_{0\le i\le n-1}[\mathbbm{1}_B\circ f^i-\mu_{\mathcal{M}}(B)]}{\sqrt{n}}<x)- \frac{1}{\sqrt{2\pi\sigma^2}}\int_{-\infty}^x e^{-\frac{1}{2\sigma^2 } t^2}dt |$.
    \end{enumerate}

\subsection{Proof of Theorem \ref{clt}}

\subsection*{Step \ref{excludecobound}:  $\sigma^2=0$ and $\sigma^2\in (0,\infty)$.}
\begin{lemma}[see \cite{liveraniclt}]\label{notacoboundary}
    If $\sigma^2=0$, $\mathbbm{1}_B-\mu_{\mathcal{M}}(B)=\phi \circ f -\phi$ for some $\phi \in L^1$.
\end{lemma}

\begin{lemma}\label{variancesigma}
    If $\sigma^2\neq 0$, then $\sigma^2\in (0,\infty)$ and \[\int \Big(\frac{\sum_{0\le i\le n-1}[\mathbbm{1}_{B}\circ f^i-\mu_{\mathcal{M}}(B)]}{\sqrt{n}}\Big)^2d\mu_{\mathcal{M}}-\sigma^2=O(n^{-1+\max\{2-\min(\xi-1,\alpha),0\}}\log n).\]
\end{lemma}
\begin{proof}By Theorem \ref{mixingrate},
    $|\int [\mathbbm{1}_B-\mu_{\mathcal{M}}(B)]\circ f^i[\mathbbm{1}_B-\mu_{\mathcal{M}}(B)]d\mu_{\mathcal{M}}| \precsim i^{-\min\{\xi-1,\alpha\}}\to 0$. Hence
    \begin{align*}
    0&\le \int \Big(\frac{\sum_{0\le i\le n-1}[\mathbbm{1}_{B}\circ f^i-\mu_{\mathcal{M}}(B)]}{\sqrt{n}}\Big)^2d\mu_{\mathcal{M}}\\
    &=\int[\mathbbm{1}_{B}-\mu_{\mathcal{M}}(B)]^2 d\mu_{\mathcal{M}}+2\sum_{i \ge 1}\int [\mathbbm{1}_{B}-\mu_{\mathcal{M}}(B)][\mathbbm{1}_{B}-\mu_{\mathcal{M}}(B)]\circ f^{i}d\mu_{\mathcal{M}}\\
    &\quad -2n^{-1} \sum_{i\le n-1}\sum_{j> n-1-i}\int [\mathbbm{1}_{B}-\mu_{\mathcal{M}}(B)][\mathbbm{1}_{B}-\mu_{\mathcal{M}}(B)]\circ f^{j}d\mu_{\mathcal{M}}\\
    &=\sigma^2+n^{-1}O(\sum_{i \le n-1}(n-i)^{1-\min\{\xi-1,\alpha\}})=\sigma^2+O(n^{-1}n^{\max\{2-\min(\xi-1,\alpha),0\}}\log n) \to \sigma^2.
\end{align*} 

Hence $\sigma^2> 0$ and $\sigma^2=\int [\mathbbm{1}_{B}-\mu_{\mathcal{M}}(B)]^2 d\mu_{\mathcal{M}}+2\sum_{i \ge 1}\int [\mathbbm{1}_{B}-\mu_{\mathcal{M}}(B)][\mathbbm{1}_{B}-\mu_{\mathcal{M}}(B)]\circ f^{i}d\mu_{\mathcal{M}} \precsim  \sum_{i} i^{-\min\{\xi-1,\alpha\}} < \infty$.
\end{proof}

From now on we assume that $\sigma^2\in (0, \infty)$.
\subsection*{Step \ref{approximatedbyregularset}: approximation.}

Let $k\ge 1$ (to be determined), cover $F^{-k}\pi^{-1}B$ with a family of $ \xi_{2k,i}$ where each $\xi_{2k,i} \in \mathcal{Q}_{2k}$. Collect the $\xi_{2k,i}$ which is not completely contained in $F^{-k}\pi^{-1}B$, denote them by $\partial \xi_{2k,i}$ and let $A_2:=\bigcup_i \partial \xi_{2k,i}, A:=\bigcup_i \xi_{2k,i}$. By Lemma \ref{boundarymeasure}, $\mu_{\Delta}(A_2 )\precsim k^{-\alpha} \Leb_{\partial B}(\partial B)$, $ F^{-k}\pi^{-1}B \subseteq A \in \sigma(\mathcal{Q}_{2k})$. Then we have the following important lemma used often throughout the  remainder of the paper.
\begin{lemma}\label{approximation}
    There is $C>0$ such that for any $n>0$,
        \begin{align*}
        &\int \exp\Big(it\frac{\sum_{0\le i\le n-1}[\mathbbm{1}_B\circ f^i-\mu_{\mathcal{M}}(B)]}{\sqrt{n}}\Big)d\mu_{\mathcal{M}}\\
            &=\int \exp\Big(it\frac{\sum_{0\le i\le n-1}[\mathbbm{1}_{F^{-k}\pi^{-1}B}\circ F^i-\mu_{\Delta}(F^{-k}\pi^{-1}B)]}{\sqrt{n}}\Big)d\mu_{\Delta}\\
            &=\int \exp\Big(it\frac{\sum_{0\le i\le n-1}[\mathbbm{1}_{A}\circ F^i-\int \mathbbm{1}_{A} d\mu_{\Delta}]}{\sqrt{n}}\Big)d\mu_{\Delta}+O(|t| k^{-\min\{\xi-2,\alpha-1\}/2}), 
            \end{align*}
            \begin{align*}
            &\Big| \int\Big(\sum_{0\le i\le n-1}[\mathbbm{1}_{A}\circ F^i-\mu_{\Delta}(A)]\Big)^2d\mu_{\Delta}- \int\Big(\sum_{0\le i\le n-1}[\mathbbm{1}_{F^{-k}\pi^{-1}B}\circ F^i-\mu_{\Delta}(F^{-k}\pi^{-1}B)]\Big)^2d\mu_{\Delta}\Big|\\
            &=O(n k^{-\min\{\xi-2,\alpha-1\}})
            \end{align*}
            \begin{gather*}
                \int \Big(\sum_{0\le i\le n-1}[\mathbbm{1}_{A\setminus F^{-k}\pi^{-1}B}\circ F^i-\mu_{\Delta}(A\setminus F^{-k}\pi^{-1}B)]\Big)^2d\mu_{\Delta}=O(n k^{-\min\{\xi-2,\alpha-1\}})
                    \end{gather*}
where the constants in $O(\cdot)$ do not depend on $t, k, n$. 
\end{lemma}
\begin{proof}Let $B':=F^{-k}\pi^{-1}B$, note that 
     \begin{align*}
            &\Big|\int \exp\Big(it\frac{\sum_{0\le i\le n-1}[\mathbbm{1}_{B'}\circ F^i-\mu_{\Delta}(B')]}{\sqrt{n}}\Big)d\mu_{\Delta}-\int \exp\Big(it\frac{\sum_{0\le i\le n-1}[\mathbbm{1}_{A}\circ F^i-\int \mathbbm{1}_{A} d\mu_{\Delta}]}{\sqrt{n}}\Big)d\mu_{\Delta}\Big| \\
            & \le |t|\int \Big|\frac{\sum_{0\le i\le n-1}[\mathbbm{1}_{A\setminus B'}\circ F^i-\mu_{\Delta}(A\setminus B')]}{\sqrt{n}}\Big|d\mu_{\Delta}\\
            &\le |t|\sqrt{\int \frac{(\sum_{0\le i\le n-1}[\mathbbm{1}_{A\setminus B'}\circ F^i-\mu_{\Delta}(A\setminus B')])^2}{n}d\mu_{\Delta}}.
 \end{align*} We can expand the square term above as follows.
 \begin{align}
            &\int \frac{(\sum_{0\le i\le n-1}[\mathbbm{1}_{A\setminus B'}\circ F^i-\mu_{\Delta}(A\setminus B')])^2}{n}d\mu_{\Delta}\nonumber \\
            &=  \int [\mathbbm{1}_{A\setminus B'}-\mu_{\Delta}(A\setminus B')]^2 d\mu_{\Delta}+2n^{-1} \sum_{i<j\le n-1}\int [\mathbbm{1}_{A\setminus B'}-\mu_{\Delta}(A\setminus B')][\mathbbm{1}_{A\setminus B'}-\mu_{\Delta}(A\setminus B')]\circ F^{j-i}d\mu_{\Delta}\nonumber\\
            &\le  \int [\mathbbm{1}_{A\setminus B'}-\mu_{\Delta}(A\setminus B')]^2 d\mu_{\Delta}+2\sum_{i \ge 1}\Big|\int [\mathbbm{1}_{A\setminus B'}-\mu_{\Delta}(A\setminus B')][\mathbbm{1}_{A\setminus B'}-\mu_{\Delta}(A\setminus B')]\circ F^{i}d\mu_{\Delta}\Big|\nonumber\\
            &\quad +2n^{-1} \sum_{i\le n-1}\sum_{j> n-1-i}\Big|\int [\mathbbm{1}_{A\setminus B'}-\mu_{\Delta}(A\setminus B')][\mathbbm{1}_{A\setminus B'}-\mu_{\Delta}(A\setminus B')]\circ F^{j}d\mu_{\Delta}\Big|.\label{5}
        \end{align}

Similarly, we can expand the following square terms  \begin{align}
            &\Big| \int\Big(\sum_{0\le i\le n-1}[\mathbbm{1}_{A}\circ F^i-\mu_{\Delta}(A)]\Big)^2d\mu_{\Delta}- \int\Big(\sum_{0\le i\le n-1}[\mathbbm{1}_{F^{-k}\pi^{-1}B}\circ F^i-\mu_{\Delta}(F^{-k}\pi^{-1}B)]\Big)^2d\mu_{\Delta}\Big|\nonumber \\
            & \le \int \Big(\sum_{0\le i\le n-1}[\mathbbm{1}_{A\setminus B'}\circ F^i-\mu_{\Delta}(A\setminus B')]\Big)^2d\mu_{\Delta}\nonumber\\
            &\quad +2\Big|\int \Big(\sum_{0\le i\le n-1}[\mathbbm{1}_{A\setminus B'}\circ F^i-\mu_{\Delta}(A\setminus B')]\Big) \Big(\sum_{0\le i\le n-1}[\mathbbm{1}_{A
        }\circ F^i-\mu_{\Delta}(A)]\Big)d\mu_{\Delta}\Big|\nonumber\\
            & \le \int \Big(\sum_{0\le i\le n-1}[\mathbbm{1}_{A\setminus B'}\circ F^i-\mu_{\Delta}(A\setminus B')]\Big)^2d\mu_{\Delta}\nonumber\\
            &\quad +2\sum_{j=i\le n-1}\Big|\int [\mathbbm{1}_{A\setminus B'}\circ F^i-\mu_{\Delta}(A\setminus B')] [\mathbbm{1}_{A
        }\circ F^j-\mu_{\Delta}(A)]d\mu_{\Delta}\Big|\nonumber\\
            &\quad +2\sum_{i<j\le n-1}\Big|\int [\mathbbm{1}_{A\setminus B'}\circ F^i-\mu_{\Delta}(A\setminus B')] [\mathbbm{1}_{A
        }\circ F^j-\mu_{\Delta}(A)]d\mu_{\Delta}\Big|\nonumber\\
            &\quad +2\sum_{j<i\le n-1}\Big|\int [\mathbbm{1}_{A\setminus B'}\circ F^i-\mu_{\Delta}(A\setminus B')] [\mathbbm{1}_{A}\circ F^j-\mu_{\Delta}(A)]d\mu_{\Delta}\Big|\nonumber\\
         &\le \int \Big(\sum_{0\le i\le n-1}[\mathbbm{1}_{A\setminus B'}\circ F^i-\mu_{\Delta}(A\setminus B')]\Big)^2d\mu_{\Delta}+4n\mu_{\Delta}(A_2)\nonumber\\
        & \quad +2n\sum_{i \ge 1}\Big|\int [\mathbbm{1}_{A\setminus B'}-\mu_{\Delta}(A\setminus B')][\mathbbm{1}_{A}-\mu_{\Delta}(A)]\circ F^{i}d\mu_{\Delta}\Big|\nonumber\\
            &\quad +2 \sum_{i\le n-1}\sum_{j> n-1-i}\Big|\int [\mathbbm{1}_{A\setminus B'}-\mu_{\Delta}(A\setminus B')][\mathbbm{1}_{A}-\mu_{\Delta}(A)]\circ F^{j}d\mu_{\Delta}\Big|\nonumber\\
        & \quad +2n\sum_{i \ge 1}\Big|\int [\mathbbm{1}_{A\setminus B'}-\mu_{\Delta}(A\setminus B')]\circ F^{i}[\mathbbm{1}_{A}-\mu_{\Delta}(A)]d\mu_{\Delta}\Big| \nonumber\\
            &\quad +2 \sum_{i\le n-1}\sum_{j> n-1-i}\Big|\int [\mathbbm{1}_{A\setminus B'}-\mu_{\Delta}(A\setminus B')]\circ F^{j}[\mathbbm{1}_{A}-\mu_{\Delta}(A)]d\mu_{\Delta}\Big|. \label{100}
            \end{align} where $\mu_{\Delta}(A_2)$ is due to Lemma \ref{boundarymeasure} and $\mu_{\Delta}(A\setminus B')\le \mu_{\Delta}(A_2)$. Now we estimate (\ref{5}) and (\ref{100}). Take (\ref{5}) for example. \begin{align}
            &|\int [\mathbbm{1}_{A\setminus B'}-\mu_{\Delta}(A\setminus B')][\mathbbm{1}_{A\setminus B'}-\mu_{\Delta}(A\setminus B')]\circ F^{n}d\mu_{\Delta}| \nonumber\\
            &=|\int \mathbbm{1}_{A\setminus B'}\mathbbm{1}_{A\setminus B'}\circ F^{n}d\mu_{\Delta}-(\int \mathbbm{1}_{A\setminus B'}d\mu_{\Delta})^2|\nonumber\\
            &=|\int \mathbbm{1}_{F^{-m}(A\setminus B')}\mathbbm{1}_{F^{-m}(A\setminus B')}\circ F^{n}d\mu_{\Delta}-(\int \mathbbm{1}_{F^{-m}(A\setminus B')}d\mu_{\Delta})^2| \label{4}
        \end{align} ($m$ to be determined). Cover $F^{-m}(A\setminus B')$ with a family of $ \xi_{2m+2k,i} $ where $\xi_{2m+2k,i}\in \mathcal{Q}_{2(m+k)}$, collect the $\xi_{2m+2k,i}$, which is not totally contained in $F^{-m}(A\setminus B')$, denote them by $\partial \xi_{2m+2k,i}$.

        \textbf{Claim:} for each $\xi_{2m+2k,i}$, $F^m\xi_{2m+2k,i}$ is contained in one $\partial \xi_{2k,j}$, and $\mu_{\Delta}(\bigcup_l \partial\xi_{2m+2k,l}) \precsim (m+k)^{-\alpha}\Leb_{\partial B}(\partial B)$. 

        Here is the proof: note that $F^m \xi_{2m+2k,i} \bigcap \bigcup_i \partial \xi_{2k,i} \neq \emptyset$, say, $F^m \xi_{2m+2k,i} \bigcap \partial \xi_{2k,j} \neq \emptyset$ for some $j$. Also note that  $\widetilde{F}^m \widetilde{\xi}_{2m+2k,i}\in \widetilde{\mathcal{Q}}_{m+2k}$, so $\widetilde{F}^m \widetilde{\xi}_{2m+2k,i} \subseteq \partial \widetilde{\xi}_{2k,j}$. This implies $F^m \xi_{2m+2k,i} \subseteq \partial \xi_{2k,j}$.
        
        Besides, by (\ref{0}), $\diam (\pi F^m F^k \partial \xi_{2m+2k,i})\precsim (m+k)^{-\alpha}$. By definition $\pi F^m F^k \partial \xi_{2m+2k,i}\bigcap B \neq \emptyset$ and $\pi F^m F^k \partial \xi_{2m+2k,i}\bigcap B^c \neq \emptyset$. Hence $\mu_{\Delta}(\bigcup_i \partial\xi_{2m+2k,i})\le \mu_{\Delta}( F^{-m-k}\pi^{-1}\pi F^m F^k  \bigcup_i \partial\xi_{2m+2k,i})=\mu_{\mathcal{M}}(\pi F^m F^k  \bigcup_i \partial\xi_{2m+2k,i})\precsim \Leb_{\partial B}(\partial B) (m+k)^{-\alpha}$. So this claim is concluded.
        
        Together with Lemma \ref{youngmixing} we can continue to estimate (\ref{4}). For any $n> 4k$, choose $m=\lfloor n/4 \rfloor-k$,
        \begin{align*}
            &\le |\int [\mathbbm{1}_{\bigcup_i \xi_{2m+2k,i}}\mathbbm{1}_{\bigcup_i \xi_{2m+2k,i}}\circ F^{n}d\mu_{\Delta}-(\int [\mathbbm{1}_{\bigcup_i \xi_{2m+2k,i}}d\mu_{\Delta})^2|+O(\mu_{\Delta}(\bigcup_i \partial \xi_{2m+2k,i}))\\
            &\precsim (n-2m-2k)^{-\xi+1}+(m+k)^{-\alpha} \precsim n^{-\xi+1}+n^{-\alpha} \precsim n^{-\min\{\xi-1,\alpha\}}.
        \end{align*}
        
        When $n\le 4k$,  (\ref{4})$=O(\mu_{\Delta}(A_2))=O(k^{-\alpha})$. Therefore, \[\sum_{j\ge 1}\Big|\int [\mathbbm{1}_{A\setminus B'}-\mu_{\Delta}(A\setminus B')][\mathbbm{1}_{A\setminus B'}-\mu_{\Delta}(A\setminus B')]\circ F^{j}d\mu_{\Delta}\Big|
            \precsim \sum_{i > 4k}i^{-\min\{\xi-1,\alpha\}} +2\sum_{i \le 4k} k^{-\alpha} \precsim k^{-\min\{\xi-2,\alpha-1\}}.\] 

            To address (\ref{100}), the arguments are the same, but cover $F^{-m}(A)$ by another family of $\xi'_{2m+2k,i}$ with $\partial \xi'_{2m+2k,i}=\emptyset$. So the same bounds are obtained as follows. 
            \[\sum_{j\ge 1}\Big|\int [\mathbbm{1}_{A}-\mu_{\Delta}(A)][\mathbbm{1}_{A\setminus B'}-\mu_{\Delta}(A\setminus B')]\circ F^{j}d\mu_{\Delta}\Big|
           \precsim k^{-\min\{\xi-2,\alpha-1\}},\] 
            \[\sum_{j\ge 1}\Big|\int [\mathbbm{1}_{A\setminus B'}-\mu_{\Delta}(A\setminus B')][\mathbbm{1}_{A}-\mu_{\Delta}(A)]\circ F^{j}d\mu_{\Delta}\Big|
            \precsim k^{-\min\{\xi-2,\alpha-1\}}.\]

            Hence (\ref{5}) and (\ref{100}) can be estimated as 
        \begin{align*}
            (\ref{5}) &\le  \int [\mathbbm{1}_{A\setminus B'}-\mu_{\Delta}(A\setminus B')]^2 d\mu_{\Delta}+2\sum_{i \ge 1}\Big|\int [\mathbbm{1}_{A\setminus B'}-\mu_{\Delta}(A\setminus B')][\mathbbm{1}_{A\setminus B'}-\mu_{\Delta}(A\setminus B')]\circ F^{i}d\mu_{\Delta}\Big|\\
            &\quad +2n^{-1} \sum_{i\le n-1}\sum_{j> n-1-i}\Big|\int [\mathbbm{1}_{A\setminus B'}-\mu_{\Delta}(A\setminus B')][\mathbbm{1}_{A\setminus B'}-\mu_{\Delta}(A\setminus B')]\circ F^{j}d\mu_{\Delta}\Big|\\
            &\precsim  k^{-\alpha}+ k^{-\min\{\xi-2,\alpha-1\}}+n^{-1} n k^{-\min\{\xi-2,\alpha-1\}}\precsim  k^{-\min\{\xi-2,\alpha-1\}}.
\end{align*}
\begin{align*}
   (\ref{100})&\le \int \Big(\sum_{0\le i\le n-1}[\mathbbm{1}_{A\setminus B'}\circ F^i-\mu_{\Delta}(A\setminus B')]\Big)^2d\mu_{\Delta}+nO(\mu_{\Delta}(A_2))+8n k^{-\min\{\xi-2,\alpha-1\}} \nonumber\\
   &\precsim n k^{-\min\{\xi-2,\alpha-1\}}+n k^{-\alpha}+8n k^{-\min\{\xi-2,\alpha-1\}}\precsim n k^{-\min\{\xi-2,\alpha-1\}}
\end{align*}which concludes the proof.
\end{proof} 

\subsection*{Step \ref{severalvariances}:  prove the existence of several variances and estimate their differences.}

\begin{lemma}\label{variance}
Recall from Lemma \ref{variancesigma} that $\sigma^2=\lim_{n \to \infty}\int \Big(\frac{\sum_{0\le i\le n-1}[\mathbbm{1}_{B}\circ f^i-\int \mathbbm{1}_{B} d\mu_{\mathcal{M}}]}{\sqrt{n}}\Big)^2d\mu_{\mathcal{M}}$. Similarly, the limit 
   \[\sigma_k^2:=\lim_{n \to \infty}\int \Big(\frac{\sum_{0\le i\le n-1}[\mathbbm{1}_{A}\circ F^i-\int \mathbbm{1}_{A} d\mu_{\Delta}]}{\sqrt{n}}\Big)^2d\mu_{\Delta}\]exists
   and the estimates $\sigma_k^2-\sigma^2=O(k^{-\min\{\xi-2,\alpha-1\}})$,
    \[\int \Big(\frac{\sum_{0\le i\le n-1}[\mathbbm{1}_{A}\circ F^i-\int \mathbbm{1}_{A} d\mu_{\Delta}]}{\sqrt{n}}\Big)^2d\mu_{\Delta}-\sigma_k^2=O( k^2 n^{-\min\{\xi-2,1\}}\log n)\] hold where the constants in $O(\cdot)$ do not depend on $n,k$.
\end{lemma}
\begin{proof}
By Lemma \ref{youngmixing}, when $i\ge 4k $, $|\int [\mathbbm{1}_{A}-\mu_{\Delta}(A)][\mathbbm{1}_{A}-\mu_{\Delta}(A)]\circ F^{i}d\mu_{\Delta}|\precsim (i-2k)^{-\xi+1} \precsim i^{-\xi+1}$. Hence $\sum_{i \ge 1}\int [\mathbbm{1}_{A}-\mu_{\Delta}(A)][\mathbbm{1}_{A}-\mu_{\Delta}(A)]\circ F^{i}d\mu_{\Delta}$ converges. If $i\ge 4k$, $|\sum_{j \ge i}\int [\mathbbm{1}_{A}-\mu_{\Delta}(A)][\mathbbm{1}_{A}-\mu_{\Delta}(A)]\circ F^{j}d\mu_{\Delta}|\precsim \sum_{j \ge i}j^{-\xi+1}\precsim i^{-\xi+2} \to 0$. If $i\le 4k$, $|\sum_{j \ge i}\int [\mathbbm{1}_{A}-\mu_{\Delta}(A)][\mathbbm{1}_{A}-\mu_{\Delta}(A)]\circ F^{j}d\mu_{\Delta}|\precsim  k^{-\xi+2} +4k=O(k)$. Therefore

\begin{align*}
    &\lim_{n \to \infty}\int \Big(\frac{\sum_{0\le i\le n-1}[\mathbbm{1}_{A}\circ F^i-\int \mathbbm{1}_{A} d\mu_{\Delta}]}{\sqrt{n}}\Big)^2d\mu_{\Delta}\\
    &=\int [\mathbbm{1}_{A}-\mu_{\Delta}(A)]^2 d\mu_{\Delta}+2\sum_{i \ge 1}\int [\mathbbm{1}_{A}-\mu_{\Delta}(A)][\mathbbm{1}_{A}-\mu_{\Delta}(A)]\circ F^{i}d\mu_{\Delta}\\
    &\quad -2\lim_{n \to \infty}n^{-1} \sum_{i\le n-1}\sum_{j> n-1-i}\int [\mathbbm{1}_{A}-\mu_{\Delta}(A)][\mathbbm{1}_{A}-\mu_{\Delta}(A)]\circ F^{j}d\mu_{\Delta}\\
    &=\int [\mathbbm{1}_{A}-\mu_{\Delta}(A)]^2 d\mu_{\Delta}+2\sum_{i \ge 1}\int [\mathbbm{1}_{A}-\mu_{\Delta}(A)][\mathbbm{1}_{A}-\mu_{\Delta}(A)]\circ F^{i}d\mu_{\Delta}=:\sigma_k^2.
\end{align*} 

The approximation error is 
\begin{align*}
    &\int \Big(\frac{\sum_{0\le i\le n-1}[\mathbbm{1}_{A}\circ F^i-\int \mathbbm{1}_{A} d\mu_{\Delta}]}{\sqrt{n}}\Big)^2d\mu_{\Delta}-\sigma_k^2\\
    &=-2n^{-1} \sum_{i\le n-1}\sum_{j> i}\int [\mathbbm{1}_{A}-\mu_{\Delta}(A)][\mathbbm{1}_{A}-\mu_{\Delta}(A)]\circ F^{j}d\mu_{\Delta}\\
    &=-2n^{-1} \sum_{4k\le i\le n-1}\sum_{j> i}\int [\mathbbm{1}_{A}-\mu_{\Delta}(A)][\mathbbm{1}_{A}-\mu_{\Delta}(A)]\circ F^{j}d\mu_{\Delta}\\
    &\quad -2n^{-1} \sum_{i \le 4k}\sum_{j> i}\int [\mathbbm{1}_{A}-\mu_{\Delta}(A)][\mathbbm{1}_{A}-\mu_{\Delta}(A)]\circ F^{j}d\mu_{\Delta}\\
    &=n^{-1} \sum_{4k\le i\le n-1}O(i^{-\xi+2})+n^{-1} \sum_{ i\le 4k}O(k)\\
    &\precsim n^{-\min\{\xi-2,1\}}\log n +k^2n^{-1}\precsim k^2 n^{-\min\{\xi-2,1\}}\log n.
\end{align*} 

It remains to estimate $\sigma_k^2-\sigma^2$. It follows from Lemma \ref{approximation} that 
\begin{align*}
    &\Big|\int \Big(\frac{\sum_{0\le i\le n-1}[\mathbbm{1}_{B}\circ f^i-\int \mathbbm{1}_{B} d\mu_{\mathcal{M}}]}{\sqrt{n}}\Big)^2d\mu_{\mathcal{M}}-\int \Big(\frac{\sum_{0\le i\le n-1}[\mathbbm{1}_{A}\circ F^i-\int \mathbbm{1}_{A} d\mu_{\Delta}]}{\sqrt{n}}\Big)^2d\mu_{\Delta}\Big|\\
    &\precsim k^{-\min\{\xi-2,\alpha-1\}}.
\end{align*}

Taking $n\to \infty$, we conclude the estimate $|\sigma_k^2-\sigma^2|\precsim  k^{-\min\{\xi-2,\alpha-1\}}$.
\end{proof}

\subsection*{Step \ref{verifysutams}: verify the conditions (A4)-(A6) in \cite{Sutams}.}  
Now we collapse $A$ to $\widetilde{A}$ along stable manifolds and let $\psi_k:=\mathbbm{1}_{\widetilde{A}}-\mu_{\widetilde{\Delta}}(\widetilde{A})$. Suppose that the transfer operator of $\widetilde{F}$ is $P$ and the Jacobian of $\widetilde{F}$ (w.r.t. $\mu_{\widetilde{\Delta}}$) is $J\widetilde{F}$. According to \cite{Y}, the bounded distortion similar to the one in section \ref{cmz} still holds: for any $z_1,z_2 $ in the same element of $\widetilde{\mathcal{Q}}_1$, $
\big|\log \frac{|J\widetilde{F}(z_1)|}{|J\widetilde{F}(z_2)|}\big|
\le C \, \beta^{\, s(\widetilde{F}(z_1),\, \widetilde{F}(z_2))}$. We now verify (A4)-(A6) in \cite{Sutams}.

\begin{lemma}\label{a3toa6} 
 For any $j\ge 0,n,k>0$, 
\[\int |P^n( \psi_k)| d\mu_{\widetilde{\Delta}} \precsim k^{\xi-1} n^{-\xi+1}, \quad  \int |P^n(\psi_k^2-\int \psi_k^2 d\mu_{\widetilde{\Delta}})| d\mu_{\widetilde{\Delta}} \precsim k^{\xi-1}n^{-\xi+1},\]
\[\int \Big|P^n\Big[P^j(\psi_k) \psi_k - \int P^j(\psi_k)  \psi_k d\mu_{\widetilde{\Delta}}\Big]\Big| d\mu_{\widetilde{\Delta}} \precsim k^{\xi-1}n^{-\xi+1}\]where the constants in $\precsim$ do not depend on $j,n,k$.
\end{lemma}
\begin{proof}
    Since $ \widetilde{A}=\bigcup_i \widetilde{\xi}_{2k,i} \in \sigma(\widetilde{\mathcal{Q}}_{2k})$. By Lemma \ref{youngmixing}, when $n \ge 4k$, $\int |P^n( \psi_k)| d\mu_{\widetilde{\Delta}} \precsim (n-2k)^{-\xi+1} \precsim n^{-\xi+1}$. When $n \le 4k$, $\int |P^n( \psi_k)| d\mu_{\widetilde{\Delta}}=O(1)$. So $\int |P^n( \psi_k)| d\mu_{\widetilde{\Delta}} \precsim k^{\xi-1}n^{-\xi+1}$ holds for all $n$.

    Since $\psi^2_k-\int \psi^2_k d\mu_{\widetilde{\Delta}}=\mathbbm{1}_{\widetilde{A}}-\mu_{\widetilde{\Delta}}(\widetilde{A})-2[\mathbbm{1}_{\widetilde{A}}-\mu_{\widetilde{\Delta}}(\widetilde{A})]\mu_{\widetilde{\Delta}}(\widetilde{A})$, Lemma \ref{youngmixing} gives $\int |P^n(\psi_k^2-\int \psi_k^2 d\mu_{\widetilde{\Delta}})| d\mu_{\widetilde{\Delta}} \precsim k^{\xi-1}n^{-\xi+1}$. 

    Now we prove the third inequality. For any $x_1,x_2$ belonging to the same element of $\widetilde{Q}_1$, \begin{align*}
        |P^{2k}(&P^j(\psi_k)\psi_k)(x_1)-P^{2k}(P^j(\psi_k)\psi_k)(x_2)|\\
        &=\Big|\sum_{\widetilde{F}^{2k}(y_1)=x_1}\frac{P^j(\psi_k)(y_1)\psi_k(y_1)}{|J\widetilde{F}^{2k}(y_1)|}-\sum_{\widetilde{F}^{2k}(y_2)=x_2}\frac{P^j(\psi_k)(y_2)\psi_k(y_2)}{|J\widetilde{F}^{2k}(y_2)|}\Big|\\
        &=\Big|\sum_{\widetilde{F}^{2k}(y_1)=x_1}\Big[\frac{\psi_k(y_1)}{|J\widetilde{F}^{2k}(y_1)|}-\frac{\psi_k(y_2)}{|J\widetilde{F}^{2k}(y_2)|}\Big] \sum_{\widetilde{F}^j(z_1)=y_1}\frac{\psi_k(z_1)}{|J\widetilde{F}^j(z_1)|}\\
        &\quad+\sum_{\widetilde{F}^{2k}(y_2)=x_2}\frac{\psi_k(y_2)}{|J\widetilde{F}^{2k}(y_2)|} \sum_{\widetilde{F}^j(z_2)=y_2}\Big[\frac{\psi_k(z_1)}{|J\widetilde{F}^j(z_1)|}-\frac{\psi_k(z_2)}{|J\widetilde{F}^j(z_2)|} \Big]\Big|
    \end{align*}where $y_1,y_2$ belong to the same element of $\widetilde{Q}_{2k+1}$ and $z_1,z_2$ belong to the same element of $\widetilde{Q}_{2k+j+1}$. Therefore, $\psi_k(y_1)=\psi_k(y_2), \psi_k(z_1)=\psi_k(z_2)$. By distortions and Markovian properties described in section \ref{cmz}, $\sum_{\widetilde{F}^j(z_1)=y_1}\frac{1}{|J\widetilde{F}^j(z_1)|}=O(1)$, $\sum_{\widetilde{F}^j(z_2)=y_2}\frac{1}{|J\widetilde{F}^j(z_2)|}=O(1)$ and $|\frac{1}{|J\widetilde{F}^{2k}(y_1)|}-\frac{1}{|J\widetilde{F}^{2k}(y_2)|}|\precsim \frac{1}{|J\widetilde{F}^{2k}(y_1)|}\beta^{s(\widetilde{F}^{2k}(y_1), \widetilde{F}^{2k}(y_2))} $, $|\frac{1}{|J\widetilde{F}^{j}(z_1)|}-\frac{1}{|J\widetilde{F}^{j}(z_2)|}|\precsim \frac{1}{|J\widetilde{F}^{j}(z_2)|}\beta^{s(\widetilde{F}^{j}(z_1), \widetilde{F}^{j}(z_2))}\precsim \frac{1}{|J\widetilde{F}^{j}(z_2)|}\beta^{s(x_1, x_2)}$. Hence we can continue our estimate
    \begin{align*}
        &\precsim \sum_{\widetilde{F}^{2k}(y_1)=x_1} \frac{1}{|J\widetilde{F}^{2k}(y_1)|}\beta^{s(\widetilde{F}^{2k}(y_1), \widetilde{F}^{2k}(y_2))}+\sum_{\widetilde{F}^{2k}(y_2)=x_2}\frac{1}{|J\widetilde{F}^{2k}(y_2)|} \sum_{\widetilde{F}^j(z_2)=y_2}\frac{1}{|J\widetilde{F}^{j}(z_2)|}\beta^{s(x_1,x_2)}\\
        &\precsim \beta^{s(x_1,x_2)} +\sum_{\widetilde{F}^{2k}(y_2)=x_2}\frac{1}{|J\widetilde{F}^{2k}(y_2)|} \beta^{s(x_1,x_2)} \precsim \beta^{s(x_1,x_2)}
    \end{align*} where the constant in $\precsim$ does not depend on $k,j$. Hence $P^{2k}(P^j(\psi_k)\psi_k)$ is a locally H\"older function. Applying the result of \cite{Y1}, when $n \ge 4k$, $\int |P^n[P^j(\psi_k) \psi_k - \int P^j(\psi_k)  \psi_k d\mu_{\widetilde{\Delta}}]| d\mu_{\widetilde{\Delta}}=O((n-2k)^{-\xi+1})=O(n^{-\xi+1})$. When $n\le 4k$, $\int |P^n[P^j(\psi_k) \psi_k - \int P^j(\psi_k)  \psi_k d\mu_{\widetilde{\Delta}}]| d\mu_{\widetilde{\Delta}}=O(1)$. Hence, for any $n \ge 1$,  \[\int |P^n[P^j(\psi_k) \psi_k - \int P^j(\psi_k)  \psi_k d\mu_{\widetilde{\Delta}}]| d\mu_{\widetilde{\Delta}}\precsim k^{\xi-1}n^{-\xi+1}\]where the constant in $\precsim$ does not depend on $j,n,k$.
\end{proof}

With the decay of correlations in Lemma \ref{a3toa6}, we can obtain the following moment inequalities in \cite{Sutams}. For convenience, we introduce some new notation. Let $\Psi_k:=\mathbbm{1}_{A}-\mu_{\Delta}(A)$, $\mathbb{\widetilde{E}}(\cdot):=\int (\cdot) d\mu_{\widetilde{\Delta}}$, $\mathbb{\widetilde{E}}_{n}(\cdot):=\mathbb{\widetilde{E}}[(\cdot)|(\widetilde{F})^{-n}\sigma(\bigcup_{i\ge 0} \mathcal{\widetilde{Q}}_{i})]$, $\mathbb{{E}}(\cdot):=\int (\cdot) d\mu_{{\Delta}}$, $\mathbb{{E}}_{n}(\cdot):=\mathbb{{E}}[(\cdot)|({F})^{-n}\sigma(\bigcup_{i\ge 0} \mathcal{Q}_{i})]$. 

\begin{lemma}\label{momentbound}
$\mathbb{\widetilde{E}}|\mathbb{\widetilde{E}}_{n+m} \sum_{i=m}^{n+m-1 } \psi_k\circ \widetilde{F}^i |\precsim k^{\xi-1}$, $\mathbb{{E}}|\mathbb{{E}}_{n+m} \sum_{i=m}^{n+m-1 } \Psi_k\circ {F}^i |\precsim k^{\xi-1}$,
\[\mathbb{\widetilde{E}}\Big|\mathbb{\widetilde{E}}_{n+m} \Big[\Big(\sum_{i=m }^{n+m-1}  \psi_k\circ \widetilde{F}^i\Big)^2 \Big]-\mathbb{\widetilde{E}} \Big[\Big(\sum_{i=m }^{n+m-1}  \psi_k\circ \widetilde{F}^i\Big)^2\Big] \Big|\precsim k^{\xi-1} n^{\frac{1}{\xi-1}},\]
\begin{align*}
    &\Big|\mathbb{\widetilde{E}} \Big(\sum_{i=m}^{n+m} \psi_k\circ \widetilde{F}^i\Big)\Big(\sum_{i \in m+n+\bigcup_{j}[p_j,q_j]}\psi_k\circ \widetilde{F}^i\Big) \Big|\\
    &\precsim k^{\xi-1}\max(n,\max_j q_j-\min_jp_j)^{\max\{3-\xi, 0\}}\log \max(n,\max_j q_j-\min_jp_j),
\end{align*} 
\[\mathbb{\widetilde{E}} \Big(\sum_{m \le i \le m+n} \psi_k \circ \widetilde{F}^i\Big)^2 \precsim k^{\xi-1} n,\quad  \mathbb{\widetilde{E}}\Big|\sum_{m\le i \le m+n} \psi_k \circ \widetilde{F}^i\Big|^{\frac{3\xi-4}{\xi-1}}\precsim k^{\frac{(3\xi-4)(\xi-1)}{\xi-2}} n^{\frac{3\xi-4}{2\xi-2}},\]

\[\mathbb{E}\Big|\mathbb{{E}}_{n+m} \Big[\Big(\sum_{i=m }^{n+m-1}  \Psi_k\circ {F}^i\Big)^2 \Big]-\mathbb{{E}} \Big[\Big(\sum_{i=m }^{n+m-1}  \Psi_k\circ {F}^i\Big)^2\Big] \Big|\precsim k^{\xi-1} n^{\frac{1}{\xi-1}}.\]
\begin{align*}
    &\Big|\mathbb{{E}} \Big(\sum_{i=m}^{n+m} \Psi_k\circ {F}^i\Big)\Big(\sum_{i\in m+n+\bigcup_{j}[p_j,q_j]}\Psi_k\circ {F}^i\Big) \Big|\\
    &\precsim k^{\xi-1}\max(n,\max_j q_j-\min_jp_j)^{\max\{3-\xi, 0\}}\log \max(n,\max_j q_j-\min_jp_j),
\end{align*}
\[\mathbb{{E}} \Big(\sum_{m \le i \le m+n} \Psi_k \circ {F}^i\Big)^2 \precsim k^{\xi-1} n,\quad  \mathbb{{E}}\Big|\sum_{m\le i \le m+n} \Psi_k \circ {F}^i\Big|^{\frac{3\xi-4}{\xi-1}}\precsim k^{\frac{(3\xi-4)(\xi-1)}{\xi-2}} n^{\frac{3\xi-4}{2\xi-2}},\]
where the constants in $\precsim$ do not depend on any $m,k,n\in \mathbb{N}$ or on the disjoint time intervals $[p_j,q_j]$. 

\begin{align*}
    &\Big|\mathbb{{E}} \Big(\sum_{i\in I} \Psi_k\circ {F}^i\Big)^2-\mathbb{{E}} \Big(\sum_{i\in J} \Psi_k\circ {F}^i\Big)^2\Big| \precsim k^{\xi-1}|I\setminus J|,
\end{align*}where the constant in $\precsim$ does not depend on $k$ and any finite subsets $J\subseteq I$ in $\mathbb{N}_0$. 
\end{lemma}
\begin{proof}
    In step \ref{verifysutams} Lemma \ref{a3toa6}, we have verified the conditions (A4)-(A6) in \cite{Sutams}; we now conclude these inequalities by applying Lemma \ref{a3toa6} and mimicking the arguments of Lemmas 4.2, 4.3, 4.4, 4.6 and 4.7 in \cite{Sutams}, where we take $\epsilon=\frac{\xi-2}{\xi-1}$ in Lemma 4.7 of \cite{Sutams}. The conclusions for $\Psi_k$ are due to $(\widetilde{\pi}_{\Delta})_{*}\mu_{\Delta}=\mu_{\widetilde{\Delta}}$ and $\widetilde{F} \circ \widetilde{\pi}_{\Delta} =\widetilde{\pi}_{\Delta}\circ F$. For the last inequality, note that
    \begin{align*}
    &\Big|\mathbb{{E}} \Big(\sum_{i\in I} \Psi_k\circ {F}^i\Big)^2-\mathbb{{E}} \Big(\sum_{i\in J} \Psi_k\circ {F}^i\Big)^2\Big|\le \Big|\mathbb{{E}} \Big(\sum_{i\in I\setminus J} \Psi_k\circ {F}^i\Big)^2\Big|+2\Big|\mathbb{{E}} \Big(\sum_{i\in I\setminus J} \Psi_k\circ {F}^i\Big)\Big(\sum_{i\in J} \Psi_k\circ {F}^i\Big)\Big|\\
    &\le 2\sum_{i \in I \setminus J}\sum_{j \in  I \setminus J}|\mathbb{{E}}\Psi_k\circ F^i \Psi_k \circ F^j|+2\sum_{i \in I \setminus J}\sum_{j \in  J}|\mathbb{{E}}\Psi_k\circ F^i \Psi_k \circ F^j| \\
     &=2\sum_{i \in I \setminus J}|\mathbb{{E}}\Psi_k^2|+2\sum_{i \in I \setminus J}\sum_{i \neq j \in  I \setminus J}|\mathbb{{E}}\Psi_k\circ F^i \Psi_k \circ F^j|+2\sum_{i \in I \setminus J}\sum_{j \in J}|\mathbb{{E}}\Psi_k\circ F^i \Psi_k \circ F^j| \precsim k^{\xi-1}|I\setminus J|
    \end{align*}where the last ``$\precsim$" is due to Lemma \ref{mixingrate} and $F$-invariance.
    
    \end{proof}

\subsection*{Step \ref{estimatepreclt}: employ the block methods in \cite{Sutams} to estimate}  

\[\int \exp\Big(it\frac{\sum_{0\le i\le n-1}[\mathbbm{1}_{A}\circ F^i-\int \mathbbm{1}_{A} d\mu_{\Delta}]}{\sqrt{n}}\Big)d\mu_{\Delta}=\int \exp\Big(it\frac{\sum_{0\le i\le n-1}\psi_k \circ \widetilde{F}^i}{\sqrt{n}}\Big)d\mu_{\widetilde{\Delta}}.\] 

We follow the proof of Corollary 3.1 of \cite{Sutams}: let $I_n=[0,n-1]\bigcap \mathbb{N}_0$, $a \in (1/2,1)$, $c_n: =\lfloor n^{(1-a)} \rfloor$. Define disjoint blocks $\{I_{n,i}, 1\le i \le c_n+1\}$ of consecutive nonnegative integers in $I_n$ as shown in the image below. $|I_{n,i}|=\lfloor n^{a} \rfloor$ for any $i\in [1, c_n]$, the last block $I_{n,c_n+1}:= I_n \setminus \bigcup_{1 \le i \le c_n} I_{n,i}$. Thus $|I_{n,c_n+1}|\le 3\lfloor n^{a} \rfloor$ and $\bigcup_{1 \le i \le c_n+1} I_{n,i}=I_n$. Let $X_{n, i}:= \sum_{j \in I_{n,i}} \psi_k \circ \widetilde{F}^j$ and $a_i:=\sum_{j\le i}|I_{n,j}|+1$, then we have

\begin{tikzpicture}[xscale=0.35, yscale=0.6, >=Stealth, every node/.style={font=\small}]
    \draw[thick, ->] (-2,0) -- (37,0) node[right] {$\mathbb{N}$};

    \foreach \x/\label in {0/, 8/a_1-1, 16/a_2-1, 24/a_{c_n}-1, 32/a_{c_n+1}-1} {
        \draw (\x, -0.2) -- (\x, 0.2);
        \node at (\x, -0.8) {$\label$};
    }

    \fill[blue!20] (0, -0.15) rectangle (8, 0.15);
    \fill[blue!20] (8, -0.15) rectangle (16, 0.15);
    \fill[blue!20] (24, -0.15) rectangle (32, 0.15);

    \node at (4, 0.7) {$I_{n,1}$};
    \node at (12, 0.7) {$I_{n,2}$};
    \node at (20, 0.7) {$\cdots$};
    \node at (28, 0.7) {$I_{n,c_{n}+1}$};

    \node at (4, -0.45) {$\lfloor n^{a} \rfloor$};
    \node at (12, -0.45) {$\lfloor n^{a} \rfloor$};
    \node at (28, -0.45) {$\le 3\lfloor n^{a} \rfloor$};

    \draw[thick, decoration={brace, amplitude=2.7pt, mirror}, decorate] (0, -1.1) -- (32, -1.1);
    \node at (20, -1.5) {$I_n = [0, n-1]$};

    \draw[thin] (8, -0.15) -- (8, 0.15);
    \draw[thin] (16, -0.15) -- (16, 0.15);
    \draw[thin] (24, -0.15) -- (24, 0.15);
    \draw[thin] (32, -0.15) -- (32, 0.15);

\end{tikzpicture}

\begin{lemma}\label{twoexpdiff}
    \[ \exp\Big(-\frac{t^2}{2}\sum_{1\le i\le c_n+1} \mathbb{\widetilde{E}} \Big( \frac{X_{n,i}}{\sqrt{n}}\Big)^2\Big)-\exp\Big(- \frac{t^2}{2} \mathbb{\widetilde{E}} \Big(\frac{\sum_{1\le i\le c_n+1}X_{n,i}}{\sqrt{n}}\Big)^2 \Big)=t^2 k^{\xi-1}O(n^{\max\{3-\xi, 0\}-a}\log n)\]where the constant in $O(\cdot)$ does not depend on $t,k,n$. 
\end{lemma}
\begin{proof} By Lemma \ref{momentbound},
  \begin{align*}
     \mathbb{\widetilde{E}}(\sum_{i\le c_n+1}X_{n,i})^2-\sum_{i \le c_n+1} \mathbb{\widetilde{E}} X_{n,i}^2&=2\sum_{i \le c_n+1}\mathbb{\widetilde{E}} X_{n,i}\sum_{i<j\le c_n+1}X_{n,j}\\
     &=\sum_{i \le c_n+1} k^{\xi-1}O(n^{\max\{3-\xi, 0\}}\log n)=k^{\xi-1}O(n^{1-a+\max\{3-\xi, 0\}}\log n),
 \end{align*} where the constant in $O(\cdot)$ does not depend on $t,k,n$. The proof is concluded by noting that $\exp(\cdot)$ is a Lipschitz function on $(-\infty,0]$.
\end{proof}

With this estimate, we now estimate for $\int \exp\Big(it\frac{\sum_{0\le i\le n-1}\psi_k \circ \widetilde{F}^i}{\sqrt{n}}\Big)d\mu_{\widetilde{\Delta}}$.
 \begin{lemma}\label{preclt}
\begin{align*}
    &\Big|\mathbb{\widetilde{E}} \exp\Big(it\frac{\sum_{0\le i\le n-1}\psi_k\circ \widetilde{F}^i}{\sqrt{n}}\Big)-\exp \Big(-\frac{t^2}{2}\mathbb{\widetilde{E}} \Big(\frac{\sum_{0\le i\le n-1}\psi_k\circ \widetilde{F}^i}{\sqrt{n}}\Big)^2\Big)\Big|\\
    &\precsim  k^{\xi-1}\frac{|t|}{n^{a-0.5}}+ \frac{t^2}{2n^{\frac{a(\xi-2)}{\xi-1}}} k^{\xi-1}+  \frac{t^4k^{2\xi-2}}{n^{1-a}}+ |t|^{\frac{3\xi-4}{\xi-1}} k^{\frac{(3\xi-4)(\xi-1)}{\xi-2}} n^{(1-a)\frac{2-\xi}{2\xi-2}}+t^2 k^{\xi-1}n^{\max\{3-\xi, 0\}-a}\log n
\end{align*}where the constant in $\precsim$ does not depend on $k,n$. 
 \end{lemma}
 \begin{proof}
Using the Taylor expansion
\[e^{-x}=1-x+O(x^2),\]
\[e^{ix}=1+ix-x^2/2+ O(x^2 \min\{|x|,1\})=1+ix-x^2/2+O(|x|^{\frac{3\xi-4}{\xi-1}}),\]
Lemma \ref{twoexpdiff} and the argument of the proof of Corollary 3.1 in \cite{Sutams}, we have 
\begin{align*}
&\Big|\mathbbm{\widetilde{E}} \exp\Big(it\frac{\sum_{i\le c_n+1}X_{n,i}}{\sqrt{n}}\Big)-\exp \Big(-\frac{t^2}{2}\mathbbm{\widetilde{E}} \Big(\frac{\sum_{i\le c_n+1}X_{n,i}}{\sqrt{n}}\Big)^2\Big)\Big|\\
&=\Big|\mathbbm{\widetilde{E}} \exp\Big(it\frac{\sum_{i\le c_n+1}X_{n,i}}{\sqrt{n}}\Big)-\exp \Big(-\frac{t^2}{2}\sum_{i\le c_n+1}\mathbbm{\widetilde{E}} \Big(\frac{X_{n,i}}{\sqrt{n}}\Big)^2\Big)\Big|+t^2 k^{\xi-1}O(n^{\max\{3-\xi, 0\}-a}\log n)\\
&\le \sum_{i=1}^{c_n+1}\mathbbm{\widetilde{E}} \Big|\mathbbm{\widetilde{E}}_{a_{i}}\Big\{\Big[1+i\frac{tX_{n,i}}{\sqrt{n}}-\frac{t^2X_{n,i}^2}{2n}+O\Big(\Big|\frac{tX_{n,i}}{\sqrt{n}}\Big|^{\frac{3\xi-4}{\xi-1}}\Big)\Big]\Big|\Big\}-\Big\{1-\frac{\mathbb{\widetilde{E}}(t X_{n,i})^2}{2n}+O\Big(\Big|\frac{\mathbb{\widetilde{E}}{t^2X_{n,i}^2}}{n}\Big|^2\Big)\Big\}\Big|\\
&\quad +t^2 k^{\xi-1}O(n^{\max\{3-\xi, 0\}-a}\log n)\\
    & \le \sum_{ i=1}^{ c_n+1} \Big\{\frac{|t|\widetilde{\mathbb{E}}|\mathbb{\widetilde{E}}_{a_i}X_{n,i}|}{\sqrt{n}}+ \frac{t^2}{2n} \mathbb{\widetilde{E}}\Big|\mathbb{\widetilde{E}}_{a_i}\Big[X^2_{n,i}-\mathbb{E}X^2_{n,i}\Big]\Big|+  \frac{t^4|\mathbb{\widetilde{E}}X^2_{n,i}|^2}{n^2}+ \Big|\frac{t}{\sqrt{n}}\Big|^{\frac{3\xi-4}{\xi-1}}  \mathbb{\widetilde{E}}|X_{n,i}|^{\frac{3\xi-4}{\xi-1}}\Big\}\\
    &\quad +t^2 k^{\xi-1}O(n^{\max\{3-\xi, 0\}-a}\log n)
    \end{align*}

    Applying Lemma \ref{momentbound} we can continue to estimate
    \begin{align*}
    & \precsim \sum_{ i=1}^{ c_n+1} \Big\{k^{\xi-1}\frac{|t|}{\sqrt{n}}+ \frac{t^2}{2n} k^{\xi-1} n^{\frac{a}{\xi-1}}+  \frac{t^4k^{2\xi-2}n^{2a}}{n^2}+ \Big|\frac{t}{\sqrt{n}}\Big|^{\frac{3\xi-4}{\xi-1}} k^{\frac{(3\xi-4)(\xi-1)}{\xi-2}} n^{\frac{3a\xi-4a}{2\xi-2}}\Big\}\\
    &\quad +t^2 k^{\xi-1}n^{\max\{3-\xi, 0\}-a}\log n\\
    &\precsim k^{\xi-1}\frac{|t|}{n^{a-0.5}}+ \frac{t^2}{2n^{\frac{a(\xi-2)}{\xi-1}}} k^{\xi-1}+  \frac{t^4k^{2\xi-2}}{n^{1-a}}+ |t|^{\frac{3\xi-4}{\xi-1}} k^{\frac{(3\xi-4)(\xi-1)}{\xi-2}} n^{(1-a)\frac{2-\xi}{2\xi-2}}+t^2 k^{\xi-1}n^{\max\{3-\xi, 0\}-a}\log n
\end{align*}where the constant in $\precsim$ does not depend on $k,n$.
\end{proof}

\subsection*{Step \ref{provingclt}: prove CLT with convergence rates.}  

\begin{lemma}\label{cltwithrates}
    \begin{align*}
        &\Big|\int \exp\Big(it\frac{\sum_{0\le i\le n-1}[\mathbbm{1}_B\circ f^i-\mu_{\mathcal{M}}(B)]}{\sqrt{n}}\Big)d\mu_{\mathcal{M}}-e^{-0.5 \sigma^2 t^2}\Big|\precsim (\vert{}t\vert{}+ t^4) n^{- \frac{\min\{\xi-2, \alpha-1\}  \min\{\xi-2, 1\}  (\xi-2)^2}{32(\xi-1)^2(3\xi-4)}}
    \end{align*} where the constant in $\precsim$ does not depend on $t,n$.
\end{lemma}
\begin{proof} Combining Lemmas \ref{approximation}, \ref{variance}, \ref{preclt} and recalling that $a \in (0.5, 1)$, we obtain 
    \begin{align*}
 &\Big|\int \exp\Big(it\frac{\sum_{0\le i\le n-1}[\mathbbm{1}_B\circ f^i-\mu_{\mathcal{M}}(B)]}{\sqrt{n}}\Big)d\mu_{\mathcal{M}}-e^{-0.5 \sigma^2 t^2}\Big|\\
        &= \Big|\int \exp\Big(it\frac{\sum_{0\le i\le n-1}[\mathbbm{1}_{A}\circ F^i-\int \mathbbm{1}_{A} d\mu_{\Delta}]}{\sqrt{n}}\Big)d\mu_{\Delta}-e^{-0.5 \sigma_k^2 t^2}\Big|+O(|t| k^{-\min\{\xi-2,\alpha-1\}/2})\\
        &\quad +|e^{-0.5 \sigma_k^2 t^2}-e^{-0.5 \sigma^2 t^2}|  \\
        &\precsim k^{\xi-1}\frac{|t|}{n^{a-0.5}}+ \frac{t^2}{n^{\frac{a(\xi-2)}{\xi-1}}} k^{\xi-1}+  \frac{t^4k^{2\xi-2}}{n^{1-a}}+ |t|^{\frac{3\xi-4}{\xi-1}} k^{\frac{(3\xi-4)(\xi-1)}{\xi-2}} n^{(1-a)\frac{2-\xi}{2\xi-2}} +t^2 k^{\xi-1}n^{\max\{3-\xi, 0\}-a}\log n\\
        &\quad +t^2  k^2 n^{-\min\{\xi-2,1\}}\log n +|t| k^{-\min\{\xi-2,\alpha-1\}/2}+t^2k^{-\min\{\xi-2,\alpha-1\}}.
        \end{align*}

        There are many choices for $a,k$, e.g., $a = 1 - \frac{\min\{\xi-2,1\}}{4}, k = \lfloor n^{\frac{\min\{\xi-2,1\}(\xi-2)^2}{16(\xi-1)^2(3\xi-4)}} \rfloor$, we obtain  an upper bound \begin{align*}
            &\Big|\int \exp\Big(it\frac{\sum_{0\le i\le n-1}[\mathbbm{1}_B\circ f^i-\mu_{\mathcal{M}}(B)]}{\sqrt{n}}\Big)d\mu_{\mathcal{M}}-e^{-0.5 \sigma^2 t^2}\Big|\precsim (\vert{}t\vert{}+ t^4) n^{- \frac{\min\{\xi-2, \alpha-1\}  \min\{\xi-2, 1\}  (\xi-2)^2}{32(\xi-1)^2(3\xi-4)}},
        \end{align*}where the constant in $\precsim$ does not depend on $t,n$. 
\end{proof}

\subsection*{Step \ref{provefurtherclt}: apply Lemma \ref{cltwithrates} to prove} 
\begin{lemma}
    \begin{align*}
        &\sup_{x\in \mathbb{R}}\Big|\mu_{\mathcal{M}}(\frac{\sum_{0\le i\le n-1}[\mathbbm{1}_B\circ f^i-\mu_{\mathcal{M}}(B)]}{\sqrt{n}}<x)- \frac{1}{\sqrt{2\pi\sigma^2}}\int_{-\infty}^x e^{-\frac{1}{2\sigma^2 } t^2}dt \Big|\precsim n^{-\frac{\min\{\xi-2, \alpha-1\} \min\{\xi-2, 1\} (\xi-2)^2}{192(\xi-1)^2(3\xi-4)}} .
    \end{align*}
\end{lemma}
\begin{proof}
   Using Lemma \ref{cltwithrates}, we argue as in equation (3.13) of Chapter XVI.3  in \cite{feller}, 
    \begin{align*}
        &\sup_{x\in \mathbb{R}}\Big|\mu_{\mathcal{M}}(\frac{\sum_{0\le i\le n-1}[\mathbbm{1}_B\circ f^i-\mu_{\mathcal{M}}(B)]}{\sqrt{n}}<x)- \frac{1}{\sqrt{2\pi\sigma^2}}\int_{-\infty}^x e^{-\frac{1}{2\sigma^2 } t^2}dt \Big|\\
        &\precsim \int_{-T_n}^{T_n}(1+ |t|^3)  n^{-\frac{\min\{\xi-2, \alpha-1\} \min\{\xi-2, 1\} (\xi-2)^2}{32(\xi-1)^2(3\xi-4)}}dt+\frac{1}{T_n} \precsim  n^{-\frac{\min\{\xi-2, \alpha-1\} \min\{\xi-2, 1\} (\xi-2)^2}{192(\xi-1)^2(3\xi-4)}} 
    \end{align*}where $T_n:= n^{ \frac{\min\{\xi-2, \alpha-1\} \min\{\xi-2, 1\} (\xi-2)^2}{192(\xi-1)^2(3\xi-4)}}$. 
\end{proof}

With all steps above, we conclude the proof of Theorem \ref{clt}.

\section{Almost sure invariance principles (ASIP)}\label{nonsmoothasip}
The ASIP is a strong statistical property that provides a close approximation of the Birkhoff sum by a Brownian motion. It implies the CLT and several other limit theorems. When the observables are regular, the ASIP was established in \cite{MN1,MN2,alexey,alexey1} for a class of hyperbolic billiards (a special class of hyperbolic systems with CMZ structures). Now we prove such an ASIP for our class of irregular observables. The arguments require several new and more delicate techniques, which are summarized in subsection~\ref{scheme}. As in Theorem \ref{clt}, we suppose that $B\in  \mathcal{B}$ with $\Leb_{\mathcal{M}}(B)\in (0, \Leb_{\mathcal{M}}(\mathcal{M}))$, adopt the same notation $\sigma^2$, and let $S_n:=\sum_{0\le i\le n}[\mathbbm{1}_B\circ f^i-\mu_{\mathcal{M}}(B)]$.
\begin{theorem}[ASIP for indicator observables]\label{asip} \par
   Suppose that a hyperbolic dynamical system $(\mathcal{M},f, \mu_{\mathcal{M}})$ with a CMZ structure described in section \ref{cmz} satisfies additional assumptions $\alpha>1$, $\frac{d\mu_{\mathcal{M}}}{d\Leb_{\mathcal{M}}}\in L^{\infty}$ and $\int R_p^{\xi} d\mu_{\Lambda}< \infty$ for some $\xi>2$, then we have
   \begin{enumerate}
       \item If $\sigma^2=0$, then there is an $L^1$-function $\phi$ such that $\mathbbm{1}_B-\mu_{\mathcal{M}}(B)=\phi \circ f -\phi$.
       \item If $\sigma^2>0$, one can redefine $(S_n)_{n\ge 1}$ without changing its
distribution on a (richer) probability space  $\Omega$ on which there exists a Brownian motion $(\mathbb{B}_t)_{t \ge 0}$ defined on $\Omega$, such that \begin{align*}
   &S_n= \mathbb{B}_{\sigma^2 n}+O\!\left(
n^{\frac12-
\frac{\min\{\xi - 2, \alpha - 1\}(\xi - 2)\min\{\xi - 2, 1\}^2}{49152(\xi - 1)^2(3\xi - 4) +4 \min\{\xi - 2, \alpha - 1\}(\xi - 2)\min\{\xi - 2, 1\}^2}
}
\right) \text{ a.s. for any }n\in \mathbb{N}.
   \end{align*}
     \end{enumerate}
\end{theorem}

\begin{remark}
The ASIP implies the corollaries below. 
    \begin{enumerate}
        \item 
   Law of the iterated logarithm (LIL):
    \begin{align*}
  \limsup_{n\to \infty}\frac{S_n}{\sqrt{2n\log\log n}}=\sigma \text{ a.s.}
\end{align*}
\item Functional CLT (FCLT): $S^n \to_d  \sigma B$  on  $C[0,1]$, where $S^n$ is a random piecewise continuous function on $[0,1]$:

\[S^n_t: =\frac{S_i}{\sqrt{n}}+\frac{t-\frac{i}{n}}{n^{-1}} \cdot \frac{S_{i+1}-S_i}{\sqrt{n}}, \quad t \in [\frac{i}{n},\frac{i+1}{n}], \]
where $B$  is a standard Brownian motion.

\item Arcsine laws: for any $i=1,2,3$, $\tau_n^i \to_d \sin^2 U$ where $U\sim U(0,2\pi)$, $\tau^1_n:=n^{-1}\sum_{i\le n}\mathbbm{1}_{S_i>0}$, $\tau^2_n:=n^{-1}\min \{k \ge 0: S_k=\max_{j\le n}S_j\}$, $\tau_n^3:=n^{-1}\max\{k \le n: S_kS_n\le 0\}$.
\end{enumerate}
\end{remark}

\begin{remark}
     As in the case of Theorem \ref{clt}, the assumption $\xi>2$ is optimal. Since the observable \(\mathbbm{1}_{B}\) is not dynamically regular, the ASIP rate of convergence is not as good as that in the usual setting of the ASIP \cite{alexey1,alexey, KOREPANOVcoboundary}.
\end{remark}

\subsection{Proof scheme and conventions}\label{scheme}
Throughout this section we adopt the following convention: \(\mathbb{E}(\cdot)\) (resp. \(P(\cdot)\)) denotes the expectation (resp. probability) of a random variable whenever the underlying probability space is clear. When $\sigma^2=0$, $\mathbbm{1}_B-\mu_{\mathcal{M}}(B)=\phi \circ f-\phi$ has been proved in Theorem \ref{clt}. When $\sigma^2>0$, the proof proceeds in several steps. 
\begin{enumerate}
\item \label{blocks} Consider $S_n \circ \pi:=\sum_{0\le i\le n}[\mathbbm{1}_{\pi^{-1}B}\circ F^i-\mu_{\Delta}(\pi^{-1}B)]$ on the hyperbolic Young tower $\Delta$, introduce a block division of $\mathbb{N}_0=\bigcup_n I_n$. Let $X_n:=\sum_{i \in I_n}[\mathbbm{1}_{\pi^{-1}B}\circ F^i-\mu_{\Delta}(\pi^{-1}B)]$.
\item \label{approximateasip} Approximate $X_n$ by a suitable $\sigma(\bigcup_{i\ge 0}\mathcal{Q}_i)$-measurable random variable $X'_n$, where $(X'_n)_n$ is a nonstationary process. This step is new and crucial for the ASIP when observables are not dynamically smooth; the classical approach uses a coboundary argument \cite{MN1,MN2,KOREPANOVcoboundary} that requires smoothness.
\item \label{iidapproximate} Approximate each \(X'_n\) by a sum of independent random variables. To do this we split each block \(I_n\) further into many subblocks and the independent random variables are associated with these subblocks. This refined blocking is different from the usual methods \cite{BP,MN2,Sutams}; the reason is explained in Remark \ref{cannotmax} and shows why the techniques of \cite{BP,MN2,Sutams} (based on regular observables) cannot be applied directly.
\item \label{gaussianapproximate} Approximate the independent random variables by suitable Gaussian variables. The arguments here are inspired by the proofs in \cite{Sudcds}, which work effectively for our one‑dimensional nonstationary process and our refined subblocks.
\item \label{mx} Derive a double maximal inequality for the nonstationary process $X_n'$. This estimate is required by our unusual refined blocking arguments. 
\item \label{asipontower} Combine the approximations to obtain an ASIP for \(S_n\circ\pi\) on the tower \(\Delta\) with a Brownian motion \((B_t)_{t\ge0}\). Finally, project this ASIP down to the original system on \(\mathcal{M}\).
\end{enumerate}

\subsection{Some useful probabilistic lemmas}
In this subsection we present several standard probabilistic results.
\begin{lemma}\label{variancegrow}Let $\psi:=\mathbbm{1}_{\pi^{-1}B}-\mu_{\Delta}(\pi^{-1}B)$. There is $C>0$ such that for any $m,q \ge 0, n,p\ge 1$, \[\mathbb{E}\Big(\sum_{k=m}^{ n+m-1} \psi \circ F^k \Big)^2 \le Cn,\]
  \[\Big|\mathbb{E}  \Big[(\sum_{k=m}^{n+m-1} \psi \circ F^k)  (\sum_{k=q+m+n}^{q+n+m+p-1}\psi \circ F^k)\Big]\Big|\le C\max(n,p)^{\max(2-\min\{\alpha, \xi-1\}, 0)}\log \max(n,p). \] 
\end{lemma}
\begin{proof}
    The proof makes use of $\pi_{*}\mu_{\Delta}=\mu_{\mathcal{M}}$ and Theorem \ref{mixingrate}. The arguments are the same as those in Lemmas 4.2 and 4.6 of \cite{Sutams}.
\end{proof}

\begin{lemma}[See \cite{serfling}, Corollary A3.1 and Theorem B]\label{usefulprob} \par
    Given random variables $x_i$ with zero means and finite variances, suppose that there are $\varepsilon, C>0$ such that $\sup_m\mathbb{E}(\sum_{k=m}^{ n+m-1} x_k )^2 \le C n$ (resp. $\sup_m\mathbb{E}|\sum_{k=m}^{ n+m-1} x_k |^{2+\varepsilon} \le C n^{1+0.5\varepsilon}$), then we have $\sup_m\mathbb{E}\max_{u\le n}(\sum_{k=m}^{ u+m-1} x_k )^2 \le C n\log_2^2(2n)$ (resp. $\sup_m\mathbb{E}\max_{u\le n}|\sum_{k=m}^{ u+m-1} x_k |^{2+\varepsilon} \precsim n^{1+0.5\varepsilon}$), where the constant in  $\precsim$ depends only on $\sup_{i}|x_i|, C$ and $\varepsilon$.

    Together with Lemma \ref{variancegrow} we obtain $\sup_m\mathbb{E}\max_{u\le n}(\sum_{k=m}^{ u+m-1} \psi \circ F^k )^2 \le Cn\log_2^2(2n) $.
\end{lemma}

\begin{lemma}[See \cite{BP}, Theorem 2]\label{usefulprob1} \par
     Let \(\{Y_k, k \geq 1\}\) be a sequence of random variables with values in \(\mathbb{R}\) and let \(\{\mathcal{E}_k, k \geq 1\}\) be a family of \(\sigma\)-fields where \(Y_k\) is \(\mathcal{E}_k\)-measurable. Suppose that for some \(\phi_k > 0\) we have $|P(AB) - P(A)P(B)| \leq \phi_k P(A)$ for all $k\ge 1$, \(A \in \bigvee_{j < k} \mathcal{E}_j\) and \(B \in \mathcal{E}_k\). Then without changing its distribution we can redefine the sequence \(\{Y_k, k \geq 1\}\) on a richer probability space together with a sequence \(\{Y'_k, k \geq 1\}\) of independent random variables such that \(Y'_k=_dY_k\) and $P\{|Y_k-Y'_k| \geq 6\phi_k\} \leq 6\phi_k$ for any $k\ge 1$.
\end{lemma}

 \subsection{Proof of Theorem \ref{asip} when $\sigma^2>0$}
\subsubsection*{Step \ref{blocks}: introduction of blocks and notation}
Let $a\ge 1$ (to be determined), and introduce blocks $\{I_n, n\ge 1\}$ in $\mathbb{N}_0$ as follows: $I_n:=[a_{n-1},a_n)\bigcap \mathbb{N}_0$ where $a_0:=0, a_n: =\sum_{i=1}^{ n} |I_i|$ and the cardinality $|I_n|=\lfloor n^a \rfloor$. Hence $a_n \approx n^{a+1}$.

 Let $\varepsilon \in (0,1/2)$ (to be determined), $k_n:= \lfloor |I_n|^{\varepsilon} \rfloor$. Define $I_n'=[a_{n-1}, a_n-k_n)\bigcap I_n$, $J_{n+1}=[a_n-k_{n+1}, a_{n})\bigcap I_n$. Hence the cardinality  $|I_n'|\approx |I_n| \approx n^a, |J_{n}|\approx n^{a\varepsilon }$.

 Let $ c \in (1/2,1)$ (to be determined), partition the block $I_n'=\bigcup_{i \le c_n}I_{n,i}$ such that $I_{n,i}$ are disjoint subblocks of consecutive integers in $I_n'$,  the cardinality $|I_{n,i}| := \lfloor |I_n'|^c \rfloor $ for any $i \le c_n-1$  and $ \lfloor |I_n'|^c \rfloor\le |I_{n,c_n}| \le 2|I_n'|^c$. Hence $c_n \approx |I_n'|^{1-c}\approx n^{a(1-c)}, |I_{n,i}|\approx |I_{n}|^c \approx n^{ac}$.

 Let $\varepsilon_1\in (\varepsilon/c,1)$ (to be determined), partition the block
 $I_{n,i} = I_{n,i}' \cup J_{n,i}$ such that $I_{n,i}', J_{n,i}$ are disjoint subblocks of consecutive integers and the cardinality $|J_{n,i}|:= 4 \lfloor  |I_{n,i}|^{\varepsilon_1} \rfloor$. Hence  $|J_{n,i}| \approx 4|I_n|^{c\varepsilon_1}\approx n^{ac\varepsilon_1}$, $|I_{n,i}'|\approx|I_{n,i}| \approx n^{ac}$. The arrangement of these blocks is illustrated below.

\begin{tikzpicture}[xscale=0.8, yscale=1.2]
    \draw[thick, ->] (0,0) -- (16,0) node[right] {$\mathbb{N}$};

    \draw[fill=gray!20] (1, -0.2) rectangle (5, 0.2);
    \node at (3, 0.5) {$I_n$};
    
    \draw[fill=gray!20] (6, -0.2) rectangle (10, 0.2);
    \node at (8, 0.5) {$I_{n+1}$};

    \draw[fill=gray!20] (11, -0.2) rectangle (15, 0.2);
    \node at (13, 0.5) {$I_{n+2}$};

    \draw[fill=blue!10] (1.1, -0.1) rectangle (2.5, 0.1);
    \node at (1.95, -0.4) {\footnotesize $I_{n,1}$};
    \node at (2.9, 0) {$\dots$};
    \draw[fill=blue!10] (3.15, -0.1) rectangle (4.5, 0.1);
    \node at (3.75, -0.4) {\footnotesize $I_{n,i}$};
    \draw[fill=red!10] (4.5, -0.1) rectangle (4.9, 0.1);
    \node at (4.75, -0.4) {\tiny $J_{n+1}$};

    \draw[decorate, decoration={brace, amplitude=5pt, mirror}] (3.15, -0.6) -- (4.5, -0.6);
    \node at (3.75, -1.0) {\footnotesize $I_{n,i} = I_{n,i}' \cup J_{n,i}$};
    \draw[fill=green!10] (3.15, -1.2) rectangle (3.9, -1.4);
    \node at (3.45, -1.6) {\tiny $I_{n,i}'$};
    \draw[fill=orange!10] (3.95, -1.2) rectangle (4.45, -1.4);
    \node at (4.2, -1.6) {\tiny $J_{n,i}$};

    \draw[fill=blue!10] (6.1, -0.1) rectangle (7.5, 0.1);
    \node at (6.95, -0.4) {\footnotesize $I_{n+1,1}$};
    \node at (7.9, 0) {$\dots$};
    \draw[fill=blue!10] (8.15, -0.1) rectangle (9.5, 0.1);
    \node at (8.75, -0.4) {\footnotesize $I_{n+1,i}$};
    \draw[fill=red!10] (9.5, -0.1) rectangle (9.9, 0.1);
    \node at (9.75, -0.4) {\tiny $J_{n+2}$};

    \draw[decorate, decoration={brace, amplitude=5pt, mirror}] (8.15, -0.6) -- (9.5, -0.6);
    \node at (8.75, -1.0) {\footnotesize $I_{n+1,i} = I_{n+1,i}' \cup J_{n+1,i}$};
    \draw[fill=green!10] (8.15, -1.2) rectangle (8.9, -1.4);
    \node at (8.45, -1.6) {\tiny $I_{n+1,i}'$};
    \draw[fill=orange!10] (8.95, -1.2) rectangle (9.45, -1.4);
    \node at (9.4, -1.6) {\tiny $J_{n+1,i}$};

    \draw[fill=blue!10] (11.1, -0.1) rectangle (12.5, 0.1);
    \node at (11.95, -0.4) {\footnotesize $I_{n+2,1}$};
    \node at (12.9, 0) {$\dots$};
    \draw[fill=blue!10] (13.15, -0.1) rectangle (14.5, 0.1);
    \node at (13.75, -0.4) {\footnotesize $I_{n+2,i}$};
    \draw[fill=red!10] (14.5, -0.1) rectangle (14.9, 0.1);
    \node at (14.75, -0.4) {\tiny $J_{n+3}$};

    \draw[decorate, decoration={brace, amplitude=5pt, mirror}] (13.15, -0.6) -- (14.5, -0.6);
    \node at (13.75, -1.0) {\footnotesize $I_{n+2,i} = I_{n+2,i}' \cup J_{n+2,i}$};
    \draw[fill=green!10] (13.15, -1.2) rectangle (13.9, -1.4);
    \node at (13.45, -1.6) {\tiny $I_{n+2,i}'$};
    \draw[fill=orange!10] (13.95, -1.2) rectangle (14.45, -1.4);
    \node at (14.4, -1.6) {\tiny $J_{n+2,i}$};

\end{tikzpicture}

\subsubsection*{Step \ref{approximateasip}: approximate $X_n:=\sum_{i\in I_n}[\mathbbm{1}_{\pi^{-1}B}\circ F^i-\mu_{\Delta}(\pi^{-1}B)]$}

Observe that $\min I_n>k_n$ when $n\gg 1$, then $X_n=\sum_{j \in I_n} [ \mathbbm{1}_{F^{-k_n} \pi^{-1} B} \circ F^{j-k_n} - \mu_\Delta(F^{-k_n} \pi^{-1} B)]$. Cover $F^{-k_n}\pi^{-1}B$ with a family of $\xi_{2k_n,i} \in \mathcal{Q}_{2k_n}$. Let $\overline{A}_{2k_n}:=\bigcup_i \xi_{2k_n,i}$. Now we approximate $X_n$ by $X_n':=\sum_{i \in I_n}[\mathbbm{1}_{\overline{A}_{2k_n}}\circ F^{i-k_n}-\mu_{\Delta}(\overline{A}_{2k_n})]$ which is a $\sigma(\bigcup_{i\ge 0} \mathcal{Q}_{i})$-measurable function.

\begin{lemma}\label{asip1} When $a\varepsilon \min\{\xi-2,\alpha-1\}>4$ and $\varepsilon \min\{\xi-2,\alpha-1\}<2$, then for any $N \in \mathbb{N}$ and $p \in [a_{N},a_{N+1})$, we have
 $$\sum_{n \le N} X_n - \sum_{n \le N} \sum_{j \in I_n} [\mathbbm{1}_{\overline{A}_{2k_n}} \circ F^{j-k_n} - \mu_\Delta(\overline{A}_{2k_n})]\precsim p^{0.5-[\varepsilon \min\{\xi-2,\alpha-1\}/4 - 1/(a+1)]} \text{ a.s.}$$
 \begin{align*}
     &\sum_{i\le p}[\mathbbm{1}_{\pi^{-1}B}\circ F^{i}-\mu_{\Delta}(\pi^{-1}B)]- \sum_{n\le N+1}\sum_{i \in I_n\bigcap [0,p]}[\mathbbm{1}_{\overline{A}_{2k_n}}\circ F^{i-k_n}-\mu_{\Delta}(\overline{A}_{2k_n})]\\
    &\precsim p^{0.5-[\varepsilon \min\{\xi-2,\alpha-1\}/4 - 1/(a+1)]} \text{ a.s.}
 \end{align*}
\end{lemma}
\begin{proof} By the invariance of $F$ and Lemma \ref{approximation}, we have $\mathbb{E}|X_n-X_n'|^2=O(|I_n| k_n^{-\min\{\alpha-1, \xi-2\}})=O(|I_n||I_n|^{-\varepsilon \min\{\xi-2,\alpha-1\}})$. In other words,
\begin{align*}
    &\mathbb{E} \left[ \frac{\sum_{j \in I_n} [\mathbbm{1}_{\pi^{-1}B} \circ F^j - \mu_\Delta(\pi^{-1}B)] - \sum_{j \in I_n} [\mathbbm{1}_{\overline{A}_{2k_n}} \circ F^{j-k_n} - \mu_\Delta(\overline{A}_{2k_n})]}{|I_n|^{0.5-\varepsilon \min\{\xi-2,\alpha-1\}/4}} \right]^2 \\
    &\lesssim \frac{|I_n|^{1-\varepsilon \min\{\alpha-1, \xi-2\}}}{|I_n|^{1-\varepsilon \min\{\alpha-1, \xi-2\}/2}} \precsim \frac{1}{|I_n|^{\varepsilon [\min\{\alpha-1, \xi-2\}/2]}} \precsim \frac{1}{n^{a\varepsilon [\min\{\alpha-1, \xi-2\}/2]}}.
\end{align*}

By the choice of $a$ and the Borel-Cantelli lemma, we have $\sum_{j \in I_n} [\mathbbm{1}_{\pi^{-1}B} \circ F^j - \mu_\Delta(\pi^{-1}B)] - \sum_{j \in I_n} [\mathbbm{1}_{\overline{A}_{2k_n}} \circ F^{j-k_n} - \mu_\Delta(\overline{A}_{2k_n})]= O(|I_n|^{0.5-\varepsilon \min\{\xi-2,\alpha-1\}/4}) \text{ a.s.}$ Therefore, \begin{align*}
    &\sum_{n \le N} \sum_{j \in I_n} [\mathbbm{1}_{\pi^{-1}B} \circ F^j - \mu_\Delta(\pi^{-1}B)] - \sum_{n \le N} \sum_{j \in I_n} [\mathbbm{1}_{\overline{A}_{2k_n}} \circ F^{j-k_n} - \mu_\Delta(\overline{A}_{2k_n})]\\
    &= \sum_{n \le N} O(|I_n|^{0.5-\varepsilon \min\{\xi-2,\alpha-1\}/4}) = O(N^{a(0.5-\varepsilon \min\{\xi-2,\alpha-1\}/4) + 1})\precsim p^{0.5-[\varepsilon \min\{\xi-2,\alpha-1\}/4 - 1/(a+1)]} \text{ a.s.}
\end{align*}

Now we pass this estimate to $p \in [a_N, a_{N+1}]$. Fix $n$ and consider the stationary random variables $ G \circ F^{t-k_n} := \mathbbm{1}_{F^{-k_n} \pi^{-1} B} \circ F^{t-k_n}- \mu_\Delta(F^{-k_n} \pi^{-1} B) - \mathbbm{1}_{\overline{A}_{2k_n}} \circ F^{t-k_n} + \mu_\Delta(\overline{A}_{2k_n})$ for any $t \ge k_n$. By the invariance of $F$, Lemma \ref{approximation} and Lemma \ref{usefulprob}, 
\[\sup_{u \ge k_n} \mathbb{E} \Big| \sum_{u \le t \le u+m} G \circ F^{t-k_n} \Big|^2 \lesssim m k_n^{-\min\{\alpha-1, \xi-2\}}\lesssim m |I_n|^{-\varepsilon \min\{\alpha-1, \xi-2\}},\]  
\[\sup_{w \ge k_n} \mathbb{E} \sup_{u \le m} \Big| \sum_{w \le t \le w+u} G \circ F^{t-k_n} \Big|^2  \lesssim  k_n^{-\min\{\alpha-1, \xi-2\}} m  \log_2^2 (2m) \lesssim  |I_n|^{-\varepsilon \min\{\alpha-1, \xi-2\}} m  \log_2^2 (2m),\]  where the constants in $\precsim$ do not depend on $n,m$. Therefore
\[\mathbb{E} \sup_p \Big| \sum_{t \in I_n \cap [0,p]} \left[ \mathbbm{1}_{\pi^{-1}B} \circ F^t - \mu_\Delta(\pi^{-1}B) - \mathbbm{1}_{\overline{A}_{2k_n}} \circ F^{t-k_n} + \mu_\Delta(\overline{A}_{2k_n}) \right] \Big|^2 \precsim  \frac{|I_n| \log^2 |I_n| }{|I_n|^{\varepsilon \min\{\alpha-1, \xi-2\}}}\]
where the constant in $\precsim$ does not depend on $n$. In other words, \begin{align*}
    &\mathbb{E} \left[ \frac{\sup_p|\sum_{j \in I_n\bigcap [0,p]} [\mathbbm{1}_{\pi^{-1}B} \circ F^j - \mu_\Delta(\pi^{-1}B)] - \sum_{j \in I_n} [\mathbbm{1}_{\overline{A}_{2k_n}} \circ F^{j-k_n} - \mu_\Delta(\overline{A}_{2k_n})]|}{|I_n|^{0.5-\varepsilon \min\{\xi-2,\alpha-1\}/4}} \right]^2 \\
    &\lesssim \frac{|I_n|^{1-\varepsilon \min\{\alpha-1, \xi-2\}}\log^2 |I_n|}{|I_n|^{1-\varepsilon \min\{\alpha-1, \xi-2\}/2}} \precsim \frac{\log^2 |I_n|}{|I_n|^{\varepsilon [\min\{\alpha-1, \xi-2\}/2]}}\precsim \frac{\log^2 n}{n^{a\varepsilon [\min\{\alpha-1, \xi-2\}/2]}}.
\end{align*}

By the choice of $a$ and the Borel-Cantelli lemma, we have \[ \sup_p \Big| \sum_{t \in I_n \cap [0,p]} \left[ \mathbbm{1}_{\pi^{-1}B} \circ F^t - \mu_\Delta(\pi^{-1}B) - \mathbbm{1}_{\overline{A}_{2k_n}} \circ F^{t-k_n} + \mu_\Delta(\overline{A}_{2k_n}) \right] \Big| \lesssim |I_n|^{0.5-\varepsilon \min\{\xi-2,\alpha-1\}/4}\text{ a.s.} \]

Then for any $p \in [a_N,a_{N+1}]$, 
\begin{align*}
    &\sum_{i\le p}[\mathbbm{1}_{\pi^{-1}B}\circ F^{i}-\mu_{\Delta}(\pi^{-1}B)]- \sum_{n\le N+1}\sum_{i \in I_n\bigcap [0,p]}[\mathbbm{1}_{\overline{A}_{2k_n}}\circ F^{i-k_n}-\mu_{\Delta}(\overline{A}_{2k_n})]\\
    &\precsim \sum_{n\le N+1}|I_n|^{0.5-\varepsilon \min\{\xi-2,\alpha-1\}/4} \precsim p^{0.5-[\varepsilon \min\{\xi-2,\alpha-1\}/4 - 1/(a+1)]} \text{ a.s.}
\end{align*} which concludes the proof.
\end{proof}

\subsubsection*{Step \ref{iidapproximate}: approximate $X'_n:= \sum_{i \in I_n}[\mathbbm{1}_{\overline{A}_{2k_n}}\circ F^{i-k_n}-\mu_{\Delta}(\overline{A}_{2k_n})]$ by  independent random variables}

Note that \begin{align*}
    &\sum_{t \in I_n} \left[ \mathbbm{1}_{\overline{A}_{2k_n}} \circ F^{t-k_n} - \mu_\Delta(\overline{A}_{2k_n}) \right]= \sum_{t \in I_n'} \left[ \mathbbm{1}_{\overline{A}_{2k_n}} \circ F^t - \mu_\Delta(\overline{A}_{2k_n}) \right] + \sum_{t \in J_n} \left[ \mathbbm{1}_{\overline{A}_{2k_n}} \circ F^t - \mu_\Delta(\overline{A}_{2k_n}) \right]\\
    & = \sum_{i \le c_n}\sum_{t \in I_{n,i}'} \left[ \mathbbm{1}_{\overline{A}_{2k_n}} \circ F^t - \mu_\Delta(\overline{A}_{2k_n}) \right]+ \sum_{i \le c_n}\sum_{t \in J_{n,i}} \left[ \mathbbm{1}_{\overline{A}_{2k_n}} \circ F^t - \mu_\Delta(\overline{A}_{2k_n}) \right] \\
    &\quad + \sum_{t \in J_n} \left[ \mathbbm{1}_{\overline{A}_{2k_n}} \circ F^t - \mu_\Delta(\overline{A}_{2k_n}) \right].
\end{align*} Now we show below that $J_n, J_{n,i}$ parts can be neglected.
\begin{lemma}\label{ruleoutsmall}  When $a\varepsilon\xi > 1,  \varepsilon\xi < 1/8$, for any $N \in \mathbb{N}$ and $p \in [a_{N},a_{N+1})$, \[\sum_{n \le N+1} \sum_{j \in J_n \cap [0,p]} \left[ \mathbbm{1}_{\overline{A}_{2k_n}} \circ F^j - \mu_{\Delta}(\overline{A}_{2k_n}) \right] \precsim p^{0.25} \text{ a.s.}\] 

When $ a[\varepsilon(\xi-1) + c\varepsilon_1+c-1] > 1$,  $ 1-c < 1/32, \varepsilon(\xi-1) < 1/32,  \varepsilon < c\varepsilon_1 < 1/32$, 
 \begin{align*}&\sum_{n\le N+1} \sum_{i=1}^{c_n} \sum_{j \in J_{n,i} \cap [0,p]} \left[ \mathbbm{1}_{\overline{A}_{2k_n}} \circ F^j - \mu_\Delta(\overline{A}_{2k_n}) \right]=O(p^{0.3}) \text{ a.s.}
\end{align*}
   
\end{lemma}
\begin{proof}
   Fix $\overline{A}_{2k_n}$ and apply Lemma \ref{momentbound} to the stationary process $\mathbbm{1}_{\overline{A}_{2k_n}} \circ F^t - \mu_\Delta(\overline{A}_{2k_n})$, we obtain $\sup_u\mathbb{E}  \Big| \sum_{t \in [u,u+m] } \left[ \mathbbm{1}_{\overline{A}_{2k_n}} \circ F^t - \mu_\Delta(\overline{A}_{2k_n}) \right] \Big|^2 \lesssim k_n^{\xi-1} m$ where the constant in $\precsim$ does not depend on any $n,m$. Applying Lemma \ref{usefulprob} to the bound above and taking $m = |J_n|,|J_{n,i}|$ respectively, we obtain
 \[\mathbb{E} \sup_u \Big| \sum_{t \in J_n \cap [0,u]} \left[ \mathbbm{1}_{\overline{A}_{2k_n}} \circ F^t - \mu_\Delta(\overline{A}_{2k_n}) \right] \Big|^2 \lesssim k_n^{\xi-1} |J_n| \log^2 |J_n|\lesssim |I_n|^{\varepsilon(\xi-1)} |J_n|\log^2 n,\]
  \[\mathbb{E} \sup_u \Big| \sum_{j \in J_{n,i} \cap [0,u]} \left[ \mathbbm{1}_{\overline{A}_{2k_n}} \circ F^j - \mu_\Delta(\overline{A}_{2k_n}) \right] \Big|^2 \lesssim k_n^{\xi-1} |J_{n,i}| \log^2 |J_{n,i}| \lesssim |I_n|^{\varepsilon(\xi-1)} |I_n|^{c\varepsilon_1}\log^2 n.\]

In other words, \[\mathbb{E} \sup_m \Big| \frac{\sum_{t \in J_n \cap [0,m]} \left[ \mathbbm{1}_{\overline{A}_{2k_n}} \circ F^t - \mu_\Delta(\overline{A}_{2k_n}) \right]}{|I_n|^{\varepsilon \xi}} \Big|^2 \lesssim |I_n|^{-\varepsilon\xi} \log^2 n,\]
 \[\mathbb{E} \sum_i \sup_m \Big| \frac{\sum_{t \in J_{n,i} \cap [0,m]} \left[ \mathbbm{1}_{\overline{A}_{2k_n}} \circ F^t - \mu_\Delta(\overline{A}_{2k_n}) \right]}{|I_n|^{\varepsilon(\xi-1)+c\varepsilon_1} } \Big|^2 \lesssim c_n |I_n|^{-\varepsilon(\xi-1)-c\varepsilon_1}  \log^2 n.\]

  By the Borel-Cantelli lemma, when $a\varepsilon\xi > 1$, \[\sup_m \Big| \sum_{t \in J_n \cap [0,m]} \left[ \mathbbm{1}_{\overline{A}_{2k_n}} \circ F^t - \mu_\Delta(\overline{A}_{2k_n}) \right] \Big|= O(|I_n|^{\varepsilon\xi}) \text{ a.s.},\]
and when $ a[\varepsilon(\xi-1) + c\varepsilon_1+c-1] > 1$, 
\[\sup_i\sup_m \Big| \sum_{j \in J_{n,i} \cap [0,m]} \left[ \mathbbm{1}_{\overline{A}_{2k_n}} \circ F^j - \mu_\Delta(\overline{A}_{2k_n}) \right] \Big| = O( |I_n|^{\varepsilon(\xi-1)+c\varepsilon_1} ) \text{ a.s.}\]

Therefore, when  $\varepsilon\xi < 1/8, a > 8$,
 for any $ p \in [a_N, a_{N+1}]$, \[\sum_{n \le N+1} \sum_{j \in J_n \cap [0,p]} \left[ \mathbbm{1}_{\overline{A}_{2k_n}} \circ F^j - \mu_{\Delta}(\overline{A}_{2k_n}) \right]\lesssim \sum_{n \le N+1} O(|I_n|^{\varepsilon\xi}) \lesssim N^{a\varepsilon\xi+1}  \precsim p^{0.25}  \text{ a.s.},\] and when $ 1-c <1/32, \varepsilon(\xi-1) < 1/32,  \varepsilon < c\varepsilon_1 <1/32$, for any $ p \in [a_N, a_{N+1}]$,
 \begin{align*}&\sum_{n\le N+1} \sum_{i=1}^{c_n} \sum_{j \in J_{n,i} \cap [0,p]} \left[ \mathbbm{1}_{\overline{A}_{2k_n}} \circ F^j - \mu_\Delta(\overline{A}_{2k_n}) \right] \\
&= O( \sum_{n\le N+1} |I_n|^{1-c + \varepsilon(\xi-1) + c\varepsilon_1} )= O( \sum_{n\le N+1} |I_n|^{0.1})=O(p^{0.3}) \text{ a.s.}
\end{align*} which concludes the proof.
\end{proof}
 
Now we approximate $Y_{n,i}:=\left[ \sum_{j \in I_{n,i}'} \mathbbm{1}_{\overline{A}_{2k_n}} \circ F^j - \mu_\Delta(\overline{A}_{2k_n}) \right] $ by independent random variables. For convenience, we introduce a partial order for $\{(n,i): n \in \mathbb{N}, 1\le i \le c_n\}$: $(n_1,i_1)< (n_2,i_2)$ if ``$n_1< n_2$" or ``$n_1=n_2$ and $i_1<i_2$". 

 \begin{lemma}\label{approximateindependnt}
     When $a[c\varepsilon_1(\xi-1)+c-1] > 1, 1-c < \frac{\varepsilon_1}{1+\varepsilon_1} < \frac{1}{8}, a>9$, without changing its distribution we can redefine the sequence \(\{Y_{n,i}\}_{(n,i)}\) on a richer probability space together with a sequence \(\{\overline{Y}_{n,i}\}_{(n,i)}\) of independent random variables such that \(Y_{n,i}=_d \overline{Y}_{n,i}\) and $\sum_{n \le N} \sum_{i \le c_n} |\overline{Y}_{n,i} - Y_{n,i}|\le a_N^{0.25}$ a.s.
 \end{lemma}
 
 \begin{proof}
Note that $ Y_{N,I}$ is a discrete-valued random variable, take one of its values $a_{N,I}$, $\{Y_{N,I} = a_{N,I}\}$ is a union of some elements in $ F^{-\sum_{k < N} |I_k| - \sum_{j \le I-1} |I_{N,j}| - 1} \bigvee_{j=0}^{|I_{N,I}'|-1} F^{-j} Q_{2k_N}$  where $Q_{2k_N} = \bigvee_{j=0}^{2k_N-1} F^{-j} Q_1$. Observe that the indices appearing in the sum for \(Y_{N,I}\) and those for \(Y_{N,I+1}\) are separated by a gap of at least \(0.5|J_{N,I}| = 2\lfloor |I_{N,I}|^{\varepsilon_1}\rfloor \approx 2|I_N'|^{c\varepsilon_1}\); similarly, the indices for \(Y_{N+1,1}\) and \(Y_{N,c_N}\) are separated by at least \(0.5|J_{N,c_N}| \approx 2|I_N'|^{c\varepsilon_1}\).

 If $I>1$, then $\bigcap_{(n,i)< (N,I)} \{ Y_{n,i} = a_{n,i} \}$ is a union of some elements in
 $$\bigvee_{n=1}^{N-1}  \bigvee_{i \le c_n} \bigvee_{j \in I_{n,i}'} F^{-j} Q_{2k_n} \bigvee \bigvee_{i \le I-1} \bigvee_{j \in I_{N,i}'} F^{-j} Q_{2k_N}.$$ If $I=1$, then $\bigcap_{(n,i)< (N,I)} \{ Y_{n,i} = a_{n,i} \}$ is a union of some elements in
 $\bigvee_{n=1}^{N-1} \bigvee_{i \le c_n} \bigvee_{j \in I_{n,i}'} F^{-j} Q_{2k_n}$. Let $D$ be such an element and take any value $a_{N,I}$ of $Y_{N,I}$. According to Lemma \ref{youngmixing}, there is a constant $C>0$ (depending on $F$ only) such that \[\left|\mu_\Delta(D \bigcap \{Y_{N,I} = a_{N,I}\}) - \mu_\Delta(D) \mu_\Delta(Y_{N,I} = a_{N,I})  \right|\le C\frac{1}{|I_N'|^{c\varepsilon_1(\xi-1)}}\mu_\Delta(D).\] 

By Lemma \ref{usefulprob1}, without changing its distribution we can redefine the sequence \(\{Y_{N,I}\}_{(N,I)}\) on a richer probability space together with a sequence \(\{\overline{Y}_{N,I}\}_{(N,I)}\) of independent random variables such that \(Y_{N,I}=_d \overline{Y}_{N,I}\) and $P\{|Y_{N,I}-\overline{Y}_{N,I}| \geq 6C\frac{1}{|I_N'|^{c\varepsilon_1(\xi-1)}}\} \leq 6C\frac{1}{|I_N'|^{c\varepsilon_1(\xi-1)}}$ for any $(N,I)$. Therefore, for any $N\ge 1$, \[P\{\sum_{I\le c_N}|Y_{N,I}-\overline{Y}_{N,I}| \geq \sum_{I\le c_N}6C\frac{1}{|I_N'|^{c\varepsilon_1(\xi-1)}}\} \leq \sum_{I\le c_N}6C\frac{1}{|I_N'|^{c\varepsilon_1(\xi-1)}}\precsim \frac{|I_N|^{1-c}}{|I_N|^{c\varepsilon_1(\xi-1)}}.\] 

When $a[c\varepsilon_1(\xi-1)+c-1] > 1, 1-c < \frac{\varepsilon_1}{1+\varepsilon_1} < \frac{1}{8}$, then by the Borel-Cantelli Lemma, 

\[\sum_{I\le c_N} |\overline{Y}_{N,I} - Y_{N,I}| \precsim \frac{|I_N|^{1-c}}{|I_N|^{c\varepsilon_1(\xi-1)}} \precsim |I_N|^{1-c} \text{ a.s.}\]
holds for any $N\ge 1$. Therefore, when $a>9$, $\sum_{n \le N} \sum_{i \le c_n} |\overline{Y}_{n,i} - Y_{n,i}| \le N^{a(1-c)+1} \le a_N^{0.25} \text{ a.s.}$
\end{proof}

\subsubsection*{Step \ref{gaussianapproximate}: approximate $\overline{Y}_{n,i}$ by Gaussian variables using Skorokhod embedding}
In addition to the partial order, we introduce successor and predecessor operations for $\{(n,i): n \in \mathbb{N}, 1\le i \le c_n\}$: \begin{align*}
    (n,i)+1:=\begin{cases}
 (n,i+1),      &\text{if } i<c_n\\
(n+1,1), & \text{if } i=c_n\\
\end{cases}, \quad  (n,i)-1:=\begin{cases}
 (n,i-1),      &\text{if } i>1\\
(n-1,c_{n-1}), &\text{if }  i=1\\
\end{cases}.
\end{align*} Let $\sigma_{(N,I)}^2:= \mathbbm{E}(\sum_{(n,i) \le (N,I)} \overline{Y}_{n,i})^2$, $R_{N,I} :=\sum_{(n,i) \ge (N,I)} \frac{\overline{Y}_{n,i}}{\sigma_{(n,i)}^2}$. Since $\{\overline{Y}_{n,i}\}_{(n,i)} $ are independent, \[\delta_{(N,I)}^2:=\mathbb{E}R_{N,I}^2= \sum_{(n,i) \ge (N,I)}\mathbb{E}\frac{\overline{Y}_{n,i}^2}{\sigma_{(n,i)}^4}= \sum_{(n,i) \ge (N,I)}\frac{\sigma_{(n,i)}^2-\sigma_{(n,i)-1}^2}{\sigma_{(n,i)}^4} \approx \sigma_{(N,I)}^{-2}.\] Hence $(R_{N,I})_{(N,I)}$ is a one-dimensional reverse martingale w.r.t. a decreasing filtration $(\mathcal{F}_{(N,I)})_{(N,I)}:=(\sigma\{\overline{Y}_{n,i}:(n,i)\ge (N,I)\})_{(N,I)}$. By the Skorokhod embedding \cite{skorokhod}, there is a Brownian motion $(B_t)_{t \ge 0}$ and a family of decreasing stopping times $\tau_{(N,I)}$ such that $\lim_{N \to \infty}\tau_{(N,I)}=0$ and  \[R_{N,I} = B_{\tau_{(N,I)}}, \quad \mathbb{E}\left(\tau_{(N,I)}-\tau_{(N,I)+1} \mid \mathcal{G}_{(N,I)+1} \right) = \mathbb{E} \frac{\overline{Y}_{N,I}^2}{\sigma_{(N,I)}^4} = \frac{\sigma_{(N,I)}^2 - \sigma_{(N,I)-1}^2}{\sigma_{(N,I)}^4},\]\[\overline{Y}_{N,I} = \sigma_{(N,I)}^2 \Big( B_{\tau_{(N,I)}} - B_{\tau_{(N,I)+1}}\Big),\]
\begin{align}\label{sk}
    \mathbb{E}\Big( |\tau_{(N,I)}-\tau_{(N,I)+1}|^{1+\frac{\xi-2}{2(\xi-1)}} \mid \mathcal{G}_{(N,I)+1} \Big) \le C \mathbb{E}\Big|\frac{\overline{Y}_{N,I}}{\sigma_{(N,I)}^2}\Big|^{2(1+\frac{\xi-2}{2(\xi-1)})}
\end{align} where the constant $C>0$ does not depend on $(N,I)$ and $\mathcal{G}_{(N,I)}:=\sigma\{\overline{Y}_{n,i}, \tau_{(n,i)}:(n,i)\ge (N,I)\}$. Now we can approximate $\overline{Y}_{n,i}$.
\begin{align*}
    &\sum_{(n,i) \le (N,I)} \overline{Y}_{n,i} = \sum_{(n,i) \le (N,I)} \sigma_{(n,i)}^2 (B_{\tau_{(n,i)}} - B_{\tau_{(n,i)+1}})\\
    &= \sum_{(n,i) \le (N,I)} \sigma_{(n,i)}^2 \left( B_{\tau_{(n,i)}} - B_{\delta_{(n,i)}^2} - (B_{\tau_{(n,i)+1}} - B_{\delta_{(n,i)+1}^2}) \right)+ \sum_{(n,i) \le (N,I)} \sigma_{(n,i)}^2 (B_{\delta_{(n,i)}^2} - B_{\delta_{(n,i)+1}^2})\\
    &= \sum_{(n,i) \le (N,I)} (\sigma_{(n,i)}^2 - \sigma_{(n,i)-1}^2) (B_{\tau_{(n,i)}} - B_{\delta_{(n,i)}^2}) -\sigma_{(N,I)}^2 (B_{\tau_{(N,I)+1}} - B_{\delta_{(N,I)+1}^2}) \\
    &\quad +\sum_{(n,i) \le (N,I)} \sigma_{(n,i)}^2 (B_{\delta_{(n,i)}^2} - B_{\delta_{(n,i)+1}^2}).
    \end{align*}

$\sigma_{(n,i)}^2 (B_{\delta_{(n,i)}^2} - B_{\delta_{(n,i)+1}^2})$ is the required Gaussian variable. Now we address the error terms $B_{\tau_{(n,i)}} - B_{\delta_{(n,i)}^2}$ by using the following ASIP criterion.
Fix a small $ \varepsilon_{00} \in (0, \frac{1}{8})$ (to be determined) and apply the fact that the Brownian motion is $(0.5-\epsilon_{00})$-H\"older continuous near time $0$. 
\begin{lemma}[ASIP criterion]\label{asipcriteria}
    If $\tau_{(n,i)} - \delta_{(n,i)}^2 = O( \sigma_{(n,i)}^{-2(1+4\varepsilon_{00})})$ a.s., then \[\sum_{(n,i) \le (N,I)} \overline{Y}_{n,i}=\sum_{(n,i) \le (N,I)} \sigma_{(n,i)}^2 (B_{\delta_{(n,i)}^2} - B_{\delta_{(n,i)+1}^2})+ O(\sigma_{(N,I)}^{2\left(\frac{1}{2} - (\varepsilon_{00} - 4\varepsilon_{00}^2)\right)}) \text{ a.s.}\]
\end{lemma}
\begin{proof}
Using the properties of Brownian motions, the error term becomes \begin{align*}
&\Big|\sum_{(n,i) \le (N,I)} (\sigma_{(n,i)}^2 - \sigma_{(n,i)-1}^2) (B_{\tau_{(n,i)}} - B_{\delta_{(n,i)}^2}) -\sigma_{(N,I)}^2 (B_{\tau_{(N,I)+1}} - B_{\delta_{(N,I)+1}^2})\Big| \\
    &= \sum_{(n,i) \le (N,I)} (\sigma_{(n,i)}^2 - \sigma_{(n,i)-1}^2) O( |\tau_{(n,i)} - \delta_{(n,i)}^2|^{\frac{1}{2} - \varepsilon_{00}})+\sigma_{(N,I)}^2 O(|\tau_{(N,I)+1}-\delta_{(N,I)+1}^2|^{\frac{1}{2}-\epsilon_{00}}).
\end{align*}

Applying the assumptions, we have \[\sum_{(n,i) \le (N,I)} (\sigma_{(n,i)}^2 - \sigma_{(n,i)-1}^2) O(|\tau_{(n,i)} - \delta_{(n,i)}^2|^{\frac{1}{2} - \varepsilon_{00}})\lesssim \sum_{(n,i) \le (N,I)} \frac{\sigma_{(n,i)}^2 - \sigma_{(n,i)-1}^2}{\sigma_{(n,i)}^{2(\frac{1}{2} - \varepsilon_{00})(1+4\varepsilon_{00})}} \lesssim \sigma_{(N,I)}^{2\left(\frac{1}{2} - (\varepsilon_{00} - 4\varepsilon_{00}^2)\right)},\] \[\sigma_{(N,I)}^2 O(|\tau_{(N,I)+1}-\delta_{(N,I)+1}^2|^{\frac{1}{2}-\varepsilon_{00}})\lesssim  \frac{\sigma_{(N,I)}^2}{\sigma_{(N,I)}^{2(\frac{1}{2} - \varepsilon_{00})(1+4\varepsilon_{00})}} \lesssim \sigma_{(N,I)}^{2\left(\frac{1}{2} - (\varepsilon_{00} - 4\varepsilon_{00}^2)\right)}\] which concludes the proof. 
\end{proof}

Before verifying the conditions of Lemma \ref{asipcriteria}, we need to know more about $\sigma_{(N,I)}^2$.
\begin{lemma}\label{blockvariance}
    When  $\varepsilon \xi < \frac{1}{2} , \varepsilon_1 < \frac{1}{4} , 1-c < \frac{1}{4}$ and $\varepsilon(\xi-1) + \max\{3-\xi, 0\} + 1 - c < 1$, $\sigma_{(n,i)}^2\approx n^{a+1}$ where the constant in ``$\approx$" is independent of $n,i$.
\end{lemma}
\begin{proof}
By Lemma \ref{approximation} and Lemma \ref{variancesigma}, $\mathbb{E} \left[ \frac{\sum_{t \in I_n} [\mathbbm{1}_{\pi^{-1}B} \circ F^t - \mu_\Delta(\pi^{-1}B)]}{|I_n|^{\frac{1}{2}}} \right]^2 \approx \sigma^2$ and

\begin{align*}
    & \Big|\frac{ \mathbb{E} (\sum_{t \in I_n} [ \mathbbm{1}_{\pi^{-1}B} \circ F^t - \mu_\Delta(\pi^{-1}B) ] )^2- \mathbb{E} (\sum_{t \in I_n} [ \mathbbm{1}_{\overline{A}_{2k_n}} \circ F^{t-k_n} - \mu_\Delta(\overline{A}_{2k_n}) ])^2 }{ |I_n| }\Big|\\
    &= \Big|\frac{ \mathbb{E} (\sum_{t \in I_n} [ \mathbbm{1}_{F^{-k_n}\pi^{-1}B} \circ F^{t-k_n} - \mu_\Delta(F^{-k_n}\pi^{-1}B) ] )^2- \mathbb{E} (\sum_{t \in I_n} [ \mathbbm{1}_{\overline{A}_{2k_n}} \circ F^{t-k_n} - \mu_\Delta(\overline{A}_{2k_n}) ])^2 }{ |I_n| }\Big|\\
    &\lesssim \frac{1}{k_n^{\min\{\alpha-1, \xi-2\}}} \lesssim \frac{1}{|I_n|^{\varepsilon \min\{\alpha-1, \xi-2\}}}
\end{align*}where the constants in ``$\approx, \precsim$" do not depend on $n$. These two estimates and $F$-invariance imply \[\frac{ \mathbb{E}\left( \sum_{t \in I_n} [ \mathbbm{1}_{\overline{A}_{2k_n}} \circ F^{t} - \mu_\Delta(\overline{A}_{2k_n}) ] \right)^2 }{ |I_n| }=\frac{ \mathbb{E}\left( \sum_{t \in I_n} [ \mathbbm{1}_{\overline{A}_{2k_n}} \circ F^{t-k_n} - \mu_\Delta(\overline{A}_{2k_n}) ] \right)^2 }{ |I_n| } \approx \sigma^2\]
where the constant in ``$\approx$" is independent of $n$. On the other hand, by Lemma \ref{momentbound}, 
\begin{align*}
    \mathbb{E}&\left( \sum_{j \in \bigcup_i I_{n,i}'} [ \mathbbm{1}_{\overline{A}_{2k_n}} \circ F^j - \mu_\Delta(\overline{A}_{2k_n}) ] \right)^2\\
    &= \mathbb{E}\left( \sum_{t \in I_n} [ \mathbbm{1}_{\overline{A}_{2k_n}} \circ F^{t} - \mu_\Delta(\overline{A}_{2k_n}) ] \right)^2+k_n^{\xi-1}O(|J_n|+c_n \cdot |I_n|^{c\varepsilon_1})\\
    &= \mathbb{E}\left( \sum_{t \in I_n} [ \mathbbm{1}_{\overline{A}_{2k_n}} \circ F^{t} - \mu_\Delta(\overline{A}_{2k_n}) ] \right)^2+ O(|I_n|^{\varepsilon(\xi-1)}[|I_n|^{\varepsilon}+|I_n|^{1-c}|I_n|^{c\varepsilon_1}]).
\end{align*}

When $\varepsilon \xi< \frac{1}{2} , \varepsilon_1 < \frac{1}{4} , 1-c < \frac{1}{4}$, \[\frac{\mathbb{E}\left(\sum_i Y_{n,i}\right)^2}{|I_n|}=\frac{ \mathbb{E}\left( \sum_{j \in \bigcup_i I_{n,i}'} [ \mathbbm{1}_{\overline{A}_{2k_n}} \circ F^j - \mu_\Delta(\overline{A}_{2k_n}) ] \right)^2 }{ |I_n| } \approx \sigma^2\] where the constant in ``$\approx$" is independent of $n$. Using the fact that $\overline{Y}_{n,i}$ are independent, $Y_{n,i}=_d\overline{Y}_{n,i} $, and Lemma \ref{momentbound}, we have 
\begin{align}
    \mathbb{E}\left( \sum_i \overline{Y}_{n,i} \right)^2 &= \sum_i \mathbb{E} \left( \sum_{j \in I_{n,i}'} [ \mathbbm{1}_{\overline{A}_{2k_n}} \circ F^j - \mu_\Delta(\overline{A}_{2k_n}) ] \right)^2\nonumber\\
    &= \mathbb{E}\left[ \sum_i Y_{n,i} \right]^2 - 2\sum_{i < j} \mathbb{E} Y_{n,i} Y_{n,j}\nonumber\\
    &= \mathbb{E} \left[ \sum_i Y_{n,i} \right]^2 - 2\sum_{i=1}^{c_n}\mathbb{E} [Y_{n,i}  \sum_{i < j \le c_n} Y_{n,j}]\nonumber\\
    &= \mathbb{E}\left[ \sum_i Y_{n,i} \right]^2 + \sum_{i=1}^{c_n} O(k_n^{(\xi-1)} |I_n|^{\max\{3-\xi, 0\}}\log n)\nonumber\\
    &= \mathbb{E}\left[ \sum_i Y_{n,i} \right]^2 + \sum_{i=1}^{c_n} O(|I_n|^{\varepsilon(\xi-1)} |I_n|^{\max\{3-\xi, 0\}}\log n).\label{6}
\end{align}

Hence, when $\varepsilon(\xi-1) + \max\{3-\xi, 0\} + 1 - c < 1$, we have
\[\frac{\mathbb{E}\left(\sum_i \overline{Y}_{n,i}\right)^2}{|I_n|} = \frac{\mathbb{E}\left(\sum_i Y_{n,i}\right)^2}{|I_n|} + \frac{O(|I_n|^{\varepsilon(\xi-1) + \max\{3-\xi, 0\} + 1 - c}\log n)}{|I_n|}\approx \sigma^2,\]
where the constant in ``$\approx$" is independent of $n$. Together with the fact that \[\sigma_{(n,i)}^2 \in \left[ \sum_{m \le n-1} \mathbb{E}\left(\sum_i \overline{Y}_{m,i}\right)^2, \sum_{m \le n} \mathbb{E}\left(\sum_i \overline{Y}_{m,i}\right)^2 \right],\] we obtain $\sigma_{(n,i)}^2 \approx \sum_{j \le n }|I_j| \approx n^{a+1} $ where ``$\approx$" is independent of $i, n$.
\end{proof} 

Now we can verify the conditions of Lemma \ref{asipcriteria}. Note that $\tau_{(N,I)} - \delta_{(N,I)}^2 = \sum_{(n,i) \ge (N,I)} \Delta \tau_{(n,i)} - \mathbb{E}[\Delta \tau_{(n,i)}|\mathcal{G}_{(n,i)+1}]$ is a reverse martingale where $\Delta \tau_{(n,i)}:=\tau_{(n,i)}-\tau_{(n,i)+1}$. Let $\varepsilon_0:=\frac{\xi-2}{2(\xi-1)}\in (0,1/2)$. 

\begin{lemma}\label{verifyasipcriteria}
    When $\frac{(3\xi-4)(\xi-1)}{\xi-2}\varepsilon < \frac{(1-c)\varepsilon_0}{4}, 8\varepsilon_{00} < \frac{(1-c)\varepsilon_0}{4} < \frac{\varepsilon_0}{2}, a\frac{(1-c)\varepsilon_0}{2} > 1$,  $\varepsilon \xi < \frac{1}{2} , \varepsilon_1 < \frac{1}{4} , 1-c < \frac{1}{4}$ and $\varepsilon(\xi-1) + \max\{3-\xi, 0\} + 1 - c < 1$, we have  $\sup_{(n,i) \ge (N,I)} |\tau_{(n,i)} - \delta_{(n,i)}^2| = O(\sigma_{(N,I)}^{-2(1+4\varepsilon_{00})})$ a.s.
\end{lemma}
\begin{proof}
 By the Burkholder–Davis–Gundy inequalities,
\begin{align*}
   & \mathbb{E} \sup_{(n, i) \ge (N,I)} \left| \tau_{(n,i)} - \delta_{(n,i)}^2 \right|^{1+\varepsilon_0} \\
    &= \mathbb{E} \sup_{(n, i) \ge (N,I)} \left| \sum_{(n_1,i_1) \ge (n, i)} \Delta \tau_{(n_1,i_1)} - \mathbb{E}[\Delta \tau_{(n_1,i_1)} \mid \mathcal{G}_{(n_1,i_1)+1}] \right|^{1+\varepsilon_0}\\
    &\approx \mathbb{E} \left( \sqrt{ \sum_{(n,i) \ge (N,I)} \left[ \Delta \tau_{(n,i)} - \mathbb{E}[\Delta \tau_{(n,i)}|\mathcal{G}_{(n,i)+1}] \right]^2 } \right)^{1+\varepsilon_0}\\
    &\lesssim \sum_{(n,i) \ge (N,I)} \mathbb{E} \left| \Delta \tau_{(n,i)} - \mathbb{E}[\Delta \tau_{(n,i)}|\mathcal{G}_{(n,i)+1}]  \right|^{1+\varepsilon_0} \lesssim \sum_{(n,i) \ge (N,I)} \mathbb{E} \left|\Delta \tau_{(n,i)}\right|^{1+\varepsilon_0}
\end{align*}

By (\ref{sk}), Lemma \ref{blockvariance} and Lemma \ref{momentbound}, we can continue to estimate
\begin{align*}
\lesssim \sum_{(n,i) \ge (N,I)} \mathbb{E} \left| \frac{\overline{Y}_{n,i}^2}{\sigma_{(n,i)}^4} \right|^{(1+\varepsilon_0)} &\lesssim \sum_{(n,i) \ge (N,I)} \mathbb{E} \left| \frac{Y_{n,i}^2}{\sigma_{(n,i)}^4} \right|^{(1+\varepsilon_0)}\\
&\lesssim \sum_{(n,i) \ge (N,I)} \frac{|I_{n,i}|^{1+\varepsilon_0} k_n^{\frac{(3\xi-4)(\xi-1)}{\xi-2}}}{\sigma_{(n,i)}^{2(2+2\varepsilon_0)}}\\
&\lesssim \sum_{(n,i) \ge (N,I)} \frac{n^{ac(1+\varepsilon_0)} n^{a\varepsilon\frac{(3\xi-4)(\xi-1)}{\xi-2}}}{n^{(a+1)(2+2\varepsilon_0)}}\\
&\lesssim \sum_{n \ge N} \frac{n^{a(1-c) + ac(1+\varepsilon_0) + \varepsilon a\frac{(3\xi-4)(\xi-1)}{\xi-2}}}{n^{(a+1)(2+2\varepsilon_0)}}\\
&\precsim \sum_{n \ge N} \frac{n^{a(1+c\varepsilon_0+\varepsilon\frac{(3\xi-4)(\xi-1)}{\xi-2})}}{n^{a(2+2\varepsilon_0)} n^{2+2\varepsilon_0}} \precsim \frac{1}{N^{a(1+(2-c)\varepsilon_0-\varepsilon\frac{(3\xi-4)(\xi-1)}{\xi-2})+1+2\varepsilon_0}}
\end{align*}

Choose $\frac{(3\xi-4)(\xi-1)}{\xi-2}\varepsilon < \frac{(1-c)\varepsilon_0}{4}, 8\varepsilon_{00} < \frac{(1-c)\varepsilon_0}{4} < \frac{\varepsilon_0}{2}, a\frac{(1-c)\varepsilon_0}{2} > 1$, then $a(1+(2-c)\varepsilon_0 - \varepsilon\frac{(3\xi-4)(\xi-1)}{\xi-2} - 1 - \varepsilon_0 - 4\varepsilon_{00} - 4\varepsilon_{00}\varepsilon_0) + 1 + 2\varepsilon_0 - 1 - \varepsilon_0 - 4\varepsilon_{00} - 4\varepsilon_0\varepsilon_{00}= a[(1-c)\varepsilon_0 - \varepsilon\frac{(3\xi-4)(\xi-1)}{\xi-2} - 4\varepsilon_{00} - 4\varepsilon_{00}\varepsilon_0] + \varepsilon_0 - 4\varepsilon_{00} - 4\varepsilon_0\varepsilon_{00}> a\frac{(1-c)\varepsilon_0}{2}>1$.
Therefore,
\begin{align*}
  \mathbb{E}\frac{\sup_{(n, i) \ge (N, 1)} |\tau_{(n,i)} - \delta_{(n,i)}^2|^{1+\varepsilon_0}}{\sigma_{(N, 1)}^{-2(1+4\varepsilon_{00})(1+\varepsilon_0)}} \lesssim \frac{N^{(a+1)(1+4\varepsilon_{00})(1+\varepsilon_0)}}{N^{a(1+(2-c)\varepsilon_0-\varepsilon\frac{(3\xi-4)(\xi-1)}{\xi-2})+1+2\varepsilon_0}} \lesssim \frac{1}{N^{a\frac{(1-c)\varepsilon_0}{2}}}.
\end{align*}

By the Borel-Cantelli lemma, $\sup_{(n,i) \ge (N,I)} |\tau_{(n,i)} - \delta_{(n,i)}^2|\le \sup_{(n,i) \ge (N,1)} |\tau_{(n,i)} - \delta_{(n,i)}^2| = O(\sigma_{(N,1)}^{-2(1+4\varepsilon_{00})})=O(\sigma_{(N,I)}^{-2(1+4\varepsilon_{00})})$ a.s.
\end{proof}

Thus Lemma \ref{asipcriteria} gives independent Gaussian variables $G_{n,i}:=\sigma_{(n,i)}^2 (B_{\delta_{(n,i)}^2} - B_{\delta_{(n,i)+1}^2})$ such that $\mathbb{E}G_{n,i}^2 = \mathbb{E}\overline{Y}_{n,i}^2 = \mathbb{E}Y_{n,i}^2$. Combining all lemmas in this step and Lemma \ref{approximateindependnt}, we obtain
\begin{lemma}\label{asip2}
   Under the conditions of Lemma \ref{blockvariance}, Lemma \ref{approximateindependnt} and Lemma \ref{verifyasipcriteria}, $$|\sum_{(n,i) \le (N,I)} \overline{Y}_{n,i} - G_{n,i}| \lesssim \sigma_{N,I}^{2\left(\frac{1}{2} - (\varepsilon_{00} - 4\varepsilon_{00}^2)\right)} \lesssim (N^{a+1})^{\left(\frac{1}{2} - (\varepsilon_{00} - 4\varepsilon_{00}^2)\right)} \text{ a.s.}, $$
   $$|\sum_{(n,i) \le (N,I)} Y_{n,i} - G_{n,i}|  \lesssim (N^{a+1})^{\left(\frac{1}{2} - (\varepsilon_{00} - 4\varepsilon_{00}^2)\right)}+(N^{a+1})^{0.25} \text{ a.s.} $$
\end{lemma}

\subsubsection*{Step \ref{mx}: a double maximal inequality}

Since the Gaussian approximation in Lemma \ref{asip2} is obtained for blocks $I_{n,i}'$, in order to pass it to any time, we need the following new maximal inequality. 

\begin{lemma}\label{maxinequality}
   When $a[\frac{0.5(\xi-2)}{2\xi-2}(1-c)-0.5\varepsilon {\frac{(3\xi-4)(\xi-1)}{\xi-2}}] > 1,  \varepsilon < \frac{(1-c)(\xi-2)^2}{2(\xi-1)^2(3\xi-4)}$, then $$\sup_I\sup_u \left| \sum_{j \in I_{N, I}' \cap [0, u]} \left[ \mathbbm{1}_{\overline{A}_{2k_{N}}} \circ F^j - \mu_\Delta(\overline{A}_{2k_{N}}) \right] \right|\lesssim N^{a\left( \frac{1}{2} - [\frac{0.5(\xi-2)(1-c)}{6\xi-8}-\frac{\varepsilon(\xi-1)^2}{2\xi-4}] \right)} \text{ a.s.}$$ holds for any $N\ge 1$.
 \end{lemma} 
\begin{proof}
Fixing $\overline{A}_{2k_{N}}$ and applying Lemma \ref{momentbound} to the stationary process $\mathbbm{1}_{\overline{A}_{2k_{N}}} \circ F^t - \mu_\Delta(\overline{A}_{2k_{N}})$, we obtain $\sup_u\mathbb{E}  \Big| \sum_{t \in [u,u+v] } \left[ \mathbbm{1}_{\overline{A}_{2k_{N}}} \circ F^t - \mu_\Delta(\overline{A}_{2k_{N}}) \right] \Big|^{\frac{(3\xi-4)}{\xi-1}} \lesssim k_N^{\frac{(3\xi-4)(\xi-1)}{\xi-2}}  v^{^{\frac{(3\xi-4)}{2\xi-2}} }$ where the constant in ``$\precsim$" does not depend on any $N,v$. Applying Lemma \ref{usefulprob} to the bound above and taking $v= |I'_{N,I}|$, we obtain
\begin{align*}
&\mathbb{E} \left[\sup_u\left| \sum_{j \in I_{N, I}'\bigcap [0,u] } \left[ \mathbbm{1}_{\overline{A}_{2k_{N}}} \circ F^j - \mu_\Delta(\overline{A}_{2k_{N}}) \right] \right|^{\frac{(3\xi-4)}{\xi-1}} \right] \lesssim  k_N^{^{\frac{(3\xi-4)(\xi-1)}{\xi-2}} } |I_{N, I}'|^{\frac{(3\xi-4)}{2\xi-2}}.
\end{align*}

By the $F$-invariance, we have
\begin{align*}
&\mathbb{E} \left[\sup_I\sup_u\left| \sum_{j \in I_{N, I}'\bigcap [0,u] } \left[ \mathbbm{1}_{\overline{A}_{2k_{N}}} \circ F^j - \mu_\Delta(\overline{A}_{2k_{N}}) \right] \right|^{\frac{(3\xi-4)}{\xi-1}} \right] \\
&\le \sum_I\mathbb{E} \left[\sup_u\left| \sum_{j \in I_{N, I}'\bigcap [0,u] } \left[ \mathbbm{1}_{\overline{A}_{2k_{N}}} \circ F^j - \mu_\Delta(\overline{A}_{2k_{N}}) \right] \right|^{\frac{(3\xi-4)}{\xi-1}} \right] \lesssim  c_N k_N^{^{\frac{(3\xi-4)(\xi-1)}{\xi-2}} } |I_{N, I}'|^{\frac{(3\xi-4)}{2\xi-2}}.
\end{align*}

which implies \begin{align*}
    &\mathbb{E}\left[\frac{ \sup_I\sup_u \left| \sum_{j \in I_{N, I}' \cap [0, u]} \left[ \mathbbm{1}_{\overline{A}_{2k_{N}}} \circ F^j - \mu_\Delta(\overline{A}_{2k_{N}}) \right] \right|^{\frac{(3\xi-4)}{\xi-1}}}{N^{a\left( \frac{1}{2} - \frac{w}{2} \right)\frac{(3\xi-4)}{\xi-1}}}\right]\\
    &\lesssim \frac{  c_N k_N^{^{\frac{(3\xi-4)(\xi-1)}{\xi-2}} } |I_{N, I}'|^{\frac{(3\xi-4)}{2\xi-2}}}{ N^{a(1+\frac{\xi-2}{2\xi-2}-w\frac{(3\xi-4)}{2\xi-2})} }   \lesssim \frac{1}{ N^{a[\frac{\xi-2}{2\xi-2}(1-c)-w\frac{(3\xi-4)}{2\xi-2}-\varepsilon {\frac{(3\xi-4)(\xi-1)}{\xi-2}}]}}\lesssim \frac{1}{ N^{a[\frac{0.5(\xi-2)}{2\xi-2}(1-c)-0.5\varepsilon {\frac{(3\xi-4)(\xi-1)}{\xi-2}}]}}
\end{align*}where $w:=\frac{0.5(\xi-2)(1-c)}{3\xi-4}-\frac{\varepsilon(\xi-1)^2}{\xi-2}>0$. By the Borel-Cantelli lemma, we obtain \[\sup_I\sup_u \left| \sum_{j \in I_{N, I}' \cap [0, u]} \left[ \mathbbm{1}_{\overline{A}_{2k_{N}}} \circ F^j - \mu_\Delta(\overline{A}_{2k_{N}}) \right] \right|\lesssim N^{a\left( \frac{1}{2} - \frac{\frac{0.5(\xi-2)(1-c)}{3\xi-4}-\frac{\varepsilon(\xi-1)^2}{\xi-2}}{2} \right)} \text{ a.s.}\]
\end{proof}
\begin{remark}\label{cannotmax}
    Since the regularity of $\mathbbm{1}_{\overline{A}_{2k_{N}}}-\mu_\Delta(\overline{A}_{2k_{N}})$ would blow up as $N$ increases,  we cannot derive a moment estimate for $\sup_u|\sum_{i\in I_{N} \bigcap [0,u]}[\mathbbm{1}_{\overline{A}_{2k_{N}}}\circ F^i-\mu_\Delta(\overline{A}_{2k_{N}})]|$. This is why  $I_{N}$ is partitioned further and a maximal inequality is derived for the smaller block $I'_{N,I}$ uniformly: $\sup_I\sup_u|\sum_{i\in I'_{N,I}\bigcap [0,u]}[\mathbbm{1}_{\overline{A}_{2k_{N}}}\circ F^i-\mu_\Delta(\overline{A}_{2k_{N}})]|$. In the existing ASIP literature for regular observables (e.g., \cite{MN2}), $\mathbbm{1}_{\overline{A}_{2k_{N}}}$ would not appear and the maximal inequality in that setting was simpler.
\end{remark}

\subsubsection*{Step \ref{asipontower}: ASIP for $S_m\circ \pi$ and $S_m$}
Throughout this step, we suppose that $m \in [a_N, a_{N+1})$ for some $N$ and $I_{N+1,I}$ is the rightmost complete subblock contained in $ [0,m)$. For convenience, we set $\sum_{j \in I_{n, c_n+1}' }(\cdot):=0$ because $I_{n, i}'$ is defined only for $i\le c_n $. Also set $I=1$ if $m\in I_{N+1,1}$.  Now we can approximate $\sum_{r \le m} \left[ \mathbbm{1}_{\pi^{-1}B} \circ F^r - \mu_\Delta(\pi^{-1}B) \right]$ by $\sum_{(n,i) \le (N+1, I)} G_{n,i}$. 
\begin{lemma}
    Under the conditions of Lemma \ref{maxinequality}, Lemma \ref{asip1}, Lemma \ref{asip2}, and Lemma \ref{ruleoutsmall}, \begin{align*}
        &\sum_{r \le m} \left[ \mathbbm{1}_{\pi^{-1}B} \circ F^r - \mu_\Delta(\pi^{-1}B) \right]=\sum_{(n,i) \le (N+1, I)} G_{n,i} + O\left( m^{\frac{1}{2} - (\varepsilon_{00} - 4\varepsilon_{00}^2)} \right)\\
        &+O\left( m^{0.3} \right)+O\left( m^{0.25} \right) + O\left( m^{0.5-[\varepsilon \min\{\xi-2,\alpha-1\}/4 - 1/(a+1)]} \right) + O\left(m^{ \frac{1}{2} - [\frac{0.5(\xi-2)(1-c)}{6\xi-8}-\frac{\varepsilon(\xi-1)^2}{2\xi-4}] }\right) \text{ a.s.}
    \end{align*}
\end{lemma} 
\begin{proof}By Lemma \ref{asip1} and Lemma \ref{ruleoutsmall}, we obtain almost surely,
    \begin{align*}
    &\sum_{r \le m} \left[ \mathbbm{1}_{\pi^{-1}B} \circ F^r - \mu_\Delta(\pi^{-1}B) \right]  \\
    &=\sum_{n \le N+1} \sum_{r \in [0, m] \cap I_n} \left[ \mathbbm{1}_{\overline{A}_{2k_n}} \circ F^{r-k_n} - \mu_\Delta(\overline{A}_{2k_n}) \right]+ O\Big( m^{0.5-[\varepsilon \min\{\xi-2,\alpha-1\}/4 - 1/(a+1)]}  \Big)\\
    &=\sum_{n \le N+1} \sum_{j \in [0, m] \cap I_n'} \left[ \mathbbm{1}_{\overline{A}_{2k_n}} \circ F^j - \mu_\Delta(\overline{A}_{2k_n}) \right]+O\left( m^{0.25} \right) + O\left( m^{0.5-[\varepsilon \min\{\xi-2,\alpha-1\}/4 - 1/(a+1)]} \right)\\
    &=\sum_{n \le N+1} \sum_{j \in [0, m] \cap \bigcup_i I_{n,i}'} \left[ \mathbbm{1}_{\overline{A}_{2k_n}} \circ F^j - \mu_\Delta(\overline{A}_{2k_n}) \right] +O\left( m^{0.3} \right)+O\left( m^{0.25} \right) \\
    &\quad+ O\left( m^{0.5-[\varepsilon \min\{\xi-2,\alpha-1\}/4 - 1/(a+1)]} \right).
    \end{align*}

    Now by Lemma \ref{asip2}, we can continue our estimate
    \begin{align*}
    &=\sum_{j \in \bigcup_{(n,i) \le (N+1, I)} I_{n,i}'} \left[ \mathbbm{1}_{\overline{A}_{2k_n}} \circ F^j - \mu_\Delta(\overline{A}_{2k_n}) \right]\\
    &\quad + O\left(\sup_{i \le c_{N+1}+1}\sup_v \left| \sum_{j \in I_{N+1, i}' \cap [0, v]} \left[ \mathbbm{1}_{\overline{A}_{2k_{N+1}}} \circ F^j - \mu_\Delta(\overline{A}_{2k_{N+1}}) \right] \right|\right)\\
    &\quad+O\left( m^{0.3} \right)+O\left( m^{0.25} \right) + O\left( m^{0.5-[\varepsilon \min\{\xi-2,\alpha-1\}/4 - 1/(a+1)]} \right)\\
    &=\sum_{(n,i) \le (N+1, I)} G_{n,i} + O\left( m^{\frac{1}{2} - (\varepsilon_{00} - 4\varepsilon_{00}^2)} \right) + O\left(m^{\frac{1}{4}}\right)\\
    &\quad + O\left(\sup_{i\le c_{N+1}+1}\sup_v \left| \sum_{j \in I_{N+1, i}' \cap [0, v]} \left[ \mathbbm{1}_{\overline{A}_{2k_{N+1}}} \circ F^j - \mu_\Delta(\overline{A}_{2k_{N+1}}) \right] \right|\right)\\
    &\quad +O\left( m^{0.3} \right)+O\left( m^{0.25} \right) + O\left( m^{0.5-[\varepsilon \min\{\xi-2,\alpha-1\}/4 - 1/(a+1)]} \right)\\
    &=\sum_{(n,i) \le (N+1, I)} G_{n,i} + O\left( m^{\frac{1}{2} - (\varepsilon_{00} - 4\varepsilon_{00}^2)} \right)+ O\left(m^{ \frac{1}{2} - [\frac{0.5(\xi-2)(1-c)}{6\xi-8}-\frac{\varepsilon(\xi-1)^2}{2\xi-4}] }\right) \\
    &\quad +O\left( m^{0.3} \right)+O\left( m^{0.25} \right)+ O\left( m^{0.5-[\varepsilon \min\{\xi-2,\alpha-1\}/4 - 1/(a+1)]} \right)
\end{align*}where the second  $O(\cdot)$ is due to Lemma \ref{maxinequality}.
\end{proof}

Now we approximate $\sum_{(n,i) \le (N+1, I)} G_{n,i}$ by $ B_{\sigma^2m}$.
\begin{lemma}\label{asip3}
when  $1-c<\frac{1-\max\{3-\xi,0\}}{8}, \varepsilon<\frac{1-\max\{3-\xi,0\}}{8(\xi-1)},\xi\varepsilon<1, 1-c < \frac{1}{8} , \varepsilon_1 < \frac{1}{8} , \varepsilon < \frac{1}{8(\xi-1)}, \varepsilon < \frac{1}{\min\{\alpha-1,\xi-2\}},  a \ge 8$ and $a > \frac{1}{1-\max(2-\min\{\xi-1, \alpha\},0) }$,

\[ \sum_{(n,i) \le (N+1,I)} G_{n,i} - B_{\sigma^2 m} = O(m^{u/2}\log m)+O(m^{\max\{2-\min(\xi-1,\alpha),0\}/2}\log^{2} m) \text{ a.s.}\]
where $$u:= \max \left\{\frac{a}{a+1}(1-c+(\xi-1)\varepsilon+\max\{3-\xi,0\}) + \frac{1}{a+1}, \frac{\xi a\varepsilon+1}{a+1}, \frac{a}{a+1}(1-c+c\varepsilon_1+(\xi-1)\varepsilon) + \frac{1}{a+1}, \right.$$
$$\left. \frac{a}{a+1}(1-\varepsilon \min\{\alpha-1, \xi-2\}) + \frac{1}{a+1},
  \frac{a}{a+1}, \max(2-\min\{\xi-1, \alpha\},0) + \frac{1}{a+1} \right\}<1$$  and the constants in $O(\cdot)$ are independent of $N,I,m$.
\end{lemma}
\begin{proof}
    First we need to compare the variances. By Lemma \ref{variancegrow},
    \begin{align*}
        &\mathbb{E} \left[ \sum_{r \le m} \mathbbm{1}_{B} \circ f^r - \mu_{\mathcal{M}}(B) \right]^2\\
        &= \mathbb{E} \left[ \sum_{n \le N} \sum_{i \in I_n} \mathbbm{1}_B \circ f^i - \mu_{\mathcal{M}}(B) + \sum_{i \in I_{N+1} \cap [0, m]} \mathbbm{1}_B \circ f^i - \mu_{\mathcal{M}}(B) \right]^2\\
        &= \sum_{n \le N} \mathbb{E} \left[ \sum_{i \in I_n} \mathbbm{1}_B \circ f^i - \mu_{\mathcal{M}}(B) \right]^2+ \mathbb{E} \left[ \sum_{i \in I_{N+1} \cap [0, m]} \mathbbm{1}_B \circ f^i - \mu_{\mathcal{M}}(B) \right]^2 \\
        &\quad + 2 \mathbb{E} \left[ \sum_{n \le N} \sum_{i \in I_n} \mathbbm{1}_B \circ f^i - \mu_{\mathcal{M}}(B) \right] \left[ \sum_{j \in I_{N+1} \cap [0, m]} \mathbbm{1}_B \circ f^j - \mu_{\mathcal{M}}(B) \right] \\
        &\quad + 2 \sum_{a' < b \le N} \mathbb{E} \left[ \sum_{i \in I_{a'}} \mathbbm{1}_B \circ f^i - \mu_{\mathcal{M}}(B) \right] \left[ \sum_{j \in I_b} \mathbbm{1}_B \circ f^j - \mu_{\mathcal{M}}(B) \right]\\
        &= \sum_{n \le N} \mathbb{E} \left[ \sum_{i \in I_n} \mathbbm{1}_B \circ f^i - \mu_{\mathcal{M}}(B) \right]^2 + O(|I_{N+1}| + N^{(a+1)\max(2-\min\{\alpha, \xi-1\}, 0)}\log N)\\
        &\quad + 2\sum_{a'=1}^N \sum_{a' < b} \mathbb{E} \left[ \sum_{i \in I_{a'}} \mathbbm{1}_B \circ f^i - \mu_{\mathcal{M}}(B) \right] \left[ \sum_{j \in I_b} \mathbbm{1}_B \circ f^j - \mu_{\mathcal{M}}(B) \right]\\
        &= \sum_{n \le N} \mathbb{E} \left[ \sum_{i \in I_n} \mathbbm{1}_{\pi^{-1}B} \circ F^i - \mu_{\Delta}(\pi^{-1}B) \right]^2  +O(N^a+ N^{(a+1)\max(2-\min\{\alpha, \xi-1\}, 0)}\log N) \\
        &\quad + O(N^{(a+1)\max(2-\min\{\alpha, \xi-1\}, 0) + 1}\log N).
        \end{align*}

        By Lemma \ref{momentbound}, Lemma \ref{approximation} and Lemma \ref{variancegrow} (i.e., $\mathbb{E} \left[ \sum_{i \in \bigcup_{j\le I}I_{N+1,j}} \mathbbm{1}_{\pi^{-1}B} \circ F^i - \mu_{\Delta}(\pi^{-1}B)  \right]^2 =O(N^{a})$), we continue the estimate 
        \begin{align*}
        &= \sum_{n \le N} \mathbb{E} \left[ \sum_{i \in I_n} \mathbbm{1}_{\pi^{-1}B} \circ F^i - \mu_{\Delta}(\pi^{-1}B)  \right]^2+\mathbb{E} \left[ \sum_{i \in \bigcup_{j\le I}I_{N+1,j}} \mathbbm{1}_{\pi^{-1}B} \circ F^i - \mu_{\Delta}(\pi^{-1}B)  \right]^2+O(N^a) \\
        &\quad + O(N^a + N^{(a+1)\max(2-\min\{\alpha, \xi-1\}, 0) + 1}\log N)\\
        &=\sum_{n \le N}\mathbb{E}\left[ \sum_{i \in I_n} \mathbbm{1}_{\overline{A}_{2k_n}} \circ F^{i-k_n} - \mu_\Delta(\overline{A}_{2k_n})\right]^2+\mathbb{E}\left[ \sum_{i \in \bigcup_{j \le I}I_{N+1,j}} \mathbbm{1}_{\overline{A}_{2k_{N+1}}} \circ F^{i-k_{N+1}} - \mu_\Delta(\overline{A}_{2k_{N+1}})\right]^2\\
        &\quad +\sum_{n \le N+1}O(|I_n|k_n^{-\min\{\xi-2,\alpha-1\}}) + O(N^a + N^{(a+1)\max(2-\min\{\alpha, \xi-1\}, 0) + 1}\log N)\\
         &=\sum_{n \le N}\mathbb{E}\left[\sum_{i} Y_{n,i}\right]^2+\mathbb{E}\left[\sum_{i\le I} Y_{N+1,i}\right]^2+\sum_{n \le N+1}O\left(\left[ \sum_{j \in J_n \bigcup \bigcup_{i}J_{n,i}} 1\right]\right)k_n^{\xi-1}\\
         &\quad +\sum_{n \le N+1}O(|I_n|k_n^{-\min\{\xi-2,\alpha-1\}}) + O(N^a + N^{(a+1)\max(2-\min\{\alpha, \xi-1\}, 0) + 1}\log N)
        \end{align*}
        
       By (\ref{6}) and $\mathbb{E}G_{n,i}^2=\mathbb{E}Y_{n,i}^2=\mathbb{E}\overline{Y}_{n,i}^2$, we continue the estimate
        \begin{align*}
         &= \sum_{n \le N} \sum_i \mathbb{E}G_{n,i}^2+\sum_{i\le I} \mathbb{E}G_{N+1,i}^2+\sum_{n\le N+1}\sum_{i=1}^{c_n} O(|I_n|^{\varepsilon(\xi-1)} |I_n|^{\max\{3-\xi, 0\}}\log n)\\
         &\quad + \sum_{n \le N+1}O\left(|I_n|^{\varepsilon} + \sum_{i=1}^{c_n}|I_n|^{c\varepsilon_1} \right)k_n^{\xi-1}  + O\left( \sum_{n \le N+1} |I_n| |I_n|^{-\varepsilon \min\{\alpha-1, \xi-2\}} \right)\\
         &\quad + O \left( N^a + N^{(a+1)\max(2-\min\{\alpha, \xi-1\}, 0) + 1}\log N\right)\\
    &= \sum_{(n,i) \le (N+1,I)} \mathbb{E} G_{n,i}^2+O\left(\sum_{n\le N+1}|I_n|^{1-c+\varepsilon(\xi-1)+\max\{3-\xi, 0\}}\log n\right)\\
    &\quad + O\left( \sum_{n \le N+1} |I_n|^{\xi \varepsilon} + |I_n|^{(1-c)+c\varepsilon_1+(\xi-1)\varepsilon} \right)+ O\left( \sum_{n \le N+1} |I_n| |I_n|^{-\varepsilon \min\{\alpha-1, \xi-2\}} \right)\\
    &\quad + O\left( N^a + N^{(a+1)\max(2-\min\{\alpha, \xi-1\}, 0) + 1}\log N\right)\\
    &= \sum_{(n,i) \le (N+1,I)} \mathbb{E} G_{n,i}^2+O((N^{a+1})^{u})
\end{align*}where $$u:= \max \left\{\frac{a}{a+1}(1-c+(\xi-1)\varepsilon+\max\{3-\xi,0\}) + \frac{1}{a+1}, \frac{\xi a\varepsilon+1}{a+1}, \frac{a}{a+1}(1-c+c\varepsilon_1+(\xi-1)\varepsilon) + \frac{1}{a+1}, \right.$$
$$\left. \frac{a}{a+1}(1-\varepsilon \min\{\alpha-1, \xi-2\}) + \frac{1}{a+1},
  \frac{a}{a+1}, \max(2-\min\{\xi-1, \alpha\},0) + \frac{1}{a+1} \right\}<1$$ when  $1-c<\frac{1-\max\{3-\xi,0\}}{8}, \varepsilon<\frac{1-\max\{3-\xi,0\}}{8(\xi-1)},\xi\varepsilon<1, 1-c < \frac{1}{8} , \varepsilon_1 < \frac{1}{8} , \varepsilon < \frac{1}{8(\xi-1)}, \varepsilon < \frac{1}{\min\{\alpha-1,\xi-2\}},  a \ge 8$ and $a > \frac{1}{1-\max(2-\min\{\xi-1, \alpha\},0) }$.

Therefore $B_{\mathbb{E}[\sum_{i\le m} \mathbbm{1}_B \circ f^i - \mu_{\mathcal{M}}(B)]^2} - \sum_{(n,i) \le (N+1,I)} G_{n,i} = O(m^{u/2}\log m)$ a.s. On the other hand, by Lemma \ref{variancesigma}, $\mathbb{E}[\sum_{i\le m} \mathbbm{1}_B \circ f^i - \mu_{\mathcal{M}}(B)]^2 = m\sigma^2 + O(m^{\max\{2-\min(\xi-1,\alpha),0\}}\log m)$. Hence, \[ B_{\mathbb{E}[\sum_{i\le m} \mathbbm{1}_B \circ f^i - \mu_{\mathcal{M}}(B)]^2} - B_{\sigma^2 m} = O(m^{\max\{2-\min(\xi-1,\alpha),0\}/2}\log^2 m) \text{ a.s.}\] Therefore, $ \sum_{(n,i) \le (N+1,I)} G_{n,i} - B_{\sigma^2 m} = O(m^{u/2}\log m)+O(m^{\max\{2-\min(\xi-1,\alpha),0\}/2}\log^2 m)$ a.s.
\end{proof}

Summarizing all results above, under the conditions of Lemma \ref{maxinequality}, Lemma \ref{asip3}, Lemma \ref{asip1}, Lemma \ref{blockvariance}, Lemma \ref{approximateindependnt}, Lemma \ref{verifyasipcriteria} and Lemma \ref{ruleoutsmall},  we conclude that,  \begin{align}
    &\sum_{r \le m} \left( \mathbbm{1}_{B} \circ f^r - \mu_{\mathcal{M}}(B) \right)\circ \pi=B_{\sigma^2 m}+ O\left( m^{\frac{1}{2} - (\varepsilon_{00} - 4\varepsilon_{00}^2)} \right)+O\left( m^{0.3} \right)+O\left( m^{0.25} \right)\nonumber \\
    &\quad + O\left( m^{0.5-[\varepsilon \min\{\xi-2,\alpha-1\}/4 - 1/(a+1)]} \right)  + O\left(m^{ \frac{1}{2} - [\frac{0.5(\xi-2)(1-c)}{6\xi-8}-\frac{\varepsilon(\xi-1)^2}{2\xi-4}] }\right) \nonumber\\
&\quad +O(m^{u/2}\log m)+O(m^{\max\{2-\min(\xi-1,\alpha),0\}/2}\log^{2} m)\text{ a.s.}\label{finalasip}\end{align}

\subsection*{Choice of parameters}

(\ref{finalasip}) is the ASIP on $\Delta$. To satisfy the conditions, we can choose a small $\varepsilon_1$, then choose $c$ sufficiently close to $1$, and then a sufficiently small $\varepsilon<c\varepsilon_1$. Finally choose a sufficiently large $a$. The possible choices are the following. \[
c=1-\frac{\min\{\xi-2,1\}}{64},\quad 
\varepsilon_1=\frac{\min\{\xi-2,1\}}{32},\quad
\varepsilon=\frac{(\xi-2)\min\{\xi-2,1\}^2}{1024(\xi-1)^2(3\xi-4)},
\]

$$
a = \frac{12288(\xi - 1)^2(3\xi - 4)}{\min\{\xi - 2, \alpha - 1\}(\xi - 2)\min\{\xi - 2, 1\}^2}
$$

$$
u = \frac{12288(\xi - 1)^2(3\xi - 4)}{12288(\xi - 1)^2(3\xi - 4) + \min\{\xi - 2, \alpha - 1\}(\xi - 2)\min\{\xi - 2, 1\}^2}
$$

$$
\varepsilon_{00} = \frac{\min\{\xi - 2, \alpha - 1\}(\xi - 2)\min\{\xi - 2, 1\}^2}{12288(\xi - 1)^2(3\xi - 4) + \min\{\xi - 2, \alpha - 1\}(\xi - 2)\min\{\xi - 2, 1\}^2}
$$

Then  (\ref{finalasip}) becomes 
\begin{align}\label{finalasip2}
   S_m \circ \pi = B_{\sigma^2 m} +O\!\left(
m^{\frac12-
\frac{\min\{\xi - 2, \alpha - 1\}(\xi - 2)\min\{\xi - 2, 1\}^2}{49152(\xi - 1)^2(3\xi - 4) +4 \min\{\xi - 2, \alpha - 1\}(\xi - 2)\min\{\xi - 2, 1\}^2}
}
\right) \text{ a.s. }
\end{align}

Now we can conclude the proof by projecting this ASIP from $\Delta$ to $\mathcal{M}$.

\begin{lemma}\label{finalasip1} One can redefine $(S_n)_{n\ge 1}$ without changing its
distribution on a (richer) probability space  $\Omega$ supporting a Brownian motion $(\mathbb{B}_t)_{t \ge 0}$, such that for any $n>0$, \begin{align*}
   &S_n= \mathbb{B}_{\sigma^2 n}+O\!\left(
n^{\frac12-
\frac{\min\{\xi - 2, \alpha - 1\}(\xi - 2)\min\{\xi - 2, 1\}^2}{49152(\xi - 1)^2(3\xi - 4) +4 \min\{\xi - 2, \alpha - 1\}(\xi - 2)\min\{\xi - 2, 1\}^2}
}
\right) \text{ a.s.}
   \end{align*}
\end{lemma}
\begin{proof}
 We argue as in the proof of Theorem 6.1 of \cite{suetds}: there is a Brownian motion $(\mathbb{B}_t)_{t \ge 0}$ defined on a richer probability space such that $((S_n\circ \pi)_{n \ge 1}, B)=_d((S_n)_{n \ge 1},\mathbb{B})$. Hence (\ref{finalasip2}) implies the required approximations.
\end{proof}

With all steps above, we conclude the proof of Theorem \ref{asip}.

\section{Applications to hyperbolic billiards with CMZ structures}\label{app}
Denote by $Q$ a billiard table, i.e., a closed region in the Euclidean plane with piecewise $C^3$-smooth boundary $\partial Q$. The phase space of a billiard  $\mathcal {M}$ is $\partial Q \times [-\pi/2, \pi/2]$. Define $\pi_{\partial_{Q}}x:=q$ for any $x=(q, \phi) \in \mathcal{M}$. Let $f: \mathcal{M}\to \mathcal{M}$ be a billiard map which maps points of the phase space at reflection times to their images at the next reflection time. It preserves an invariant measure $d\mu_{\mathcal{M}}:=(2\Leb_{\partial Q} \partial Q)^{-1}\cos \phi d\phi dq$. In what follows all billiards under consideration are assumed to have CMZ structures (see section \ref{cmz}, and proofs in \cite{hongkun,markarian, billiardwithvariousrates}). We will call the boundary components with zero curvature flat components, while dispersing components are convex inwards, and focusing components are convex outwards billiard tables. The singularities of $f$ and $f^{-1}$  partition $\mathcal{M}$ into countably many pieces. A locally H\"older function means that it is H\"older continuous on each such piece. Limit theorems for hyperbolic billiard systems are already proven for  regular (H\'older) observables. However, the real experiments, for example the ones with optical cavities \cite{laser}, deal with refractive light rays. Therefore, the observable in this experiment is an indicator function of the set of configurations in which the light rays reach the boundary of the dielectric cavity with small angles smaller than some $\phi_0$, i.e., $\mathbbm{1}_{\{(q,\phi) \in \mathcal{M}: |\phi|< \phi_0 \}}$. Such an indicator function is not continuous, nor is it locally H\"older. If we define $\mathcal{B}:=\{B \subseteq \mathcal{M}: B \text{ is a finite union of polygonal regions}\}$, then the support of this function is a horizontal strip in the phase space,  belongs to $\mathcal{B}$. 

Since all billiard systems considered here have the CMZ structures with $\int R_p^{\xi}d\mu_{\Lambda}< \infty$ and $\xi>1$, Theorem \ref{mixingrate} holds for such indicator observables. We will now verify the conditions of other theorems for hyperbolic billiard systems with such indicator observables.

\subsection{Squashes and other Stadium-type billiards}
A billiard table $Q$ of a squash billiard is a convex domain bounded by two circular arcs and two straight (flat) segments  tangent to these arcs at their common endpoints. A squash billiard is called a stadium if the flat sides are parallel, see Figure \ref{F3}. (Initially being called squash billiards, they were later sometimes called ``skewed" stadia, drive-belt billiards, etc, see Figure \ref{F4}). Note that squashes contain a boundary arc, which is longer than a half of the smaller circle. This class of billiards was introduced in \cite{buni74,bunimovich3}. Let $X\subseteq \mathcal{M}$ be the region where the first collision (in a series of consecutive collisions with one and the same circular arc) with the circular arcs occurs. Denote by $R$ the first return time to $X$ for the billiard map $f$. Then, Theorem \ref{mldp} holds with $\xi=2,\alpha=1$, where $D_{\ell}$ in Assumption \ref{A} is $\{(q,\phi)\in \mathcal{M}: \phi \in [-\ell^{-1},\ell^{-1}] \}$. 

\begin{figure}[!htb]
   \begin{minipage}{0.6\textwidth}
     \includegraphics[width=.6\linewidth]{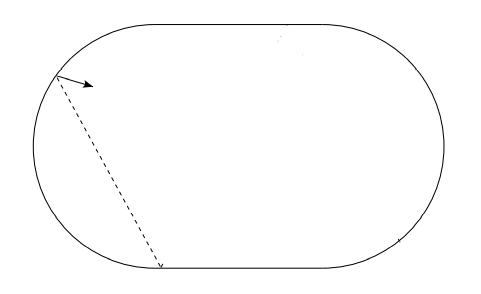}
     \caption{Stadium billiard}
     \label{F3}
   \end{minipage}\hfill
   \begin{minipage}{0.6\textwidth}
     \includegraphics[width=.6\linewidth]{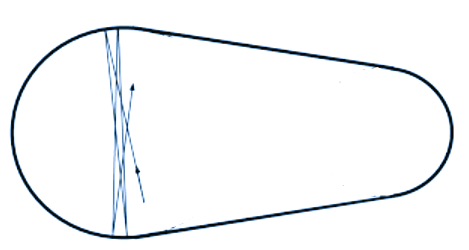}
     \caption{Squash billiard}\label{F4}
   \end{minipage}
\end{figure}

\subsection{Another class of billiards with focusing components} In this section we consider billiard tables $Q$, for which each smooth component $\Gamma_i \subseteq \partial Q$ of the boundary is either dispersing, i.e., convex inwards, or focusing, i.e., convex
outwards. The curvature of every dispersing component is bounded
away from zero and infinity. We assume that every focusing component is an arc of a circle, and that there are no points of the boundary $\partial Q$ on that circle or inside it, other than this arc itself (that is, the SFC-condition is satisfied). We assume also that two dispersing components intersect (if they do) transversely (i.e., there are no cusps) and, besides, each focusing arc is not longer than a half of the corresponding circle. Denote the union of dispersing components by $\partial Q^{+}$, and the union of focusing components by $\partial Q^{-}$. Let $X\subseteq \mathcal{M}$ be \begin{gather*}
    X:=(\partial Q^{+}\times [-\pi/2,\pi/2]) \bigcup \{x \in \mathcal{M}: \pi_{\partial Q}x \in \partial Q^{-}, \pi_{\partial Q}x \text{ and }\pi_{\partial Q}(f^{-1}x)\text{ belong to  different }\Gamma_i \},
\end{gather*}  i.e., only the first collisions in a series of consecutive reflections at the circular arcs, and any collisions with the dispersing components, are included in $X$. Therefore, the case with $R>1$ may occur only in a series of consecutive reflections off a circular arc.

\begin{figure}[!htb]
   \begin{minipage}{0.6\textwidth}
\includegraphics[width=.6\linewidth]{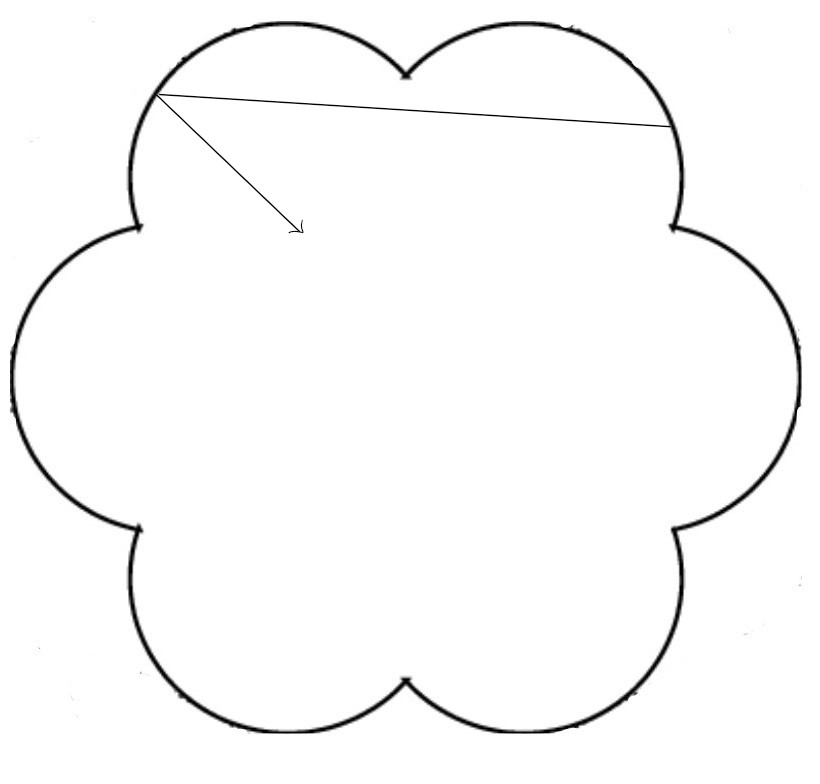}
     \caption{Flower billiard}
     \label{F5}
   \end{minipage}\hfill
   \begin{minipage}{0.48\textwidth}
     \includegraphics[width=.6\linewidth]{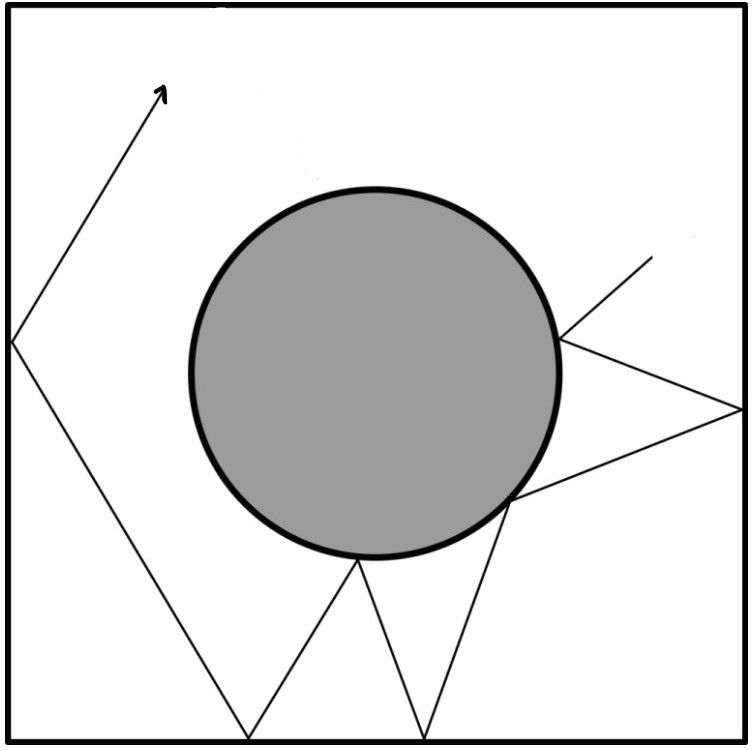}
     \caption{Semi-dispersing billiard}\label{F6}
   \end{minipage}
\end{figure}

This example belongs to the class of  (the so-called) flower billiards, e.g., see Figure \ref{F5}. Denote by  $e_i$ the endpoints of the circular arcs. Then  Theorem \ref{mldp} holds for $\alpha=2,\xi=3$, where $D_{\ell}$ in Assumption \ref{A} is $\{(q,\phi)\in \mathcal{M}:  \phi \in [- \pi/2, - \pi/2+\ell^{-1}] \cup [\pi/2-\ell^{-1}, \pi/2] \}$. Theorem \ref{asip} and Theorem \ref{clt} hold for $\alpha=2$ and any $\xi\in (2.9,3)$.

\subsection{Semi-dispersing billiards} In this subsection we consider billiard tables of the following type. Let $R_0 \subseteq \mathbb{R}^2$ be a rectangle, and the scatterers $B_1, \cdots, B_r \subseteq \interior R_0$ are closed strictly convex subdomains with smooth (at least $C^3)$ or piecewise smooth boundaries, with curvatures bounded away from zero. We also assume that $B_i \bigcap B_j=\emptyset$ for $i\neq j$. The boundary of a billiard table $Q = R_0\setminus \bigcup_i\interior{B_i}$ is partially dispersing (convex inwards), and it is  partially neutral (flat), e.g. see Figure \ref{F6}. The corresponding flat part is $\partial R_0$.

Denote  by $\partial Q^{+}$ the union of dispersing components, and the union of four flat sides is denoted by $\partial R_0$. Let $X:=\{x\in \mathcal{M}: \pi_{\partial Q}x\in \partial Q^+\}$, where $R$ is the first return time to $X$. If $\sup R<\infty$, then this billiard system has exponential decay of correlations, i.e., it is not slowly mixing. So we assume that $\sup R=\infty$. Then $f^R$ is a Sinai billiard map with exponential decay of correlations. 

Thus Theorem \ref{mldp} holds with $\xi=2,\alpha=1$, where $D_{\ell}$ in Assumption \ref{A} is $
\{ (q, \phi) : q \in \partial R_0, |\phi | < \ell^{-1} \}$. 

\subsection{Dispersing billiards with and without finite horizon}
Although this paper addresses slowly mixing billiards, we still include here two examples of fast mixing dispersing billiards, for which our limit theorems hold, and these theorems are new, in contrast to \cite{MN1,Y,pene}. In dispersing billiards the boundary is convex inwards into the billiard table. These billiards are hyperbolic dynamical systems with singularities that arise from orbits tangent to the boundary and to orbits hitting singular points of the boundary. Among the corresponding examples are Sinai billiards on the 2-dimensional torus (Figure \ref{F1}) and diamond billiards without cusps (Figure \ref{F2}). Theorem \ref{clt} and Theorem \ref{asip} hold because $\xi, \alpha$ can assume arbitrarily large values. 

\begin{figure}[!htb]
   \begin{minipage}{0.6\textwidth}
     \includegraphics[width=.7\linewidth]{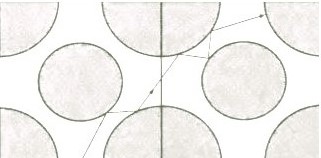}
     \caption{Sinai billiard}
     \label{F1}
   \end{minipage}\hfill
   \begin{minipage}{0.5\textwidth}     \includegraphics[width=.6\linewidth]{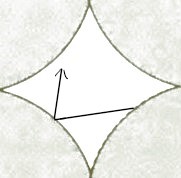}
     \caption{Diamond billiard}\label{F2}
   \end{minipage}
\end{figure}

\section*{Acknowledgements} Y. Su thanks Prof. Carlangelo Liverani for helpful discussions.


%
%

\medskip

\bibliography{bibtext}

\end{document}